\documentclass{amsart}

\RequirePackage{amsmath}
\RequirePackage{amssymb}
\RequirePackage{amsxtra}
\RequirePackage{amsfonts}
\RequirePackage{latexsym}
\RequirePackage{euscript}
\RequirePackage{amscd}
\RequirePackage{amsthm}
\RequirePackage{xypic}
\xyoption{all}

\def\C{\mathbb C}
\def\F{\mathbb F}

\def\Q{\mathbb{Q}}

\def\Z{\mathbb{Z}}
\def\Fbar{\overline{\F}}

\def\scs{\mathrm{scs}}
\def\cs{\mathrm{cs}}

\def\Qbar{\overline{\Q}}

\def\hPhi{\hat{\Phi}}

\def\chibar{\overline{\chi}}

\def\Id{\mathrm{Id}}

\def\O{\mathcal O}
\def\J{\mathfrak J}
\def\CS{\mathcal S}
\def\CK{\mathcal K}

\def\Tbar{\overline{T}}
\def\Ubar{\overline{U}}
\def\Vbar{\overline{V}}
\def\Wbar{\overline{W}}
\def\Xbar{\overline{X}}
\def\Ybar{\overline{Y}}

\def\tV{\widetilde{V}}

\def\tA{\widetilde{A}}
\def\tk{\widetilde{k}}

\def\thetabar{\overline{\theta}}
\def\pibar{\overline{\pi}}

\def\unif{\varpi}

\def\id{\mathrm{id}}

\def\tor{\mathrm{tor}}
\def\tf{\mathrm{tf}}

\def\sm{\mathrm{sm}}

\def\ss{\mathrm{ss}}
\def\Fss{\mathrm{F-ss}}

\def\GL{\mathrm{GL}}

\def\Gal{\mathrm{Gal}}

\def\Aut{\mathrm{Aut}}

\def\Ext{\mathrm{Ext}}

\def\End{\mathrm{End}}

\def\Rep{\mathrm{Rep}}

\def\Hom{\mathop{\mathrm{Hom}}\nolimits}

\def\Spec{\mathop{\mathrm{Spec}}\nolimits}

\def\Ind{\mathop{\mathrm{Ind}}\nolimits}
\def\Res{\mathop{\mathrm{Res}}\nolimits}
\def\cInd{\mathop{\mathrm{c-Ind}}\nolimits}

\def\soc{\mathop{\mathrm{soc}}\nolimits}
\def\cosoc{\mathop{\mathrm{cosoc}}\nolimits}
\def\env{\mathop{\mathrm{env}}\nolimits}

\def\rhobar{\overline{\rho}}

\def\abs{\mid~\;~\mid}
\def\absbar{\overline{\mid~\;~\mid}}

\def\St{\mathrm{St}}
\def\Sp{\mathrm{Sp}}

\def\kernel{\mathop{\mathrm{ker}}\nolimits}

\def\plim#1{\displaystyle \lim_{\longleftarrow \atop #1}}

\def\iso{\buildrel \sim \over \longrightarrow}

\def\Nbar{\overline{{N}}}

\def\LL{\pi}
\def\LLbar{\overline{\pi}}
\def\LLcheck{\widetilde{\pi}}
\def\LLbarcheck{\widetilde{\overline{\pi}}}

\swapnumbers

\newtheorem{theorem}[subsubsection]{Theorem}
\newtheorem{lemma}[subsubsection]{Lemma}
\newtheorem{df}[subsubsection]{Definition}
\newtheorem{cor}[subsubsection]{Corollary}
\newtheorem{conj}[subsubsection]{Conjecture}

\newtheorem{prop}[subsubsection]{Proposition}
\newtheorem{remark}[subsubsection]{Remark}
\newtheorem{example}[subsubsection]{Example}
\newtheorem{condition}[subsubsection]{Condition}

\numberwithin{equation}{section}

\title[Local Langlands in families]
{The local Langlands correspondence
for $\GL_n$ in families}
\author{Matthew Emerton and David Helm}
\thanks{The first author was supported in part by NSF grants DMS-0401545,
DMS-0701315, and DMS-1002339}
\address[Matthew Emerton]{Mathematics Department, Northwestern
University, 2033 Sheridan Rd., Evanston, IL 60208}
\address[David Helm]{Mathematics Department, University of Texas at Austin,
1 University Station C1200, Austin, TX, 78712}
\email[Matthew Emerton]{emerton@math.northwestern.edu}
\email[David Helm]{dhelm@math.utexas.edu}

\begin{document}

Draft: April 1, 2011

\maketitle

\setcounter{tocdepth}{1}
\tableofcontents

\section{Introduction}
The goal of this paper 
is to extend, for any non-archimedean local field $E$ of residue
characteristic $\ell$, the local Langlands correspondence
between $n$-dimensional representations of the local Galois group
$G_E$ and admissible smooth representations of
$\GL_n(E)$ to a correspondence defined on
$p$-adic families of $G_E$-representations (for primes $p$ distinct from $\ell$). 

\subsection{The local Langlands correspondence for $\GL_n$ in $p$-adic families}
Let $p$ and $\ell$ be distinct primes,
and let $E$ be a local field of residue characteristic $\ell$.
If $A$ is a complete local domain of characteristic zero and residue
characteristic $p$,
with field of fractions $\mathcal K$,
then the classical local Langlands correspondence establishes a map
$\rho \mapsto \LL(\rho)$
from the set of isomorphism classes of continuous representations
$\rho: G_E \rightarrow \GL_n(\mathcal K)$
to the set of admissible smooth representations of $\GL_n(E)$
over~$\mathcal K$.\footnote{The local Langlands correspondence is
usually stated in the situation when $\mathcal K$ is a finite
extension of $\Q_p$.  However, it is straightforward to extend
it to the more general context considered above.  Also, let us note
that we work with a suitably normalized
``generic'' version of the local Langlands correspondence, which
to a non-generic Weil--Deligne representation attaches a generic
but reducible principal series representation.  See
Section~\ref{sec:LL zero}
for details on both points.}

In Section~\ref{sec:LL families} we describe the extension of
the local Langlands correspondence to $p$-adic families.
Before stating our result, we introduce further notation:
given $\rho: G_E \rightarrow \GL_n(\mathcal K)$
as above,
we write $\LLcheck(\rho)$ to denote
the smooth contragredient of~$\LL(\rho)$.

\begin{theorem}
\label{thm:intro families}
Let $A$ be a reduced complete Noetherian local ring with maximal
ideal $\mathfrak m$ and finite residue field $k$ of characteristic $p$.
Suppose furthermore that each minimal prime of $A$ has residue characteristic
$0$ {\em (}or equivalently,
that $A$ is $p$-torsion free{\em )}.  If
$\rho: G_E \rightarrow \GL_n(A)$ is continuous 
{\em (}when the target is
given its $\mathfrak m$-adic topology{\em )}, then there exists at
most one admissible smooth $\GL_n(E)$-representation $V$ over $A$,
up to isomorphism,
satisfying the following conditions:
\begin{enumerate}
\item $V$ is $A$-torsion free.
\item If $\mathfrak a$ is a minimal prime of $A$, with residue field
$\kappa(\mathfrak a)$,
then there is a $\kappa(\mathfrak a)$-linear $\GL_n(E)$-equivariant
isomorphism
$$ \LLcheck\bigl(\kappa(\mathfrak a)\otimes_A \rho\bigr) \iso
\kappa(\mathfrak a)\otimes_A V.$$ 
\item If we write
$\Vbar := k\otimes_A V$,
then the $\GL_n(E)$ cosocle $\cosoc(\Vbar)$ of $\Vbar$ is absolutely
irreducible and generic,
while the kernel of the surjection $\Vbar \rightarrow \cosoc(\Vbar)$
contains no generic Jordan--H\"older factors.

\smallskip

\noindent
Furthermore, if such a $V$ exists,
then:

\smallskip

\item
There exists an open dense subset $U$ of $\Spec A[\dfrac{1}{p}]$
such that for each prime $\mathfrak p$ in $U$, there is
a $\GL_n(E)$-equivariant, nonzero surjection
$$\LLcheck\bigl(\kappa(\mathfrak p)\otimes_A \rho\bigr) \rightarrow 
\kappa(\mathfrak p)\otimes_A V,$$
where $\kappa(\mathfrak p)$ is the residue field of $\mathfrak p$.
\end{enumerate}
\end{theorem}

If a representation $V$ satisfying the conditions of this theorem
with respect to a given Galois representation $\rho:G_E \rightarrow \GL_n(A)$
exists, then we write $V := \LLcheck(\rho)$.  (Note that the
theorem ensures that $V$ is unique up to isomorphism, so that $\LLcheck(\rho)$
is then uniquely determined by $\rho$, up to isomorphism, if it exists.)
Part~(4) of the theorem describes the 
precise sense in which $V$ interpolates
the local Langlands correspondences attached to the Galois representations
$\kappa(\mathfrak p)\otimes_A \, \rho$ as $\mathfrak p$ ranges over the 
points of $\Spec A[\dfrac{1}{p}]$.  Conjecturally, we can take
$U$ equal to all of $\Spec A[\dfrac{1}{p}]$ in this statement, although
our results fall short of establishing this.  (We refer the reader
to Theorems~\ref{thm:interpolation1} and~\ref{thm:interpolation2} for
the precise results.)  On the other hand,
we will give examples in Section~\ref{sec:LL families} showing
that it is not possible in general
to strengthen ``surjection'' to ``isomorphism'' in this statement.

\begin{remark}
\label{rem:duals}
{\em Our convention for the generic local Langlands correspondence
is that the $GL_n(\mathcal K)$-representation $\pi(\rho)$ attached to a
continuous Galois representation $\rho:G_E \rightarrow \GL_n(\mathcal K)$
should have generic socle.  It is this convention that seems to 
fit best with global applications of the type considered in~\cite{Em8}
and~\cite{Em9},
for example.
On the other hand, when working with families,
it turns out to be easier to interpolate representations whose cosocle
is generic.  This explains the appearance of the various contragredient
representations in Theorem~\ref{thm:intro families}, and in our notation
for the representations that it describes.
}
\end{remark}

Just as in the traditional setting, it seems to be easier to characterize the
local Langlands correspondence in families than to prove its existence. 
However, we make the following conjecture.

\begin{conj}
\label{conj:LL}
If $A$ is a reduced $p$-torsion free complete Noetherian local ring with maximal
ideal $\mathfrak m$ and finite residue field $k$ of characteristic $p$,
and if
$\rho: G_E \rightarrow \GL_n(A)$ is continuous,
then $\LLcheck(\rho)$ exists.
\end{conj}

One instance in which we can verify the conjecture is the case when
$A$ is the ring of integers in a finite extension of $\Q_p$. 
(In this case, it is a consequence of Proposition~\ref{prop:lattice}.)
Perhaps more significantly, in the case when $n = 2$ and $p$ is odd,
the conjecture has been established by the second author \cite{He}.

\begin{remark}
{\em
In the theory of $p$-adic automorphic forms one naturally encounters admissible
smooth representations of $p$-adic groups defined over complete local rings $A$
of the type considered in the theorem (the ring $A$ typically being
taken to be a certain $p$-adically completed Hecke algebra), and a primary
motivation for our extension of the local Langlands correspondence to this setting
is to introduce a language and techniques which are adequate for describing and investigating
the representations that arise in such global contexts.  For example, the results of this paper
in the case of the group $\GL_2$ play a crucial role in both the statement and
proof of the main result of \cite{Em8} regarding local-global compatibility in
the completed cohomology of modular curves, and also in the forthcoming work
\cite{Em9}.
}
\end{remark}

\medskip

\subsection{A mod $p$ local Langlands correspondence for $\GL_n$} 
In Section~\ref{sec:LL p} we define a mod $p$ local Langlands correspondence,
whose key properties are summarized in the following theorem.

\begin{theorem}
\label{thm:intro mod p}
There is a map
$\rhobar \mapsto \LLbar(\rhobar)$
from the set of isomorphism classes
of continuous representations $G_E \rightarrow \GL_n(k)$
{\em (}where $k$ is a finite field of characteristic~$p${\em )}
to the set of isomorphism classes of finite length admissible smooth
$\GL_n(E)$-representations
on $k$-vector spaces, uniquely determined by the following
conditions:

\begin{enumerate}
\item For any $\rhobar$, the $G$-socle
$\soc\bigl(\LLbar(\rhobar)\bigr)$
of the associated $\GL_n(E)$-representation
$\LLbar(\rhobar)$ is absolutely irreducible and generic,
and the quotient $\LLbar(\rhobar)/\soc\bigl(\LLbar(\rhobar)\bigr)$ 
contains no generic Jordan--H\"older factors.

\item 
Given $\rhobar: G_E \rightarrow \GL_n(k)$,
together with a deformation
$\rho: G_E \rightarrow \GL_n(\O)$ of~$\rhobar,$
where $\O$ is a characteristic zero discrete valuation ring with 
uniformizer $\unif$ and residue field $k'$ containing $k$,
there is $\GL_n(E)$-equivariant surjection
$\LLbarcheck(\rhobar) \otimes_k k' \rightarrow \LLcheck(\rho)/\unif \LLcheck(\rho).$
{\em (}Note that $\LLcheck(\rho)$ must exist as $\O$ is a finite extension of $\Q_p$.{\em )}

\item  The representation $\LLbar(\rhobar)$ is minimal with respect
to satisfying conditions~{\em (1)} and {\em (2)}, i.e.\ 
given any continuous representation
$\rhobar: G_E \rightarrow \GL_n(k)$ 
and any representation $\pibar$ of $\GL_n(E)$ satisfying these two
conditions with respect to~$\rhobar$, there is a $\GL_n(E)$-equivariant embedding
$\LL(\rhobar) \hookrightarrow \pibar$.
\end{enumerate}
\end{theorem}

Recall that Vigneras has already defined a mod $p$ local Langlands
correspondence for $\GL_n$ in \cite{Vig4}.
The key differences between the correspondence of our theorem
and correspondence of \cite{Vig4} are:
\begin{enumerate}
\item[(a)] The input is a Galois representation (not a Weil--Deligne
representation).
\item[(b)] The output is an admissible smooth $\GL_n(E)$-representation
that is possibly reducible, but always generic.
\item[(c)] The correspondence is compatible with reduction modulo $p$
in the direct sense given by parts~(2) and~(3). 
(The Zelevinski involution does not intervene.)
\end{enumerate}
This being said, we rely on the results of \cite{Vig4} 
for the construction of our correspondence.
The key point, whose proof relies on \cite{Vig4}, is that
for any deformation $\rho$ of~$\rhobar$, the representation
$\LLcheck(\rho)$ reduces modulo $\mathfrak m$
to a representation whose cosocle is
absolutely irreducible and generic, and is independent,
up to isomorphism, of the choice of $\rho$. 
(See the discussion following Corollary~\ref{cor:independent generic constituent} below.)

\begin{remark}
\label{rem:lifts}
{\em  One can consider the following stronger form
of condition~(2) of Theorem~\ref{thm:intro mod p}:}

\begin{enumerate}
\item[($2'$)] 
Given $\rhobar: G_E \rightarrow \GL_n(k)$,
together with a deformation~$\rho: G_E \rightarrow \GL_n(A)$ of~$\rhobar,$
where $A$ is a reduced complete Noetherian local $W(k)$-algebra,
flat over~$W(k)$,
with maximal ideal $\mathfrak m$ and residue field $k$,
and $\LLcheck(\rho)$ exists,
there is a $\GL_n(E)$-equivariant surjection
$\LLbarcheck(\rhobar) \rightarrow \LLcheck(\rho)/\unif \LLcheck(\rho).$
\end{enumerate}

{\em
In some circumstances, we are able to verify that
$\pibar(\rhobar)$ is in fact minimal with respect to conditions~(1) and ($2'$).
(In other words, $\pibar(\rhobar)$ contains as a submodule the dual of
any module that arises by specializing a family $\LLcheck(\rho)$ attached to 
a deformation of $\rhobar$.)  This is essentially a characteristic $p$ analog
of Theorems~\ref{thm:interpolation1} and~\ref{thm:interpolation2}.
We conjecture that this stronger minimality property holds in general.
}
\end{remark}

\begin{remark}
\label{rem:non-surj}
{\em In general, $\LLbar(\rhobar)$ is not irreducible,
and if this is the case, then it is not possible to
strengthen ``surjection'' to ``isomorphism'' in the 
statement of part~(2) of Theorem~\ref{thm:intro mod p}.}
\end{remark}

\subsection{The organization of the paper}

We begin by establishing some basic facts
about admissible smooth representations
of certain topological groups over a Noetherian local ring $A$;
we apply this machinery in section~\ref{subsec:lattices} to the study
of invariant lattices in representations of topological groups
over the field of fractions of a complete discrete valuation ring.
The key result we establish is Lemma~\ref{lem:lattice}, which in certain
circumstances allows us to construct an invariant lattice in such
a representation whose reduction has a prescribed socle.

Section~\ref{sec:gln} establishes results specific to the representation
theory of $\GL_n(E)$ over certain local rings $A$.  In~\ref{subsec:kirillov}
we construct a theory of Kirillov models over a large class of base rings.
Our approach essentially follows that of~\cite{BZ}, but as we are not working
over algebraically closed fields issues of descent arise.  In spite of this
one recovers almost all of the theory of the Kirillov functors developed
in~\cite[\S 4]{BZ}.  Particularly useful for us is the notion of a
``generic'' irreducible representation over an arbitrary field.

Section~\ref{subsec:AIG} introduces an essential concept for our
results: that of an essentially AIG representation over a field.  These
are representations whose socles are absolutely irreducible and generic,
and that satisfy a certain finiteness property.  The importance of these
representations stems from the fact that the ``modified Langlands correspondences''
we consider
send Galois representations to essentially AIG representations.  In
section~\ref{subsec:AIG lattice} we apply the results of section~\ref{subsec:lattices}
to establish some basic facts about the reduction theory of essentially AIG
representations.

In section~\ref{sec:LL zero} we study the behavior of the local Langlands
correspondence (over fields of characteristic zero) under specialization.
As we have previously discussed, the usual local Langlands correspondence
is not suitable for our purposes, and we instead consider a modification
of this correspondence due to Breuil and Schneider.  Our first main result
(Corollary~\ref{cor:essential}) establishes that the admissible representations
of $\GL_n(E)$ produced by the Breuil-Schneider correspondence are essentially AIG.
Once we have this, we apply the reduction theory of section~\ref{subsec:AIG lattice},
together with ideas from the Zelevinski classification, to establish
Theorem~\ref{thm:specialization}, which relates the behaviour of a Galois representation
under specialization to a characteristic zero residue field to the behaviour
(under the same specialization) of the corresponding admissible representation
constructed by Breuil-Schneider.

Section~\ref{sec:LL p} constructs a ``modified local Langlands correspondence''
in characteristic $p$, by analogy with the Breuil-Schneider correspondence;
in particular we define this correspondence to be the ``minimal'' correspondence
that satisfies a mod $p$ analogue of Theorem~\ref{thm:specialization}.  We
refer the reader to Theorem~\ref{thm:mod p LL} for the precise definition.

We finally turn to the study of the local Langlands correspondence
for families of admissible representations in section~\ref{sec:LL families}.
Section~\ref{subsec:results} discusses the main results of our theory; to avoid
obscuring this discussion with technicalities we postpone the proofs
to section~\ref{subsec:proofs}.  Surprisingly little beyond the theory of
Kirillov models is necessary to prove the basic uniqueness result
of Theorem~\ref{thm:family}.  On the other hand, establishing more precise results
about the structure of the family of admissible representations attached
to a given family of Galois representations (for instance, the interpolation theorems
\ref{thm:interpolation1} and~\ref{thm:interpolation2}) requires the full
strength of the specialization results in section~\ref{sec:LL zero}.

\subsection{Acknowledgments}
Matthew Emerton 
would like to thank Kevin Buzzard, Frank Calegari, Gaetan Chenevier,
Mark Kisin, and Eric Urban for helpful conversations on the subject of this
note.   He first began to investigate some of the questions considered
in this paper at the
conference ``Open questions and recent developments in Iwasawa theory'',
held at Boston University in June 2005, in honour of Ralph Greenberg's
60th birthday,  and was given the opportunity to further explore them 
in a series of lectures given during
the special semester on eigenvarieties at Harvard in April 2006.
He would like to take the opportunity to thank both institutions,
as well as Robert Pollack and Barry Mazur,
the principal organizers of the two events,
for providing such stimulating mathematical environments.

David Helm would like to thank Kevin Buzzard for his ideas
and perspective on the questions addressed by this note, and Richard
Taylor for his continued interest, advice, and encouragement.

Both authors benefited from a mutual exchange of ideas occasioned by
the conference ``Modular forms and arithmetic'', held at MSRI in
June/July 2008, and sponsored by MSRI and the Clay Math Institute.
They would like to thank these institutions, as well as the organizers
of the conference, for creating this fruitful opportunity.

\section{Representation theory --- general background}
\subsection{Admissible smooth representations} \label{subsec:admissible}
Let $A$ be a Noetherian local ring with maximal ideal
$\mathfrak m$ and residue field $k$.
In this subsection we recall some basic facts about
admissible smooth representations over~$A$.

\begin{df}
{\em
An $A$-linear representation of a topological group
$H$ on an $A$-module $V$ 
is called smooth if any element of $V$ is fixed by an
open subgroup of~$H$.
}
\end{df}

Clearly any $H$-invariant sub- or quotient $A$-module
of a smooth $H$-representation over $A$ is again
a smooth $H$-representation over $A$.

\begin{df}
{\em
A smooth representation of a topological group $H$
on an $A$-module $V$
is called admissible
if for any open subgroup $H_0\subset H$, the $A$-module
of fixed points $V^{H_0}$ is finitely generated.
}
\end{df}

Clearly any $H$-invariant $A$-submodule of an
admissible smooth $H$-representation over $A$
is again an admissible smooth $H$-representation
over $A$. (For the case of $H$-invariant quotients,
see Lemma~\ref{lem:quotient} below.) 

\smallskip

Consider the following condition on $H$:

\begin{condition}
\label{cond:pro-ell}
{\em $H$ contains a profinite open subgroup,
admitting a countable basis of neighbourhoods of the identity,
whose pro-order is invertible in $A$.
}
\end{condition} 

Suppose that $H$ satisfies Condition~\ref{cond:pro-ell},
and let $\{H_i\}_{i \geq 0}$ denote a decreasing sequence
of open subgroups of $H$, each of whose pro-order is
invertible in $A$,
and which forms a neighbourhood basis 
of the identity in $H$.
If $V$ is a smooth $H$-representation over $A$,
then for each $n \geq 1,$ we may define the idempotent projector
$\pi_i: V \rightarrow V^{H_i}$
via $v \mapsto \int_{H_i} h v d\mu_i,$ where $\mu_i$ denotes
Haar measure on $H_i$, normalized so that $H_i$ has total
measure $1$.  If we define $V_i := \kernel \pi_i \bigcap V^{H_{i+1}},$
then the inclusions $V_i \subset V$ induces an isomorphism
of $A$-modules
\begin{equation}
\label{eqn:decomp}
\bigoplus_i V_i \iso V.
\end{equation}
The formation of $V_i$ is evidently functorial on the category
of smooth representations of $H$ over $A$,
and thus so is the direct sum
decomposition~(\ref{eqn:decomp}).
In fact one can say something more precise:

\begin{lemma}
\label{lem:isom}
Suppose that $H$ satisfies Condition~{\em \ref{cond:pro-ell}}.
If $W$ is an $H$-invariant $A$-submodule of the smooth $H$-representation
$A$ over $V$,
then the natural maps  $W_i \hookrightarrow V_i \bigcap W$
and $V_i/(V_i \bigcap W) \rightarrow (V/W)_i$ are isomorphisms.
\end{lemma}
\begin{proof}
This is evident.
\end{proof}

\begin{lemma}
\label{lem:adm}
Suppose that $H$ satisfies Condition~{\em \ref{cond:pro-ell}}.
A smooth $H$-represen\-tation 
$V$ over $A$ is admissible if and
only if each of the $A$-modules $V_i$
is finitely generated.
\end{lemma}
\begin{proof}
This follows from the isomorphisms
$\bigoplus_{j \leq i} V_j \iso V^{H_i}$
for each $i$, and the fact
that the sequence $\{H_i\}$ is cofinal in the
collection of all of open subgroups of~$H$.
\end{proof}

\begin{lemma}
\label{lem:quotient}
Suppose that $H$ satisfies Condition~{\em \ref{cond:pro-ell}}.
If $V$ is an admissible smooth $H$-representation 
over $A$,
and if $W$ is a $G$-invariant $A$-submodule
of $V$,
then $V/W$ is again an admissible smooth $H$-representation
over $A$.
\end{lemma}
\begin{proof}
This follows from the preceding lemma, and the fact
that $V_i \rightarrow (V/W)_i$ is surjective.
\end{proof}

If $V$ is an $A$-module equipped with
an admissible smooth $H$-representation,
then typically $V$ itself will not be finitely generated
as an $A$-module.  Nevertheless, the existence of the
decomposition~(\ref{eqn:decomp}) allows us to extend 
many results about finitely generated $A$-modules
to the situation of admissible smooth $G$-representations.

\begin{lemma}
\label{lem:nakayama}
If $H$ satisfies Condition~{\em \ref{cond:pro-ell}},
and if $V$ is an admissible smooth $H$-representation
for which $V/\mathfrak m V = 0,$ then $V = 0$.
\end{lemma}
\begin{proof}
The decomposition~(\ref{eqn:decomp})
yields the isomorphism
$\bigoplus_i V_i/\mathfrak m V_i \iso V/\mathfrak m.$
Thus $V/\mathfrak m V = 0$ implies that $V_i/\mathfrak m V_i = 0$
for each value of $i$.   Since each $V_i$ is finitely generated
over $A$, this in turn implies that $V_i =0$ for each $i$,
by Nakayama's lemma.  Thus $V=0$, as claimed.
%
\end{proof}

\begin{lemma}
\label{lem:fg}
If $H$ satisfies Condition~{\em \ref{cond:pro-ell}},
and if $V$ is an admissible smooth representation
such that $V/\mathfrak m V$ is finitely generated
over $k[H]$, then $V$ is finitely generated over
$A[H]$.
\end{lemma}
\begin{proof}
Let $S \subset A$ be a finite subset whose
image in $V/\mathfrak m V$ generates this quotient
over $k[H]$,
and let $W$ be the $A[H]$-submodule of $V$ generated by
$S$.  Lemma~\ref{lem:quotient} implies that $(V/W)$ is admissible,
and by construction we see that $(V/W) / \mathfrak m (V/W) = 0.$
Thus Lemma~\ref{lem:nakayama} shows that $W = V$, and so
$V$ is also finitely generated.
\end{proof}


It will be technically useful to consider a related notion
of admissible representation.

\begin{df}
\label{df:adm cont}
{\em 
If $V$ is an $A$-module
equipped with an $A$-linear representation of $H$, we say that 
$V$ is an admissible continuous $H$-representation if:
\begin{enumerate}
\item $V$ is $\mathfrak m$-adically complete and separated.
\item The $H$-action on $V$ is continuous, when $V$ is 
equipped with its $\mathfrak m$-adic topology (i.e.\
the action map $H\times V \rightarrow V$ is jointly continuous).
\item The induced $H$-representation on $V/\mathfrak m V$
(which is automatically smooth, by~(1)) is admissible smooth.
\end{enumerate}
}
\end{df}

\begin{lemma}
Suppose that $H$ satisfies Condition~{\em \ref{cond:pro-ell}}.
If $V$ is a continuous admissible $H$-representation over $A$,
then for each $n > 0$,
the induced $H$-representation on $V/\mathfrak m^n V$ is 
admissible smooth.
\end{lemma}
\begin{proof}
Condition~(2) of Definition~\ref{df:adm cont} implies that
the $H$-action on $V/\mathfrak m^n V$ is continuous, when
the latter is equipped with its discrete topology.  In other
words, $V/\mathfrak m^n V$ is a smooth representation of $H$.
Since the formation of $H_i$-invariants
is exact, for any $i \geq 0$, we find that
$(V/\mathfrak m^n V)^{H_i}/\mathfrak m (V/\mathfrak m^n V)^{H_i}
\iso (V/\mathfrak m V)^{H_i}$ is finite dimensional over $A/\mathfrak m$,
by Condition~(3) of Definition~\ref{df:adm cont}.
Lemma~\ref{lem:complete Nak} below (applied to the module
$(V/\mathfrak m^n V)^{H_i}$ of the Artinian local ring $A/\mathfrak m^n$)
then shows that $(V/\mathfrak m^n V)^{H_i}$ is finitely generated
over $A/\mathfrak m^n$.  Consequently $V/\mathfrak m^n V$ is admissible.
\end{proof}

\begin{df}
{\em If $V$ is an $A$-module, we let $\widehat{V}$ denote
the $\mathfrak m$-adic completion of~$V$.
}
\end{df}

\begin{df}
{\em If $V$ is an $A$-module equipped with an $H$-representation,
we let $V_{\sm}$ denote the subset of $V$ consisting of
vectors which are smooth, i.e.\ which are fixed by some open
subgroup of $H$.  One immediately verifies that $V_{\sm}$ is
an $A$-submodule of $V$, closed under the action of $H$.
Thus $V_{\sm}$ is a smooth $H$-representation.
}
\end{df}

\begin{prop}
Suppose that $H$ satisfies Condition~{\em \ref{cond:pro-ell}},
and let $V$ be an admissible smooth $H$-representation over~$A$.
\begin{enumerate}
\item
 $\widehat{V}$ is a continuous admissible $H$-representation
over~$A$.
\item
If $A$ is $\mathfrak m$-adically complete, then the natural map
$V \rightarrow \widehat{V}_{\sm}$ is an isomorphism.
\end{enumerate}
\end{prop}
\begin{proof}
The $H$-action on $V$ is smooth, and thus so is the $H$-action
on $V/\mathfrak m^n V$, for each $n \geq 0.$  Passing to the
projective limit over $n$, we find that the $H$-action
on $\widehat{V}$ is $\mathfrak m$-adically continuous.
Since $\widehat{V}/\mathfrak m \widehat{V} = V/\mathfrak m V,$
it follows from Lemma~\ref{lem:quotient} that the $H$-action
on $\widehat{V}/\mathfrak m\widehat{V}$ is admissible.
Thus $\widehat{V}$ satisfies both the conditions
of Definition~\ref{df:adm cont}.  This proves~(1).

We now turn to proving~(2), and so in particular,
assume that $A$ is $\mathfrak m$-adically complete.
For each $i \geq 0,$ we find that
$$\widehat{V}^{H_i}
\iso \plim{n} (V/\mathfrak m^n V)^{H_i} \iso \plim{n} V^{H_i}/\mathfrak m^n V^{H_i}
\iso V^{H_i}$$
(the second isomorphism following from the exactness of the formation of
$H_i$-invariants, and the third following from the fact that
$V^{H_i}$ is finitely generated over $A$, by assumption, and hence
$\mathfrak m$-adically complete, since $A$ is $\mathfrak m$-adically
complete).
Consequently, the map $V^{H_i} \rightarrow \widehat{V}^{H_i}$ is an
isomorphism for each $i \geq 0$, and thus, passing to the inductive limit
over $i$, we find that 
$V \iso \widehat{V}_{\sm},$ as claimed.
\end{proof}

\begin{prop}
Suppose that $H$ satisfies Condition~{\em \ref{cond:pro-ell}}
and that $A$ is $\mathfrak m$-adically complete,
and let $V$ be an admissible continuous $H$-representation
over~$A$. 
\begin{enumerate}
\item
$V_{\sm}$ is an admissible smooth $H$-representation.
\item
The natural map $\widehat{V_{\sm}} \rightarrow V$ 
is an isomorphism.
\end{enumerate}
\end{prop}
\begin{proof}
Since $V$ is $\mathfrak m$-adically complete and separated, we see that
$V^{H_i}$ is $\mathfrak m$-adically complete and separated
for each $i \geq 0$.    Since the formation of $H_i$-invariants
is exact, we see that $V^{H_i}/\mathfrak m V^{H_i} \iso
(V/\mathfrak m V)^{H_i}$, which by assumption is finite dimensional
over $A/\mathfrak m$.
Lemma~\ref{lem:complete Nak} below then implies that $V^{H_i}$ is finitely 
generated over $A$.
Since $(V_{\sm})^{H_i} = V^{H_i}$ by the very definition of $V_{\sm}$,
we see that $V_{\sm}$ is admissible, proving~(1).

If $i \geq 0$ and $n > 0$, then
$$(V_{\sm})^{H_i}/\mathfrak m^n (V_{sm})^{H_i}
= V^{H_i}/\mathfrak m^n V^{H_i}
\iso (V/\mathfrak m^n V)^{H_i},$$
the equality holding (as was already noted above) by the very definition of $V_{\sm}$,
and the isomorphism following from the exactness of the formation
of $H_i$-invariants.  Passing to the inductive limit over $i$,
and taking into account the fact that $V_{\sm}$ and $V/\mathfrak m^n V$
are both smooth $H$-representations,
we find that
$V_{\sm}/\mathfrak m^n V_{\sm} \iso V/\mathfrak m^n V.$
Passing to the projective limit over $n$,
we find that $\widehat{V_{\sm}} \iso V,$
proving~(2).
\end{proof}

\begin{remark}
\label{rem:adm equiv}
{\em
It follows from the preceding propositions that if $A$ is $\mathfrak m$-adically complete,
then the functors
$V \mapsto \widehat{V}$ and $V \mapsto V_{\sm}$ are mutually quasi-inverse,
and induce an equivalence of categories between the category of admissible
smooth $H$-representations over $A$, and the category of admissible 
continuous $H$-representations over $A$.
}
\end{remark}

We close this subsection by recalling a version of Nakayama's lemma
in the setting of $\mathfrak m$-adically
separated modules over complete local rings.

\begin{lemma}
\label{lem:complete Nak}
Suppose that $A$ is $\mathfrak m$-adically complete.
If $M$ is an $\mathfrak m$-adically separated $A$-module
such that $M/\mathfrak m M$ is finite dimensional over $A/\mathfrak m,$
then $M$ is finitely generated over $A$.
\end{lemma}
\begin{proof}
Choose $S = \{s_1,\ldots,s_m\}\subset M$ to be finite, and such that the image of $S$
in $M/\mathfrak m M$ spans $M$ over $A/\mathfrak m$, and let 
$N$ denote the $A$-submodule of $M$ generated by $S$.
We then have that $M = N + \mathfrak m M,$
and so arguing inductively, for each $v \in M$ we may find, for each $i = 1,\ldots,s,$
a sequence of elements $a_{i,n}$ of $A$ with $a_{i,n} \in \mathfrak m^n$
for each $n$, such that for each $n$ we have
$$v \in \sum_{i = 1}^m (a_{i,0} + a_{i,1} + \cdots + a_{i,n}) s_i +
\mathfrak m^{n+1} M.$$
Writing $a_i = a_{i,0} + a_{i,1} + \cdots$ (a well-defined element of $A$,
since $A$ is $\mathfrak m$-adically complete by assumption), 
we then find (since $M$ is $\mathfrak m$-adically separated)
that $v = \sum_{i = 1}^s a_i s_i \in N,$ and thus that $M = N$,
proving the lemma.
\end{proof}

\subsection{Invariant lattices} \label{subsec:lattices}
Let $\O$ be a complete discrete valuation ring, with field
of fractions $\CK$ and residue field $K$ of characteristic different
from $\ell$.
Let $\unif$ be a choice of uniformizer of $\O$.
If $V$ is a $\CK$-vector space, then by a lattice in $V$ we mean
an $\O$-submodule $V^{\circ}$
which spans $V$ over $\CK$.

\begin{df}
\label{def:integral}
{\em
We say that a representation $V$ of
a group $H$
over $\CK$ is a good integral representation if $V$ contains a 
$\unif$-adically separated $H$-invariant lattice $V^{\circ}$
with the property that $\Vbar^{\circ} := V^{\circ}/\unif V^{\circ}$ has finite
length as a $K[H]$-module.
}
\end{df}

We now prove some basic results pertaining to this definition.

\begin{lemma}
\label{lem:sub}
Any subrepresentation of a good integral representation $V$ of $H$
over $\CK$ is again a good integral representation of $H$.
\end{lemma}
\begin{proof}
If $V^{\circ}$ is a $\unif$-adically separated $H$-invariant lattice
in $V$ for which $\Vbar^{\circ}$ is of finite length over $K[H]$,
then $W^{\circ} := V^{\circ} \cap W$ is a $\unif$-adically separated
$H$-invariant lattice in $W$, and since the natural map
$\Wbar^{\circ} \to \Vbar^{\circ}$ is injective, we see that
$\Wbar^{\circ}$ also has finite length over $K[H]$.
\end{proof}

\begin{lemma}
\label{lem:div}
Let $V$ be a good integral representation 
of a group $H$, and fix a $\unif$-adically separated
$H$-invariant lattice
$V^{\circ} \subset V$ such that $\Vbar^{\circ}$
has finite length over~$K[H]$.
If $M$ is an $H$-invariant $\O$-submodule of $V/V^{\circ}$, 
then either $M$ is of bounded exponent as an $\O$-module,
or else $M$ contains a non-zero $H$-invariant divisible $\O$-submodule.
\end{lemma}
\begin{proof}
If $M$ is not of bounded exponent,
then the map $M[\unif^n] \rightarrow M[\unif^m]$
induced by multiplication by $\unif^{n-m}$ has non-zero image
for each $n \geq m \geq 1$, and hence, since each $M[\unif^n]$ has finite
length, we find that $\plim{n} M[\unif^n] \neq 0$ (the transition maps
being given by multiplication by $\unif$).  
If $(m_n)_{n \geq 0}$ is a non-zero element of this projective limit,
then the $\O$-submodule of $M$ generated by the elements $m_n$
is evidently non-zero and divisible.  Thus the maximal divisible submodule
of $M$ is non-zero; it is also clearly $H$-invariant.
\end{proof}

\begin{lemma} 
\label{lem:commens}
If $H$ is a topological group satisfying Condition~{\em \ref{cond:pro-ell}},
and if $V$ is a good integral admissible smooth representation of
$H$ over $\CK$, then:
\begin{enumerate}
\item Any two 
$\unif$-adically separated $H$-invariant lattices in $V$
are commensurable.
\item
If $V^{\circ}$ is a
$\unif$-adically separated $H$-invariant lattice in $V$, then
$V^{\circ}$ is finitely generated over $\O[H]$,
the $K[H]$-module
$\Vbar^{\circ}:= V^{\circ}/\unif V^{\circ}$ is of finite length,
and the isomorphism class of
$(\Vbar^{\circ})^{\ss}$ {\em (}the semisimplification of $\Vbar^{\circ}$ as
a $k[H]$-module{\em )} is independent of the choice of $V^{\circ}$.
\end{enumerate}
\end{lemma}
\begin{proof}
Since $V$ is good integral by assumption, we may and do choose
an  $\unif$-adically separated $H$-invariant lattice $V^{\circ} \subset V$
such that $\Vbar^{\circ}$ is of finite length.
Then $\Vbar^{\circ}$ is certainly finitely generated over $K[H]$, and so
Lemma~\ref{lem:fg} implies that $V^{\circ}$ is finitely generated over $\O[H]$.

Let $V^{\diamond}$ be another $\unif$-adically separated $H$-invariant lattice
in $V$.  We will prove that $V^{\diamond}$ is commensurable with $V^{\circ}$.
This will prove~(1).  An easy (and standard) argument then 
proves that $\Vbar^{\diamond}$ is of finite length over $K[H]$, and that
$(\Vbar^{\circ})^{\ss}$ and $(\Vbar^{\diamond})^{ss}$ are isomorphic. 
Also Lemma~\ref{lem:fg} will imply that $V^{\diamond}$ is finitely generated over
$\O[H]$.  Thus~(2) will also follow.

Since $V^{\circ}$ is finitely generated over $\O[H]$,
we may find $m \geq 0$ such that $V^{\circ} \subset \unif^{-m}V^{\diamond}.$
In proving the commensurability of $V^{\circ}$ and $V^{\diamond}$,
it is clearly no loss of generality to replace
$V^{\diamond}$ by $\unif^{-m} V^{\diamond}$, and so we may and do assume for the
remainder of the proof that $V^{\circ} \subset V^{\diamond}$.

Consider now the quotient $V^{\diamond}/V^{\circ} \subset V/V^{\circ}$.
If $V^{\diamond}/V^{\circ}$ is not of bounded exponent, then Lemma~\ref{lem:div} 
shows that it contains a non-zero $H$-invariant divisible submodule~$D$. 
The Tate module $T_p D := \plim{n} D[\unif^n]$ (the transition
maps being given by multiplication by $\unif$) is then a non-zero
$\unif$-adically complete and separated $\O$-module,
equipped with an action of $H$, and an injection
\begin{equation}
\label{eqn:complete embedding}
T_p D \hookrightarrow T_p(V/V^{\circ}) \iso \widehat{V^{\circ}}.
\end{equation}
Now
$$T_p D/\unif T_p D \iso D[\unif] \subset \dfrac{1}{\unif} V^{\circ}/V^{\circ},$$
and hence $T_p D/\unif T_p D$ is an admissible smooth $H$-representation.
Thus $T_p D$ is an admissible continuous $H$-representation over $\O$,
and so by Remark~\ref{rem:adm equiv}, the injection~(\ref{eqn:complete embedding})
is obtained from the induced embedding
$(T_p D)_{\sm} \hookrightarrow V^{\circ}$ by passing to $\unif$-adic completions.
In particular $(T_p D)_{\sm} \neq 0.$
Also, the image of the composite
$$\CK\otimes_\O (T_p D)_{\sm} \rightarrow V \rightarrow V/V^{\circ}$$
is precisely $D$, 
and so (since $D \subset V^{\diamond}/V^{\circ}$),  we conclude that
$\CK\otimes_\O (T_p D)_{\sm} \subset V^{\diamond}$.
But $\CK\otimes_{\O} (T_p D)_{\sm}$ is a non-zero $\CK$-vector space,
and hence is not $\unif$-adically separated.  This contradicts our
assumption on $V^{\diamond}$, and hence we conclude that $V^{\diamond}/V^{\circ}$ is indeed
of bounded exponent, and thus that $V^{\circ}$ and $V^{\diamond}$ are commensurable,
as required.
\end{proof}

\begin{df}
\label{df:Vbarss}
{\em
Let $H$ be a topological group satisfying Condition~\ref{cond:pro-ell},
and let $V$ be a good integral admissible smooth representation of
a topological group $H$ over~$\CK$.
If $V^{\circ}$ is a
$\unif$-adically separated $H$-invariant lattice in $V$ such that
$\Vbar^{\circ} := V^{\circ}/\unif V^{\circ}$ is of finite length
as a $K[H]$-module (which exists, by assumption),  
then we write $\Vbar^{\ss}$ to denote the semisimplification
of $\Vbar^{\circ}$ as a $K[H]$-module.  (The preceding lemma
shows that, up to isomorphism, $\Vbar^{\ss}$ is independent
of the choice of $V^{\circ}$.)
}
\end{df}

The following lemma will allow us to choose lattices in
good integral admissible smooth representations whose reductions
modulo $\unif$ have certain specified $H$-socles.

\begin{lemma} 
\label{lem:lattice}
Let $H$ be a topological group satisfying Condition~{\em \ref{cond:pro-ell}},
and let $V$ be a good integral admissible smooth representation of $H$ over $\CK$.
Let $\mathcal S$ denote the set
of isomorphism classes of
Jordan--H\"older factors of $\Vbar^{\ss}$ {\em (}as a $K[H]$-module; the
discussion of Definition~{\em \ref{df:Vbarss}} shows that this set is well-defined{\em )},
let $\mathcal T$ be a
subset of~$\mathcal S$, and suppose that $V$ contains no non-zero subrepresentation
$W$ {\em (}necessarily also good integral, by Lemma~{\em \ref{lem:sub} )}
such that every Jordan--H\"older factor of $\Wbar^{\ss}$ belongs to
$\mathcal T$.  Then there exists a
$\unif$-adically separated $H$-invariant lattice
$V^{\diamond}$ contained in $V^{\circ}$ with
the property that $\Vbar^{\diamond} := V^{\diamond}/\unif V^{\diamond}$ contains no subobject
isomorphic to an element of~$\mathcal T$.
\end{lemma}
\begin{proof}
Choose (as we may, by assumption) a $\unif$-adically separated $H$-invariant
lattice $V^{\circ}$ with the property that
$\Vbar^{\circ} := V^{\circ}/\unif V^{\circ}$ is of finite
length as a $K[H]$-module. 
Let $M \subset V/V^{\circ}$
be the maximal $\O[H]$-submodule all of whose Jordan--H\"older factors
are isomorphic to an element of $\mathcal T$.

If we form the projective limit $\plim{n} M[\unif^n] \neq 0$ (the transition maps
being given by multiplication by $\unif$),  
then
$\plim{n} M[\unif^n] \hookrightarrow \plim{n} \dfrac{1}{\unif^n}V^{\circ}/V^{\circ}
\iso \widehat{V}^{\circ},$
with saturated and $\unif$-adically complete image.
If we write $W = \CK\otimes_{\mathcal O} (\plim{n} M[\unif^n])_{\sm}$,
then Remark~\ref{rem:adm equiv} implies that $W$ is vanishes if and only
if $\plim{n} M[\unif^n]$ does.
On the other other hand, by construction $W$ is a subrepresentation
of $V$ with the property that all the Jordan--H\"older factors
of $\Wbar^{\ss}$ belong to $\mathcal T$, and so by assumption $W$ must
vanish.  Thus $\plim{n} M[\unif^n] = 0$,
and so Lemma~\ref{lem:div} implies that $M$ is of bounded exponent,
say $M = M[\unif^n]$.

Let $V^{\diamond}$ denote the preimage of $M$ in $V$.
Since $V^{\circ} \subset V^{\diamond} \subset \unif^{-n} V^{\circ}$, we see that
$V^{\diamond}$ is $\unif$-adically separated.
Since
$\Vbar^{\diamond} \iso \unif^{-1} V^{\diamond} / V^{\diamond} \hookrightarrow V/M,$
our choice of $M$ ensures that $\Vbar^{\diamond}$ contains no subobject
isomorphic to an element of $\mathcal T$.
\end{proof}

\section{Representation theory --- the case of $\GL_n$}
\label{sec:gln}

\subsection{Kirillov models} \label{subsec:kirillov}
Let $k$ be a perfect field.
Let $\ell$ be a prime distinct from the characteristic of $k$,
and let $\tk$ be a Galois extension of $k$ containing all
$\ell$-power roots of unity.

In this
section we set $G = \GL_n(E)$, where $E$ is a non-archimedean local field
of residue characteristic $\ell$.   We will define a notion of Kirillov
models for smooth representations of $G$ over a $W(k)$-algebra $A$.

We begin by recalling the basic properties
of Kirillov models associated
to smooth $W(\tk)[G]$-modules.

In the case of smooth $\C[G]$-modules, these results are found in \cite[\S 4]{BZ};
over more general algebraically closed fields they can be
found in \cite[Ch.~III.1]{Vig2}.
The extension of these results to
coefficients in $W(\tk)$ is more or less immediate.  We summarize the key facts:

Define subgroups $P_n$ and $N_n$ of $\GL_n(E)$ by setting:
$$
\begin{array}{rl}
P_n & = \, \{ \, \left(\begin{matrix} a & b \\ 0 & 1 \end{matrix} \right)
\quad \mid \quad a \in \GL_{n-1}(E), \quad b \in E^{n-1} \, \},
  \\
N_n & = \, \{ \, \left(\begin{matrix} \Id_{n-1} & b \\ 0 & 1 \end{matrix} \right)
\quad \mid \quad b \in E^{n-1} \, \} ,\\
\end{array}
$$
We consider $\GL_{n-1}(E)$ as a subgroup of $P_n$ in the obvious way,
and identify $N_n$ with $E^{n-1}$.  Note that $P_n = \GL_{n-1}(E)N_n$.
Any character 
$\psi:E^{n-1} \rightarrow W(\tk)^{\times}$ 
induces a character of $N_n$ via
$$\left(\begin{matrix} \Id_{n-1} & b \\ 0 & 1 \end{matrix} \right) \mapsto \psi(b),$$
which we again denote by $\psi$.

We fix, for the remainder of this section, a character $\psi: E \rightarrow W(\tk)^{\times}$
whose kernel is equal to the subgroup $\O_E$ of $E$.  We consider $\psi$ as a character
of $E^{n-1}$ by setting $\psi(e_1, \dots, e_{n-1}) = \psi(e_{n-1})$, and also
as a character of $N_n$ via the isomorphism of $E^{n-1}$ with $N_n$.
The subgroup $\GL_{n-1}(E)$ of $P_n$ normalizes
$N_n$, and therefore acts on the set of characters of $N_n$ by conjugation.  The stabilizer 
of $\psi$ under this action is the subgroup $P_{n-1}$ of $\GL_{n-1}(E)$.

\begin{df}  For a $W(\tk)$-algebra $A$, let $\Rep_A(G)$ denote the category of
smooth $A[G]$-modules.
Define functors $\Psi^-,\Psi^+,\Phi^-,\Phi^+,\hPhi^+$ by:
\begin{itemize}
\item $\Psi^-: \Rep_{W(\tk)}(P_n) \rightarrow \Rep_{W(\tk)}(\GL_{n-1}(E))$ is given by
$\Psi^-(V) = V_{N_n}$, the module of $N_n$-coinvariants of $V$.
\item $\Psi^+: \Rep_{W(\tk)}(\GL_{n-1}(E)) \rightarrow \Rep_{W(\tk)}(P_n)$ is the functor that
takes a $\GL_{n-1}$-module $V$ and extends the action of $\GL_{n-1}$ to $P_n$
by letting $N_n$ act trivially.
\item $\Phi^-: \Rep_{W(\tk)}(P_n) \rightarrow \Rep_{W(\tk)}(P_{n-1})$ is given by
$\Phi^-(V) = V_{\psi}$, where $V_{\psi}$ is the largest quotient of
$V$ on which $N_n$ acts by $\psi$.  (As $P_{n-1}$ normalizes $\psi$, 
$V_{\psi}$ is naturally a $P_{n-1}$-module.)
\item $\Phi^+: \Rep_{W(\tk)}(P_{n-1}) \rightarrow \Rep_{W(\tk)}(P_n)$ is given by
$$\Phi^+(V) = \cInd_{P_{n-1}N_n}^{P_n} V^{\prime},$$ 
where $V^{\prime}$ is the $P_{n-1}N_n$-module obtained from $V$ by 
letting $N_n$ act via $\Psi$, and $\cInd$ denotes smooth induction 
with compact
support.
\item $\hPhi^+: \Rep_{W(\tk)}(P_{n-1}) \rightarrow \Rep_{W(\tk)}(P_n)$ is given by
$$\hPhi^+(V) = \Ind_{P_{n-1}N_n}^{P_n} V^{\prime}.$$
\end{itemize}
\end{df}

The natural surjections of
$V$ onto $\Psi^- V$ and of $V$ onto $\Phi^- V$ are $\GL_{n-1}(E)$
and $P_{n-1}$-equivariant, respectively.

\begin{remark} 
{\em
Note that these functors differ from the ones defined in \cite{BZ}
in that they are not ``normalized.''  More precisely, the functors
defined in \cite{BZ} are twists of the above functors by the
square roots of certain modulus characters.  This makes them unsuitable
for most of our purposes, as the descent arguments we make at the end
of this section do not apply to the twisted functors defined in \cite{BZ}.
We will thus use the ``non-normalized'' functors defined above throughout
the bulk of the paper.

An unfortunate exception to this is in the proof of Proposition~\ref{prop:AIG}.
While it would in principle be possible to give a proof of Proposition~\ref{prop:AIG}
using the non-normalized functors that we use elsewhere, the normalization
of \cite{BZ} simplifies the combinatorics immensely.  We have
thus chosen to adopt this normalization for the purposes of that proof only.
}
\end{remark}

The arguments of \cite[\S 3.2]{BZ}
carry over to this setting to show:
\begin{prop}
\label{prop:BZ}
\begin{enumerate}
\item The functors $\Psi^-,\Psi^+,\Phi^-,\Phi^+,\hPhi^+$ are exact. 
\item $\Phi^+$ is left adjoint to $\Phi^-$, $\Psi^-$ is left adjoint to
$\Psi^+$, and $\Phi^-$ is left adjoint to $\hPhi^+$.  
\item $\Psi^- \Phi^+ = \Phi^- \Psi^+ = 0$.
\item The composite functors $\Psi^- \Psi^+$, $\Phi^- \hPhi^+$,
and $\Phi^- \Phi^+$ are naturally isomorphic to identity functors.
\item One has an exact sequence of functors:
$$0 \rightarrow \Phi^+\Phi^- \rightarrow \Id 
\rightarrow \Psi^+\Psi^- \rightarrow 0.$$
\end{enumerate}
\end{prop}

For our purposes, it will be necessary to have versions of these functors
for representations over $W(k)$, rather than $W(\tk)$.  The key difficulty
is that the character $\psi$ is not defined over $W(k)$.  Nonetheless,
one has:
\begin{prop}
\label{prop:Kirillov descent}
The functors $\Psi^-,\Psi^+,\Phi^-,\Phi^+,\hPhi^+$ descend to
functors defined on representations over $W(k)$.  That is, one
has a functor:
$$\Psi^-: \Rep_{W(k)}(P_n) \rightarrow \Rep_{W(k)}(\GL_{n-1}(E))$$
such that for any $W(k)[P_n]$-module $V$, one has
$$\Psi^-(V \otimes_{W(k)} W(\tk)) = \Psi^-(V) \otimes_{W(k)} W(\tk),$$
and similarly for the remaining functors.  Moreover, the statements
of Proposition~{\em \ref{prop:BZ}} apply to these functors.
\end{prop}
\begin{proof}
For the functors $\Psi^-$ and $\Psi^+$ this is clear, as the character $\psi$
does not intervene in their definition.  We thus begin with the functor
$\Phi^-$.  

Note that $\Gal(\tk/k)$ acts naturally on $W(\tk)$, and fixes $W(k)$.  Moreover,
$W(\tk)$ is faithfully flat over $W(k)$.  For $\sigma \in \Gal(\tk/k)$,
let $\psi^{\sigma}$ be the character $\sigma \circ \psi$ of $E^{\times}$.
There exists a unique $e_{\sigma} \in \O_E^{\times}$ such that
$\psi^{\sigma}(e) = \psi(e_{\sigma} e)$.  The map $\sigma \mapsto e_{\sigma}$
is a homomorphism from $\Gal(\tk/k)$ to $\O_E^{\times}$.

Consider each $e_{\sigma}$ as an element of $P_{n-1}$ via the inclusions:
$$\O_E^{\times} \subset E^{\times} \subset \GL_{n-1}(E) \subset P_{n-1}.$$
If we consider $\psi$ as a character of $N_n$, we have 
$\psi^{\sigma}(u) = \psi(e_{\sigma} u e_{\sigma}^{-1})$ for all $u \in N_n$.

Now let $V$ be a smooth $P_n$-representation over $W(k)$, and let $\tV$
be the representation $V \otimes_{W(k)} W(\tk)$.  Then we have a
$W(\tk)$-semilinear action of $\Gal(\tk/k)$ on $\tV$, that
fixes $V$.  By definition, $\Phi^- \tV$ is the quotient
of $\tV$ by the $W(\tk)[P_{n-1}]$-submodule generated
by all vectors of the form $uv - \psi(u)v$, for $u \in N_n$ and $v \in \tV.$

The Galois action on $\tV$ does not descend to
$\Phi^- \tV$, but a twist of it does.  For an element
$[v] \in \Phi^- \tV$, represented by an element $v$ of $\tV$,
define $\sigma [v] = [e_{\sigma} \sigma v]$.  This is well-defined,
since if $v = uw - \psi(u)w$ for some $w \in \tV$ and $u \in N_n$,
we can set $w' = e_{\sigma} \sigma w$, $u' = e_{\sigma} u e_{\sigma}^{-1}$, and then
$e_{\sigma} \sigma v = u' w' - \psi(u') w'$.  We thus obtain a $W(\tk)$-semilinear
action of $\Gal(\tk/k)$ on $\Phi^- \tV$; we define $\Phi^- V$ to be the invariants
under this action.  This is clearly functorial with the desired properties.
(Note, however, that the surjection $\tV \rightarrow \Phi^- \tV$ does {\em not}
descend to a natural surjection of $V$ onto $\Phi^- V$.)

Now let $V$ be a smooth $P_{n-1}$-representation over $W(k)$.  Then
$\Phi^+ \tV$ and $\hPhi^+ \tV$ can both be realized as spaces
of functions: $f: P_n \rightarrow \tV$, such that for any $h \in P_{n-1}$
and any $u$ in $N_n$, we have $f(ghu) = \psi(u)hf(g)$.  Define an action
of $\Gal(\tk/k)$ on the space of such functions by setting
$(\sigma f)(g) = e_{\sigma}^{-1}\sigma f(g)$.  This preserves the identity
$f(ghu) = \psi(u)hf(g)$, and so defines a $W(\tk)$-semilinear action
of $\Gal(\tk/k)$ on $\Phi^+ \tV$ and $\hPhi^+ \tV$.  We set
$\Phi^+ V$ and $\hPhi^+ V$ to be the invariants of this action in
$\Phi^+ \tV$ and $\hPhi^+ \tV$, respectively.  Note that if $V$ is a smooth
$P_n$-representation, then the natural maps:
$\Phi^+ \Phi^-\tV \rightarrow \tV$ and $\tV \rightarrow \hPhi^+ \Phi^- \tV$
are $\Gal(\tk/k)$-equivariant, and hence descend to $V$.

Using this, one easily verifies the adjointness property (2) of Proposition~\ref{prop:BZ} for the functors
over $W(k)$.  Properties (1), (3), (4), and (5) then follow by base change and the
fact that $W(\tk)$ is faithfully flat over $W(k)$.
\end{proof}

Note that if $A$ is a Noetherian $W(k)$-algebra, and $V$ is a smooth representation
of $P^{n-1}$ over $A$, then the modules $\Psi^- V$, $\Phi^- V$
obtained by treating $V$ as a representation of $P^{n-1}$ over $W(k)$
and applying the appropriate functors inherit an $A$-module structure.  We
can thus define the functors $\Psi^-, \Psi^+, \Phi^-,$ etc. on suitable
categories of smooth representations over $A$.  It is then clear that
if $B$ is a Noetherian $A$-algebra, one has $\Psi^-(V \otimes_A B) = \Psi^-(V) \otimes_A B$,
and similarly for the remaining functors.

Finally, observe that the functors $\Psi^-,\Psi^+,\Phi^-,\Phi^+,\hPhi^+$ commute 
with tensor products; 
that is, if $M$ is an $A$-module,
then $\Psi^-(V \otimes_A M) \cong \Psi^-(V) \otimes_A M$, and similarly
for the other functors.



We now define the ``derivatives'' of a smooth $P_n$-representation $V$.
For $0 \leq r \leq n$, we set $V^{(r)} = \Psi^-(\Phi^-)^{r-1} V$;
$V^{(r)}$ is a representation of $\GL_{n-r}(E)$.  If $A$ is a $W(\tk)$-algebra,
then one has a
$\GL_{n-r}(E)$-equivariant surjection $V \rightarrow V^{(r)}$ (but this is not true
if $A$ is only a $W(k)$-algebra.)

Note that $V^{(n)}$ is
simply an $A$-module.  The adjointness properties of Proposition \ref{prop:BZ}
give, for any $V$, maps:
$$V \rightarrow (\hPhi^+)^{(n-1)}\Psi^+(V^{(n)}).$$
$$(\Phi^+)^{n-1}\Psi^+(V^{(n)}) \rightarrow V.$$
The image of $V$ in $(\hPhi^+)^{(n-1)}\Psi^+(V^{(n)})$ is
called the Kirillov model of $V$.

The exact sequence of Proposition \ref{prop:BZ} implies that the map
$$(\Phi^+)^{n-1}\Psi^+(V^{(n)}) \rightarrow V$$
is injective; we denote its image by $\J(V)$.  The space $\J(V)$
is often referred to as the space of Schwartz functions in $V$.

\begin{lemma}
\label{lem:Kirillov sub}
Let $V$ be a smooth $P_n$-module over $A$.
Set $\tA = A \otimes_{W(k)} W(\tk)$, and let $\tV = V \otimes_{W(k)} W(\tk)$.
The modules $V$ and $\J(V)$ each contain an $A$-submodule
$W$ such that $W$ is isomorphic to $V^{(n)}$, and
$W \otimes_A \tA$  maps isomorphically to $\tV^{(n)}$
under the surjection $\tV \rightarrow \tV^{(n)}$.
\end{lemma}
\begin{proof}
It suffices to show that $\J(V)$ contains such a submodule, as $\J(V)$
embeds in $V$.  We have 
$$\J(V) = (\Phi^+)^{n-1}\Psi^+ V^{(n)} = [(\Phi^+)^{n-1}\Psi^+ W(k)] \otimes_{W(k)} V^{(n)}.$$
Moreover, the map $\J(\tV) \rightarrow \tV^{(n)}$ arises from the surjection:
$$(\Phi^+)^{n-1}\Psi^+ W(\tk) \rightarrow W(\tk)$$
by tensoring over $W(\tk)$ with $\tV^{(n)}$.  It thus suffices to
construct a submodule $W$ of $[(\Phi^+)^{n-1}\Psi^+ W(k)]$ that is free of rank
one over $W(k)$ and such that $W \otimes_{W(k)} W(\tk)$ maps isomorphically
onto $W(\tk)$ under the surjection
$$(\Phi^+)^{n-1}\Psi^+ W(\tk) \rightarrow W(\tk).$$
Thus amounts to simply choosing any element of $(\Phi^+)^{(n-1)}\Psi^+ W(k)$
that maps to an element of $W(\tk) \setminus \unif W(\tk)$, where $\unif$
is the uniformizer of $W(\tk)$.  Such an element clearly exists,
as otherwise the image of the composition:
$$[(\Phi^+)^{n-1}\Psi^+ W(k)] \otimes_{W(k)} W(\tk) \iso
(\Phi^+)^{n-1}\Psi^+ W(\tk) \rightarrow W(\tk)$$
would be contained in $\unif W(\tk)$.
\end{proof}

Over a field, the top derivatives are multiplicative with respect
to parabolic induction:

\begin{prop}
\label{prop:leibnitz}
Let $V$ and $W$ be admissible $k$-representations of $\GL_n(E)$ and
$\GL_m(E)$, respectively.  Let $P \subset \GL_{n+m}(E)$ be the parabolic
subgroup of $G$ given by:
$$
\begin{array}{rl}
P & = \, \{ \, \left(\begin{matrix} a & b \\ 0 & d \end{matrix} \right)
\quad \mid \quad a \in \GL_n(E), \quad d \in \GL_m(E) \, \},
\end{array}
$$
and consider $V \otimes W$ as a representation of $P$ by letting
the unipotent radical of $P$ act trivially.  Then
$$(\Ind_P^{\GL_{n+m}(E)} V \otimes W)^{(n+m)} = V^{(n)} \otimes W^{(m)}.$$
\end{prop}
\begin{proof}
This is a special case of \cite[Lem.~1.10]{Vig2}.
Note that the derivative functors
used in \cite{Vig2} are normalized differently than ours; this normalization does not
affect the top derivative $V^{(n)}$ of a representation $V$ of $\GL_n(E)$.
\end{proof}

We also have:
\begin{theorem}
\label{thm:kirillov}
Let $V$ be an absolutely irreducible admissible representation of $\GL_n(E)$ over 
a field $K$ that is a $W(k)$-algebra.
Then $V^{(n)}$ is either zero or a one-dimensional $K$-vector space, and is
one-dimensional if $V$ is cuspidal.
\end{theorem}
\begin{proof}
This is \cite[III.5.10]{Vig2}, over $\overline{\F}_p$ or ${\Qbar}_p$; the
same argument works over an arbitrary field containing the $\ell$-power roots
of unity.  The general case follows by extending scalars.
\end{proof}

\begin{df}
\label{df:generic}
We say that
an absolutely irreducible admissible representation $V$ of $\GL_n(E)$ over 
a field $K$ that is a $W(k)$-algebra
is {\em generic} if $V^{(n)}$ is one-dimensional over $K$.
\end{df}

We now turn to finiteness properties of the derivative.

\begin{lemma}
\label{lem:nakayama sum}
Let $A$ be a Noetherian local ring,
and suppose that $M$ is a submodule of a direct sum of
finitely generated $A$-modules.  If $M/\mathfrak m M$ is finite dimensional,
then $M$ is finitely generated.
\end{lemma}
\begin{proof}
Our assumption on $M/\mathfrak m M$
allows us to choose a finitely generated submodule
$N$ of $M$ such that $N + \mathfrak m M = M,$ or equivalently such that
$\mathfrak m (M/N) = M/N$.  Nakayama's Lemma then shows that any finitely
generated quotient of $M/N$ must vanish. 
Since by assumption $M$ embeds into
a direct sum of finitely generated $A$-modules, we may find
a finitely generated $A$-module $X$, and an $A$-module $Y$ which
is a direct sum of finitely generated $A$-modules, and an embedding
$M \subset X \bigoplus Y$, such that $N \subset X$.
On the one hand $M/(M \bigcap X)$ is a quotient
of $M/N$, and hence has no non-vanishing finitely generated quotients.
On the other, the projection of $X \bigoplus Y$ onto $Y$ induces
an embedding of $M/(M \bigcap X)$ into $Y$.  Thus $M/(M\bigcap X) = 0,$
and so $M \subset X$ is finitely generated.
\end{proof}

We note for future reference
the following corollary of the preceding lemma.

\begin{cor}
\label{cor:zero}
Let $A$ be a Noetherian local ring,
and suppose that $M$ is a submodule of a direct sum of
finitely generated $A$-modules.  If $M/\mathfrak m M$ vanishes,
then $M$ itself vanishes.
\end{cor}
\begin{proof}
The preceding result implies that $M$ is finitely generated.
The corollary thus follows from Nakayama's Lemma.
\end{proof}

\begin{theorem}
\label{thm:fg}
Suppose that $A$ is a Noetherian
local $W(k)$-algebra with maximal ideal ${\mathfrak m}$.
Let $V$ be an admissible representation of $\GL_n$
over $A$, and suppose that $(V/{\mathfrak m}V)^{(n)}$ is finite dimensional
over $A/{\mathfrak m}$.  Then $V^{(n)}$ is a finitely generated $A$-module.  
\end{theorem}
\begin{proof}
As derivatives commute with tensor products, we have
$$V^{(n)}/{\mathfrak m}V^{(n)} \cong (V/{\mathfrak m}V)^{(n)}.$$
On the other hand, we have already
observed that $\J(V)$ (and hence $V$) contains an $A$-submodule
isomorphic to $V^{(n)}$.  As $V$ is a direct sum of finitely generated $A$-modules,
the result follows from Lemma \ref{lem:nakayama sum}.
\end{proof}

\begin{remark}
{\em
In the setting of Theorem~\ref{thm:fg},
if $V/{\mathfrak m}V$ has finite length, then $(V/{\mathfrak m}V)^{(n)}$  is finite dimensional 
over $A/{\mathfrak m}$ by Theorem \ref{thm:kirillov}.  Thus Theorem~\ref{thm:fg}
applies to all admissible representations of $\GL_n$ over $A$ such that
$V/{\mathfrak m}V$ has finite length.
}
\end{remark}

Theorem~\ref{thm:fg}
allows us to establish the following extension of Proposition~\ref{prop:leibnitz}:

\begin{cor}
\label{cor:leibnitz}
Let $A$ be a reduced Noetherian
local $W(k)$-algebra with maximal ideal ${\mathfrak m}$.
Let $V$ and $W$ be admissible smooth $A$-representations of $\GL_n(E)$ and $\GL_m(E)$,
and let $P$ be the parabolic subgroup of $\GL_{n+m}(E)$ defined in the statement of
Proposition~\ref{prop:leibnitz}.  Then, if $V^{(n)}$ and $W^{(m)}$ are free of
rank one over $A$, so is the $A$-module:
$$(\Ind_P^{\GL_{n+m}(E)} V \otimes W)^{(n + m)}.$$
\end{cor}
\begin{proof}
For each minimal prime ${\mathfrak a}$ of $A$, let $\kappa({\mathfrak a})$ be its residue field.
Set $V_{\mathfrak a} = V \otimes_A \kappa({\mathfrak a})$
and $W_{\mathfrak a} = W \otimes_A \kappa({\mathfrak a})$.
We have isomorphisms:
$$(\Ind_P^{GL_{n+m}(E)} V \otimes W)^{(n + m)} \otimes_A \kappa({\mathfrak a}) \iso
(\Ind_P^{\GL_{n+m}(E)} V_{\mathfrak a} \otimes W_{\mathfrak a})^{(n + m)},$$
and the latter is one-dimensional by Proposition~\ref{prop:leibnitz}.  Thus
in particular the annihilator of $(\Ind_P^{\GL_{n+m}(E)} V \otimes W)^{(n + m)}$ 
as an $A$-module is the zero ideal of $A$.

On the other hand, let $\Vbar = V/{\mathfrak m}V$ and $\Wbar = W/{\mathfrak m}W$.
Then we have isomorphisms:
$$(\Ind_P^{\GL_{n+m}(E)} V \otimes W)^{(n + m)} \otimes_A A/{\mathfrak m} \iso
(\Ind_P^{\GL_{n+m}(E)} \Vbar \otimes \Wbar)^{(n + m)},$$
and the latter is again one-dimensional by Proposition~\ref{prop:leibnitz}. 
Theorem~\ref{thm:fg} shows that
$(\Ind_P^{\GL_{n + m}(E)} V \otimes W)^{(n+ m)}$
is furthermore finitely generated over $A$,
and thus it follows by Nakayama's lemma that
$(\Ind_P^{\GL_{n + m}(E)} V \otimes W)^{(n+ m)}$
is a cyclic
$A$-module, and hence (taking into account that it is faithful, as we proved
above) is free of rank one.
\end{proof}

We will also need a generalization of this machinery to a product of
$\GL_n(E_i)$ for various local fields $E_i$ of residue characteristics $\ell_i$,
all prime to the residue characteristic of $k$.
Fix a finite collection of such $E_i$, indexed by a set $S$, and let $G$ be the product of the
groups $\GL_n(E_i)$ for all $i$.  Let $P_n$ be the product of the subgroups $P_n(E_i)$ of
$\GL_n(E_i)$.  

Now if we fix for each $i$ a character $\psi_i: N_n(E_i) \rightarrow W(\tk)^{\times}$,
we can define functors $\Psi^{-,i},\Phi^{-,i}, \Psi^{+,i},\Phi^{+,i},\hPhi^{+,i}$
as follows: if $H$ is any topological group, and $V$ is a $P_n(E_i) \times H$-module,
then $\Psi^{-,i}(V)$ is the $\GL_{n-1}(E_i) \times H$-module defined by applying
$\Psi^-$ to $V$ (considered as a $P_n(E_i)$-module,) and then taking the natural
action of $H$ on $\Psi^-(V)$.  The other functors are defined similarly.
Note that if $V$ is a $G$-module over $A$ (or even a module over the product of
the $P_n(E_i)$,) then 
$\Psi^{-,i}\Psi^{-,j}V = \Psi^{-,j}\Psi^{-,i}V$ (here the equality denotes
a natural isomorphism),
and the other functors have 
similar commutativity properties.   If $S^{\prime}$ is a subset of $S$,
the composition of functors
$\Psi^{-,i}$ for all $i \in S^{\prime}$ is a thus a functor that takes
an $A$-module $V$ over the product of the groups $P_n(E_i)$ to an $A$-module
with an action of $P_{n-1}(E_i)$ for each $i \in S^{\prime}$ and of $P_n(E_i)$ for each $i$
not in $S^{\prime}$.  This composition is independent (up to natural isomorphism)
of the order in which we compose the
functors; we denote it by $\Psi^{-,S^{\prime}}$.  Similarly define 
$\Phi^{-,S^{\prime}}$, $\Psi^{+,S^{\prime}}$,
etc.  Finally, if $V$ is an $A[G]$-module, define $V^{(n),S^{\prime}}$ 
to be the representation $\Phi^{-,S^{\prime}}(\Psi^{-,S^{\prime}})^{n-1}V$.  
Note that the functors $\Phi^{-,S^{\prime}}$, $\Psi^{+,S^{\prime}}$,
etc.\ satisfy analogues of properties (1)-(4) of Proposition~\ref{prop:BZ}.

When it is clear from the context what $S$ and $E_i$ we are working with, we will
denote $\Phi^{+,S}$, $\Psi^{+,S}$, etc.\ by $\Phi^+,\Psi^+$, and so forth.

Similarly, if $V$ is an $A[G]$-module, we define $\J_i(V)$ to be space of
Schwartz functions in $V$ for the action of $\GL_n(E_i)$ on $V$; this has an
action of $P_n(E_i)$ and of $\GL_n(E_j)$ for $j$ not equal to $i$.  Note that
$\J_i\J_j(V) = \J_j\J_i(V)$, so that we can define, for any $S^{\prime} \subset S$,
the functor $\J_{S^{\prime}}$ to be the composition (in any order)
of the functors $\J_i$ for $i$ in $S^{\prime}$.  Then $\J_{S^{\prime}}(V)$ 
is the smallest $A$-submodule of $V$, stable under $P_n(E_i)$ for 
$i$ in $S^{\prime}$ and $\GL_n(E_i)$ for $i$ not in $S^{\prime}$,
such that the map $\J_S(V)^{(n),S^{\prime}} \rightarrow V^{(n),S^{\prime}}$ 
is an isomorphism.  By construction,
the functor $\J_{S^{\prime}}$ is left adjoint to the functor 
$V \mapsto V^{(n),S^{\prime}}$ (when the
latter is considered as a functor from $A[P_n]$-modules to $A$-modules).

Now, precisely as in the proof of Theorem~\ref{thm:fg}, we have:
\begin{theorem}
\label{thm:fgtensor}
Let $V$ be an admissible representation of $G$
over a Noetherian local $W(k)$-algebra $A$, and suppose that 
$(V/{\mathfrak m}V)^{(n)}$ is finite dimensional over $A/{\mathfrak m}A$.  
Then $V^{(n)}$ is a finitely generated $A$-module.  
\end{theorem}

We also have an analogue of Theorem~\ref{thm:kirillov}
\begin{theorem}
\label{thm:kirillov tensor}
Let $V$ be an absolutely irreducible admissible representation of $G$ over $k$.
Then $V^{(n)}$ is either zero or a one-dimensional $k$-vector space.
\end{theorem}
\begin{proof}
The representation $V$ splits as a tensor product of absolutely irreducible
representations $V_i$ of $\GL_i(E_i)$ for all $i \in I$.  It follows
that $V^{(n)}$ is the tensor product of the $V_i^{(n)}$.  Hence this result
is an immediate consequence of Theorem~\ref{thm:kirillov}.
\end{proof}

\begin{prop}
\label{prop:schwartz endomorphisms}
Let $V$ be an $A[G]$-module, and suppose
that $V^{(n)}$ is free of rank one over $A$.  Then the map
$A \rightarrow \End_{P_n}(\J(V))$ is an isomorphism.
\end{prop}
\begin{proof}
By the adjointness properties of the functors $\Psi^+$ and $\Phi^+$ we have
natural isomorphisms:
$$\End_{P_n}(\J(\tV)) \iso \Hom_{A}(V^{(n)},\Psi^-(\Phi^-)^{n-1} \J(V)) \iso
\End_{A}(V^{(n)}).$$
The result follows immediately.
\end{proof}

\subsection{Essentially AIG representations}
\label{subsec:AIG}
Let $K$ be a field of characteristic different from $\ell$.

\begin{df}
\label{df:ess AIG}
{\em We say that a smooth
representation $V$ of $G:= \GL_n(E)$ is essentially
absolutely irreducible and generic (``essentially AIG'' for short)
if:
\begin{enumerate}
\item The $G$-socle $\soc(V)$ is absolutely irreducible and generic.
\item The quotient $V/\soc(V)$ contains no generic Jordan--H\"older factors;
equivalently, $\bigl(V/\soc(V)\bigr)^{(n)} = 0$.
\item The representation $V$ is the sum (or equivalently, the union) of its finite length submodules.
\end{enumerate}
}
\end{df}

\begin{lemma}
\label{lem:AIG basic}
\begin{enumerate}
\item If $V$ is an essentially~AIG smooth representation
of $G$,
and if $\chi: E^{\times} \rightarrow k^{\times}$ is a continuous character,
then $(\chi\circ\det)\otimes V$ is again essentially~AIG.
\item
If $V$ is an essentially~AIG smooth $G$-representation,
and if $U \subset V$ is a non-zero smooth $G$-subrepresentation,
then $U$ is also essentially~AIG, and furthermore $\soc(U) = \soc(V)$.
\item If $U$ and $V$ are essentially~AIG admissible smooth $G$-representations,
then restricting to socles induces an embedding
$$\Hom_G(U,V) \hookrightarrow \Hom_G\bigl(\soc(U),\soc(V)\bigr).$$
\item
Any non-zero $G$-equivariant homomorphism between essentially~AIG
admissible smooth $G$-representations
is an embedding.
\end{enumerate}
\end{lemma}
\begin{proof}
Claim~(1) is clear.

If $U \subset V$ is as in~(2),
then $0 \neq \soc(U) \subset \soc(V)$. Since the latter
is absolutely irreducible, we find that $\soc(U) = \soc(V)$,
and so in particular $\soc(U)$ is absolutely irreducible
and generic.   Furthermore, we see that
$U/\soc(U) \hookrightarrow V/\soc(V).$  Since the latter
representation contains no generic Jordan--H\"older factors,
neither does the former.  Finally, every element of $U$ is contained
in a finite length submodule of $V$; the intersection of this with $U$
is also finite length.  Thus $U$ is the union of its finite length submodules,
and is therefore essentially~AIG, proving~(2).

Now suppose that $\phi:U \rightarrow V$ is a map
of essentially~AIG representations, as in~(3).  If $\phi\bigl(\soc(U)\bigr) = 0,$
then $\phi$ factors to induce a map $U/\soc(U) \rightarrow V.$ But
the source of this map has no generic Jordan--H\"older factors, while its target
has generic socle.  Thus this map vanishes, and hence $\phi$ vanishes.  This
proves~(3).

To prove~(4), suppose given $\phi:U \rightarrow V$ as above.
If $\kernel \phi \neq 0,$ then it has a non-zero socle.  As $\soc(U)$
is irreducible, we conclude that $\soc(U) \subset \kernel \phi.$
Part~(3) then implies that $\phi = 0.$
\end{proof}

\begin{lemma}
\label{lem:AIG homs}
If $V$ is an essentially~AIG smooth representation
of $G$ over $K$, and if $U$ is a non-zero submodule of $V,$
then $\Hom_G(U,V)$ is one-dimensional over $K$.  In particular,
$\End_G(V) = K.$
\end{lemma}
\begin{proof}
Part~(2) of the preceding lemma shows that $U$ is again essentially~AIG
and that $\soc(U) = \soc(V)$.  Part~(3) of the same lemma shows that
restriction to socles induces an embedding
$$
\Hom_G(U,V) \hookrightarrow \Hom_G\bigl(\soc(U),\soc(V)\bigr)
= \End_G\bigl(\soc(V)\bigr) = K$$
(where the first equality follows from the already noted equality of socles,
and the second equality following from the absolute irreducibility of $\soc(V)$).
Since $\Hom_G(U,V)$
is non-zero (as $U$ embeds into $V$ by assumption), it must therefore be 
one-dimensional, as claimed.
\end{proof}

\begin{lemma}
\label{lem:AIG central}
If $V$ is an essentially~AIG smooth representation
of $G$ over $K$, then $V$ admits a central character.
\end{lemma}
\begin{proof}
The preceding lemma shows that $\Aut_G(V) = K^{\times}$.  Since the centre
$Z$ of $G$ acts as automorphisms of $V$, the lemma follows.
\end{proof}

\begin{lemma}
\label{lem:AIG descent}
Let $V$ and $W$ be essentially~AIG smooth representations of
$G$ over $K$, and let $K'$ be a finite separable extension of $K$.  For any
map $f: V \otimes_K K' \rightarrow W \otimes_K K'$, there exists a scalar
$c \in (K')^{\times}$
such that $cf$ descends uniquely to a map $V \rightarrow W$.
\end{lemma}
\begin{proof}
We may assume $K'$ is Galois over $K$, and that $f$ is nonzero (and thus injective).
For $\sigma \in \Gal(K'/K)$, define $f^{\sigma}$ by $f^{\sigma}(x) = \sigma f(\sigma^{-1} x).$
Then $f^\sigma = c_{\sigma} f$ for a scalar $c_{\sigma} \in (k')^{\times}.$  The
$c_{\sigma}$ give a cocycle in $H^1(\Gal(K'/K), (K')^{\times})$ and are therefore
a coboundary; that is, there exists a $c \in (K')^{\times}$ such that
$c_{\sigma} = \frac{c}{\sigma c}$ for all $\sigma$.  Then $cf$ is Galois-equivariant,
and thus descends to $K$.
\end{proof}

\begin{df}
\label{df:envelope}
{\em If $V$ is an essentially~AIG admissible smooth $G$-representation,
then we say that a smooth representation $W$ is an essentially~AIG
envelope of $V$ if:
\begin{enumerate}
\item $W$ is itself essentially~AIG.
\item There is an $G$-equivariant embedding $V \hookrightarrow W$
(which Lemma~\ref{lem:AIG homs} shows 
is then unique up to multiplication by a non-zero scalar).
\item $W$ is maximal with respect to properties~(1) and~(2),
i.e.\ if $V \hookrightarrow Y$ is any $G$-equivariant
embedding with $Y$ essentially~AIG admissible smooth,
then there is a $G$-equivariant embedding  $Y \hookrightarrow W.$
\end{enumerate}
}
\end{df}

\begin{prop}
\label{prop:envelope}
If $V$ is an essentially~AIG admissible smooth $G$-representation,
then $V$ admits an essentially~AIG envelope, which is unique up to isomorphism.
\end{prop}
\begin{proof}
Let $V\hookrightarrow I$
be an injective envelope of $V$ in the category of smooth representations.
Let $U$ denote the subrepresentation of $I/V$ obtained by taking the sum of all the
non-generic subrepresentations of $I/V$ (so $U$ is the maximal subrepresentation
of $I/V$ for which $U^{(n)} = 0$), and define $X$ to be the preimage of $U$ in $I$.
Let $W$ be the sum of all of the finite length submodules of $X$.  By construction,
the socle of $W$ is generic, $\bigl(W/\soc(W)\bigr)^{(n)} = 0$, and $W$ is the sum of its
finite length submodules, so $W$ is essentially AIG.

If $V \hookrightarrow  Y$ is an embedding as in~(3), then (since $I$ is injective)
we may extend the embedding of $V$ into $I$ to an embedding of $Y$ into $I$.  Since
every Jordan--H\"older constituent of $Y/V$ is nongeneric, the image of $Y$ lies in $X$.
Moreover, $Y$ is the sum of its finite length submodules, so the image of $Y$ lies in $W$.
\end{proof}

If $V$ is an essentially~AIG smooth $G$-representation,
then we write $\env(V)$ to denote the essentially~AIG envelope of $V$
(which by the preceding proposition exists, and is unique up to isomorphism).

\begin{lemma}
\label{lem:envelope twisting}
Let $V$ be an essentially~AIG admissible smooth $G$-representation.
If $\chi:E^{\times} \rightarrow K^{\times}$ is a continuous
character, then there is an isomorphism
$$\env\bigl((\chi\circ\det)\otimes V\bigr) \iso
(\chi\circ\det)\otimes\env(V);$$
i.e.\ the formation of essentially~AIG envelopes is compatible with
twisting.
\end{lemma}
\begin{proof}
This is immediate from Lemma~\ref{lem:AIG basic}.
\end{proof}

It seems likely that essentially AIG representations actually have finite length.
Unfortunately the techniques developed in this paper do not seem sufficient to
establish this in general, in the case when the characteristic of $K$ is positive.
In the case when $K$ is of characteristic zero, this finiteness follows from
Lemma~\ref{lem:char zero AIG} below, while in the case when $n=2$, it is easy
to establish for arbitrary $K$.  (See Proposition~\ref{prop:n = 2 AIG} below.)
In the remainder of this section, we establish a weaker finiteness
result for essentially AIG representations that will suffice for our purposes.
The key tool will be the notion of supercuspidal support; we recall the definition
below.

Let $\{\pi_1, \dots, \pi_r\}$ be a multiset of irreducible cuspidal
representations of the groups
$\GL_{n_1}(E), \dots, \GL_{n_r}(E)$, for $n_1, \dots, n_r$
such that $\sum n_i = n$.  

\begin{df}
An irreducible representation $\pi$ of $G$ over $\overline{K}$
has {\em supercuspidal support} equal to $\{\pi_1, \dots, \pi_r\}$ if each $\pi_i$
is supercuspidal, and there exists a parabolic subgroup $P = MU$ of $G$, with
$M$ isomorphic to the product of the $\GL_{n_i}$, such that
$\pi$ is isomorphic to a Jordan--H\"older constituent
of the normalized parabolic induction
$$\Ind_P^G \pi_1 \otimes \dots \otimes \pi_r.$$
\end{df}

\begin{df}
An irreducible representation $\pi$ of $G$ over $\overline{k}$
has {\em cuspidal support} equal to $\{\pi_1, \dots, \pi_r\}$ 
if there exists a parabolic subgroup $P = MU$ of $G$, with
$M$ isomorphic to the product of the $\GL_{n_i}$, such that
$\pi$ is isomorphic to a {\em quotient}
of the normalized parabolic induction
$$\Ind_P^G \pi_1 \otimes \dots \otimes \pi_r,$$
for some choice of ordering of $\pi_1, \dots, \pi_r$.
\end{df}

Both the cuspidal and supercuspidal support of $\pi$ are uniquely determined
(as multisets of isomorphism classes of irreducible representations),
by $\pi$.  Let $\scs(\pi)$ (resp. $\cs(\pi)$) denote the supercuspidal support 
(resp. cuspidal support) of $\pi$.
The following basic facts about cuspidal and supercuspidal support
are standard, and are easy consequences of Frobenius reciprocity.

\begin{prop}
Let $P = MU$ be a parabolic subgroup of $G$, with $M$ isomorphic
to $\prod_i \GL_{n_i}$.  Then:
\begin{enumerate}
\item Let $\pi_i$ be an irreducible admissible representation of
$\GL_{n_i}$ for all $i$.  If $\pi$ is a Jordan--H\"older constituent
of 
$$\Ind_P^G \pi_1 \otimes \dots \otimes \pi_r,$$
then $\scs(\pi)$ is the multiset sum of $\scs(\pi_i)$ for all $i$.
\item Let $\pi_i$ be an irreducible admissible representation of
$\GL_{n_i}$ for all $i$.  If $\pi$ is a submodule or quotient
of 
$$\Ind_P^G \pi_1 \otimes \dots \otimes \pi_r,$$
then $\cs(\pi)$ is the multiset sum of $\cs(\pi_i)$ for all $i$.
\item Let $\pi$ be
an irreducible admissible representation of $G$ over $\overline{k}$,
and let $\pi' = \pi_1 \otimes \dots \otimes \pi_r$ be a Jordan--H\"older
constituent of $\Res_G^P \pi$.  Then $\scs(\pi)$ {\em (}resp.~$\cs(\pi)${\em )}
is the multiset sum of $\scs(\pi_i)$ {\em (}resp.~$\cs(\pi_i)${\em )} for all $i$.
\end{enumerate}
\end{prop}

\begin{prop}
\label{prop:generic supercuspidal support}
Let $\{\pi_1, \dots, \pi_r\}$ be a multiset of supercuspidal representations
of $\GL_{n_i}$ over $\overline{K}$.  There exists, up to isomorphism, exactly
one irreducible generic representation $\pi$ of $G$ over $\overline{K}$
with supercuspidal support equal to $\{\pi_1, \dots, \pi_r\}$.
\end{prop}
\begin{proof}
A representation $\pi$ with supercuspidal support $\{\pi_1, \dots, \pi_r\}$ 
is isomorphic to a generic Jordan--H\"older constituent
of $\Ind_P^G \pi_1 \otimes \dots \otimes \pi_r$.
By Theorem~\ref{thm:kirillov} and Proposition~\ref{prop:leibnitz}, 
the top derivative of
$(\Ind_P^G \pi_1 \otimes \dots \otimes \pi_r)$ is one-dimensional,
so it has exactly one generic Jordan--H\"older constituent.
\end{proof}

We are now in a position to show:
\begin{theorem}
\label{thm:ext}
Let $\pi,\pi'$ be irreducible admissible representations of $G$
over $\overline{K}$, or more generally of a Levi subgroup $M$ of
$G$, and suppose that for some $i$, $\Ext^i(\pi,\pi')$ is
nonzero.  Then $\pi$ and $\pi'$ have the same supercuspidal support.
\end{theorem}
\begin{proof}
We first consider the case when $\pi$ is cuspidal.  Suppose
$\pi'$ is supercuspidal.  By~\cite[IV.6.2]{Vig3}, the category of smooth
representations of $M$ factors as a product of blocks; two
irreducible representations of $M$ are in the same block if,
and only if, their supercuspidal supports are {\em inertially equivalent},
that is, if and only if they coincide up to twisting by unramified characters.  
Thus, if $\pi'$ is supercuspidal and $\Ext^i(\pi,\pi')$ is nonzero,
then $\pi$ is a twist of $\pi'$ by an unramified character of $M$.
In this case, by the results of~\cite[IV.1]{Vig3},
$\pi'$ and $\pi$ both contain a supercuspidal
type $(K,\rho)$ for $M$.  The Hecke algebra attached to $(K,\rho)$
is of the form $k[x_1^{\pm 1}, \dots, x_r^{\pm 1}]$
and (because $\pi$ and $\pi'$ are supercuspidal), the block
containing $\pi$ and $\pi'$ is equivalent to the category of
modules for this Hecke algebra.  In particular, if
${\mathfrak m}$ and ${\mathfrak m}'$ are the maximal ideals
of this Hecke algebra corresponding to $\pi$ and $\pi'$,
then $\Ext^i(\pi,\pi')$ is annihilated by both
${\mathfrak m}$ and ${\mathfrak m}'$ and therefore
vanishes unless ${\mathfrak m} = {\mathfrak m}'$.  In this
case $\pi = \pi'$ and the result is established.

Next suppose $\pi$ is cuspidal and $\pi'$ is not cuspidal.
We proceed by induction on $i$, the case $i=0$ being clear.  Fix $i$,
and assume $\Ext^{i-1}(\pi,X) = 0$ for any non-cuspidal $X$ whose
supercuspidal support differs from that of $\pi$.
As $\pi'$ is non-cuspidal, there is a proper parabolic subgroup $P' = M'U'$ of~$M$, and
a cuspidal representation $\sigma$ of~$M'$, such that
$\pi'$ arises as a submodule of $\Ind_{P'}^M \sigma$.
We thus have an exact sequence:
$$0 \rightarrow \pi' \rightarrow \Ind_{P'}^M \sigma \rightarrow C \rightarrow 0,$$
where $C$ is the cokernel of the inclusion of $\pi'$ in $\Ind_{P'}^M \sigma$.
By Frobenius reciprocity $\Ext^j(\pi, \Ind_{P'}^M \sigma) = \Ext^j(\Res_M^{P'} \pi, \sigma)$
for all $j$,
and the latter vanishes because $\pi$ is cuspidal.  Thus $\Ext^i(\pi,\pi')$ is
isomorphic to $\Ext^{i-1}(\pi,C)$.  By the inductive hypothesis, the latter
vanishes unless $\pi$ and $C$ have the same supercuspidal support, but
$\pi'$ and $C$ also have the same supercuspidal support, so the result holds in
this case as well.

If $\pi$ and $\pi'$ are both cuspidal, we may assume $\pi'$ is not supercuspidal
as we have already considered that case.  Thus there is a proper
parabolic subgroup $P' = M'U'$
of~$M$, and a supercuspidal representation $\sigma$ of~$M'$, such that
$\pi'$ is a Jordan--H\"older constituent of $\Ind_{P'}^M \sigma$.  As cuspidal
representations are generic, and as $\Ind_{P'}^M \sigma$ contains a unique cuspidal
Jordan--H\"older constituent, we see that $\pi'$ is the unique cuspidal
Jordan--H\"older constituent
of $\Ind_{P'}^M \sigma$.  Thus, if $\pi$ and $\pi'$ have distinct supercuspidal
support,
the result of the previous paragraph shows that
$\Ext^i(\pi,\pi') = \Ext^i(\pi,\Ind_{P'}^M \sigma)$.
But by Frobenius reciprocity, $\Ext^i(\pi,\Ind_{P'}^M \sigma) =
\Ext^i(\Res_M^{P'} \pi, \sigma) = 0$.

We have thus established the result whenever $\pi$ is cuspidal.  By duality, it follows
that $\Ext^i(\pi,\pi') = 0$ whenever $\pi'$ is cuspidal and $\scs(\pi) \neq \scs(\pi')$.
Now assume $\pi$ and $\pi'$ are arbitrary and $\scs(\pi) \neq \scs(\pi')$.  Choose
a parabolic subgroup $P' = M'U'$ of $M$, and a cuspidal representation $\sigma$ of $M'$
such that $\pi'$ is isomorphic to a submodule of $\Ind_{P'}^M \sigma$.  We have an
exact sequence:
$$0 \rightarrow \pi' \rightarrow \Ind_{P'}^M \sigma \rightarrow C \rightarrow 0,$$
where $C$ is the cokernel of the inclusion of $\pi'$ in $\Ind_{P'}^M \sigma$.
By Frobenius reciprocity $\Ext^i(\pi, \Ind_{P'}^M \sigma) = \Ext^i(\Res_M^{P'} \pi, \sigma)$,
and the latter vanishes because if $\pi$ and $\pi'$ have different supercuspidal support,
then so do $\Res_M^{P'} \pi$ and $\sigma$, and $\sigma$ is cuspidal.
\end{proof}

\begin{cor}
\label{cor:AIG supercuspidal support}
If $V$ is an essentially AIG representation of $G$ over
$\overline{K}$, then all the Jordan--H\"older constituents of $V$
have the same supercuspidal support.
\end{cor}
\begin{proof}
Suppose otherwise.  As $V$ is the sum of its finite length submodules,
there is then a finite length submodule $W$ of $V$ that is minimal
among submodules of $V$ that have a Jordan--H\"older constituent
with supercuspidal support different from that of $\soc(V)$.  
Let $W'$ be the kernel of the map $W \rightarrow \cosoc(W)$.  The minimality
of $W$ implies that $\cosoc(W)$ is irreducible and that every
Jordan--H\"older constituent of $W'$ has the same supercuspidal support as $\soc(V)$.
Thus $\Ext^i(W',\cosoc(W))$ vanishes for all $i$, by the preceding theorem;
in particular $\cosoc(W)$ is
a direct summand of $W$.  This is impossible, since Lemma~\ref{lem:AIG homs} implies
that any essentially AIG representation is indecomposable.
\end{proof}

\begin{cor}
\label{cor:admissible AIG finiteness}
Let $V$ be an essentially AIG representation of $G$ over $K$.  If $V$ is
admissible, then $V$ has finite length.
\end{cor}
\begin{proof}
By Corollary~\ref{cor:AIG supercuspidal support} there are only finitely many isomorphism classes of
Jordan--H\"older constituents of $V$, and one can bound the number of times
any given Jordan--H\"older constituent appears in terms of the dimension
of the $U$-invariants of $V$ for a sufficiently small compact open subgroup
$U$ of $G$.
\end{proof}

Corollary~\ref{cor:AIG supercuspidal support} has additional finiteness
implications for essentially AIG representations.
More precisely, for a smooth representation $V$
of $G$, define $\soc_c(V)$ inductively by setting $\soc_1(V) = \soc(V)$, and defining
$\soc_c(V)$ to be the preimage of $\soc(V/\soc_{c-1}(V))$ under the surjection
$$V \rightarrow V/\soc_{c-1}(V).$$  We then have:

\begin{theorem} \label{thm:socn}
Let $V$ be an essentially AIG representation of $G$ over $\overline{K}$.  
Then $\soc_c(V)$ has finite length for all $c$.
\end{theorem}
\begin{proof}
By induction it suffices to show that $\soc_c(V)/\soc_{c-1}(V)$ has finite length
for all $c \geq 2$.  The space $\soc_c(V)/\soc_{c-1}(V)$ is semisimple, and every irreducible
summand of $\soc_c(V)/\soc_{c-1}(V)$ is an irreducible non-generic representation of $G$
with the same supercuspidal support as $\soc(V)$.  There are finitely many isomorphism classes
of such representations.  It thus suffices to show, for every irreducible non-generic representation
$\pi$ of $G$ with the same supercuspidal support as $\soc(V)$, that 
$\Hom(\pi,\soc_c(V)/\soc_{c-1}(V))$ is finite dimensional.  We have an exact sequence:
$$0 \rightarrow \soc_{c-1}(V) \rightarrow \soc_c(V) \rightarrow \soc_c(V)/\soc_{c-1}(V) \rightarrow 0.$$
As the socle of $\soc_c(V)$ is generic, we have $\Hom(\pi,\soc_c(V)) = 0$.  We thus obtain
an injection:
$$\Hom(\pi,\soc_c(V)/\soc_{c-1}(V)) \rightarrow \Ext^1(\pi,\soc_{c-1}(V)),$$
and as $\soc_{c-1}(V)$ has finite length by the induction hypothesis $\Ext^1(\pi,\soc_{c-1}(V))$ is
finite dimensional.
\end{proof}

\begin{cor} \label{cor:AIG bounded}
Let $V$ be an essentially AIG representation of $G$, let $c$ be a positive integer,
and let $V_i$
be an arbitrary collection of submodules of $V$ of length less than or equal to $c$.
Then the sum of the $V_i$ has finite length.
\end{cor}
\begin{proof}
Each $V_i$ is contained in $\soc_c(V)$, so their sum is as well.  The result thus follows
immediately from the theorem above.
\end{proof}

We close this subsection with the following result treating essentially AIG representations
in the case $n = 2$.

\begin{prop}
\label{prop:n = 2 AIG}
Any essentially AIG representation over $\GL_2(E)$ is of finite length.
\end{prop}
\begin{proof}
Let $V$ be an essentially AIG representation over $\GL_2(E)$; we must show that
$V$ has finite length.  Clearly we may check this after making an extension
of scalars, and so without loss of generality we may and do assume that $K  =
\overline{K}$.

If $V/\soc(V)$ is trivial then $V = \soc(V)$ is trivial, and we are done.
Thus we assume from now on that $V/\soc(V)$ is non-trivial.
The quotient $V/\soc(V)$ contains no generic constituent,
hence its Jordan--H\"older factors
are all one-dimensional, and so each is of the form $\chi\circ \det$ for some character
$\chi$.  Moreover, if there exist Jordan--H\"older factors of $V/\soc(V)$
isomorphic to $\chi \circ \det$ and $\chi' \circ \det$, then Corollary~\ref{cor:AIG supercuspidal support} 
implies $\chi \circ \det$ and $\chi' \circ \det$ have the same supercuspidal support.
From this it is easy to see that $\chi^2 = (\chi')^2$.
Replacing $V$ by an appropriate twist, we may thus assume that the center $E^{\times}$
of $V$ acts trivially on each Jordan--H\"older factor of $V/\soc(V)$.

Since $V/\soc(V)$ is the sum of its finite length subrepresentations (as $V$ is;
this is one of the conditions of being essentially AIG), and since each of
its Jordan--H\"older factors is one-dimensional, we see that the action
of $\GL_2(E)$ on $V/\soc(V)$ factors through $\det$, and in this way
we regard $V/\soc(V)$ as a representation of $E^{\times}$.
Since $V$ admits a central character, by Lemma~\ref{lem:AIG central},
and since the centre acts trivially on each Jordan--H\"older factor of $V/\soc(V)$,
we see that the centre must act trivially on $V$. (This is where we use
the non-triviality of $V/\soc(V)$.)  Thus $V/\soc(V)$ is in fact
a representation of the group $E^{\times}/(E^{\times})^2$.   
Theorem~\ref{thm:socn} shows that the socle of $V/\soc(V)$ is finite length,
and it follows from the fact that $E^{\times}/(E^{\times})^2$ is finite
that $V/\soc(V)$ itself is of finite length, as required.
\end{proof}

\subsection{Invariant lattices}
\label{subsec:AIG lattice}
We now prove some results about the reduction of finite length
essentially AIG representations of $\GL_n(E)$.  Let $\O$ be a complete discrete
valuation ring, with field of fractions $\CK$ and residue field $K$ of
characteristic different from $\ell$.  Fix a uniformizer $\unif$ of $\O$.

If $V$ is an integral admissible smooth
representation of $\GL_n(E)$ of finite length, then by
the ``Brauer-Nesbitt Theorem'' of
(\cite[Ch.~II.5.11]{Vig2}), $V$ is a good integral
representation in the sense of Definition~\ref{def:integral},
and hence Lemma~\ref{lem:commens}
shows that if $V^{\circ}$ is a $\GL_n(E)$-invariant
$\O$-lattice in $V$, then $(V^{\circ}/\unif V^{\circ})^{\ss}$
is independent of~$V^{\circ}$; we denote it $\Vbar^{\ss}$.

\begin{theorem}
\label{thm:generic reduction}
If $V$ is an essentially AIG admissible smooth representation
of $\GL_n(E)$ which is integral and of finite length, then
$\Vbar^{\ss}$ contains a unique irreducible generic summand.
\end{theorem}
\begin{proof}
Fix an invariant $\O$-lattice $V^{\circ}$ in $V$.  Then
$(V^{\circ})^{(n)}$ is a finitely generated $\O$-submodule
of $V^{(n)}$, and the latter is a one-dimensional $\CK$-vector
space.  Thus $(V^{\circ})^{(n)}$ is free of rank one over $\O$.
As the derivative commutes with tensor products, it follows
that $(\Vbar^{\ss})^{(n)}$ is a one-dimensional $K$-vector space;
the result follows.
\end{proof} 


\begin{prop}
\label{prop:lattice}
If $V$ is an essentially AIG 
admissible smooth representation
of $\GL_n(E)$ over $\CK$ which is integral and of finite length, then $V$ admits a 
$\unif$-adically separated $\GL_n(E)$-invariant lattice
$V^{\circ}$ which is admissible as a $\GL_n(E)$-representation,
and such that
$\Vbar^{\circ} := V^{\circ}/\unif V^{\circ}$
is essentially AIG.  Moreover, $V^{\circ}$ is unique up to homothety.
\end{prop}
\begin{proof}
We apply Lemma~\ref{lem:lattice} to $V$, 
taking $\mathcal T$ to consist of all the nongeneric
Jordan--H\"older factors.  
This yields an $\O$-lattice $V^{\circ}$, 
such that $\Vbar^{\circ}$ contains no nongeneric subrepresentations.
As $\Vbar^{\ss}$ has only one irreducible generic submodule,
this submodule is the socle of $\Vbar^{\ss}$, and
$\bigl(\Vbar^{\ss}/\soc(\Vbar^{\ss})\bigr)^{(n)} = 0$.
If $H$ is any open subgroup of $G$, then $(V^{\circ})^H$ is
$\unif$-adically separated, and its $E$-span coincides with~$V^H$,
which is finite dimensional, since $V$ is admissible.
It follows that $(V^{\circ})^H$ is finitely generated over $\mathcal O$,
and so $V^{\circ}$ is an admissible smooth representation
of $\GL_n(E)$.  Thus $\Vbar^{\circ}$ is admissible smooth, and therefore
essentially AIG.

Suppose now that
$V^{\diamond}$ is a second lattice
in $V$ satisfying the conditions of
the corollary.
Scaling it appropriately, we may assume that
$V^{\diamond} \subset V^{\circ},$
but that the induced map
$\Vbar^{\diamond} \rightarrow \Vbar^{\circ}$
is non-zero.
Since both source and target are essentially AIG,
this map is necessarily injective by Lemma~\ref{lem:AIG basic}, and hence (since
source and target are of the same length) an isomorphism.
\end{proof}

%

\section{The local Langlands correspondence in characteristic zero}
\label{sec:LL zero}

Let $F$ be an algebraically closed field of characteristic zero.
The local Langlands correspondence for $\GL_n(E)$~\cite{HT}
establishes a certain bijection between the set of isomorphism classes
of irreducible admissible smooth representations of $\GL_n(E)$ on
$F$-vector spaces, and the set of isomorphism classes of
$n$-dimensional Frobenius semisimple Weil--Deligne representations over
$F$ 
(as defined in \cite[\S 8]{De} or \cite[\S 4]{Ta}).

In fact there are various choices of correspondence, depending on the
desired normalization.  The so-called unitary correspondence
is uniquely determined by the requirement that the local
$L$- and $\varepsilon$-factors attached to a pair of corresponding
isomorphism classes should coincide.  On the other hand, this
correspondence depends on the choice of a square root of $\ell$ in $F$,
and (because of this) is not compatible in general with change
of coefficients (although a suitably chosen twist will be; we refer the
reader to \ref{subsec:BS} for details.)


However, even if we normalize the local Langlands correspondence to be
compatible with change of coefficients, the correspondence
as usually defined is not suitable for the applications we have in mind.
In particular, the usual local Langlands correspondence fails to be
compatible with specialization.  More precisely, let $\O$ be a complete
discrete valuation ring containing $\Q_p$, with field of fractions $\CK$ 
and residue field $K$, and let $\rho: G_E \rightarrow \GL_n(\O)$ be
a continuous Galois representation.  Then the local Langlands correspondence
associates admissible representations $\pi$ and $\pibar$ to the Weil--Deligne
representations induced by $\rho \otimes_\O \CK$ and $\rho \otimes_\O K$,
but there need not be a close relationship between $\pibar$ and
the reduction of $\pi$.  (For example, $\pibar$ could be a character even
if $\pi$ is infinite-dimensional.)

We therefore work with a modification of the usual local Langlands
correspondence, which we describe fully in~\ref{subsec:BS}.  We denote this
correspondence by $\rho \mapsto \LL(\rho)$, where $\rho$ is a continuous
$n$-dimensional representation of $G_E$ over an extension $K$ of $\Q_p$.
The correspondence $\rho \mapsto \LL(\rho)$
is essentially the generic local Langlands correspondence introduced by
Breuil and Schneider in~\cite{BS}.  Unlike more standard formulations of
local Langlands, the representation $\LL(\rho)$ of $\GL_n(E)$ will in general
be reducible (in fact, it will be an essentially AIG representation of $\GL_n(E)$).
The map $\rho \mapsto \LL(\rho)$ will not be a bijection in any meaningful
sense but simply a map from isomorphism classes of $n$-dimensional representations
of $G_E$ over $K$ to indecomposable admissible representations of $\GL_n(E)$ over $K$.
The advantage of this choice is that the map $\rho \mapsto \LL(\rho)$ will be
compatible with change of coefficients (in the sense that $\LL(\rho \otimes_K K^{\prime})$
will be isomorphic to $\LL(\rho) \otimes_K K^{\prime}$ for an extension $K^{\prime}$ of~$K$,)
and also compatible with specialization (in the sense of Theorem~\ref{thm:specialization}
below.)

\subsection{Galois representations and Weil--Deligne representations}

In order to give a precise description of the map $\rho \mapsto \LL(\rho)$,
we first recall some basic facts about Frobenius-semisimple Weil--Deligne
representations.  Recall that a Weil--Deligne representation over 
a field $K$ containing $\Q_p$ is a pair $(\rho^{\prime},N)$,
where $\rho^{\prime}: W_E \rightarrow \GL_n(K)$ is a smooth representation
of $W_E$ with coefficients in $K$ and $N$ is a nilpotent endomorphism of
$K^n$ satisfying $\rho^{\prime}(w)N\rho^{\prime}(w)^{-1} = |w|N$.  The
representation $(\rho^{\prime},N)$ is called Frobenius-semisimple if
$\rho^{\prime}$ is absolutely semisimple.

We first consider absolutely irreducible representations
$\rho': W_E \rightarrow \GL_n(K)$.  Let $I_E$ be the inertia subgroup of $E$.
Then $\rho'(I_E)$ is a finite group,
and so all of its irreducible representations are defined over a finite extension $K_0$ of
${\Qbar_p}$.  After replacing $K$ with an algebraic extension we may assume
$K$ contains a subfield isomorphic to $K_0$.
Then the restriction of $\rho'$ to $I_E$ splits
as a direct sum of absolutely irreducible representations $\tau_i$ of
$I_E$ over $K_0$.  

Let $\Phi$ be a Frobenius element of $W_E$, and let
$V$ be an $I_E$-stable subspace of $K^n$ isomorphic to $\tau_1 \otimes_{K_0} K$
as an $I_E$-representation.  Then $I_E$ acts on $\Phi V$ by the conjugate
$\tau_1^{\Phi}$ of $\tau_1$.  In fact, we have:

\begin{lemma} \label{lem:clifford}
Let $r$ be the order of the orbit of $\tau_1$ under the action of
$\Phi$ on the set of isomorphism classes of
absolutely irreducible representations of $I_E$ over $K_0$.
Then we have a direct sum decomposition:
$$\rho'|_{I_E} = \bigoplus_{i=0}^{r-1} \tau_1^{\Phi^i} \otimes_{K_0} K$$
and the action of $\Phi$ on this decomposition permutes the summands.
\end{lemma}
\begin{proof}
As $I_E$ is normal in $W_E$, this is a standard result in Clifford theory.
\end{proof}

In particular, the vector space $\Hom_{K[I_E]}(\tau_1 \otimes_{K_0} K,\rho'|_{I_E})$
is one-dimensional.  If we fix an isomorphism $\tau_1 \iso \tau_1^{\Phi^r}$,
then we get an endomorphism $\Psi$ of this vector space via:
$$
\Hom_{K[I_E]}(\tau_1 \otimes_{K_0} K, \rho') \stackrel{\Phi^r}{\rightarrow}
\Hom_{K[I_E]}(\tau_1^{\Phi^r} \otimes_{K_0} K, \rho') \iso
\Hom_{K[I_E]}(\tau_1 \otimes_{K_0} K, \rho').
$$
The action of $\Psi$ is given by a scalar $\lambda$ in $K^{\times}$.

\begin{lemma} \label{lem:lambda}
For any $\lambda \in K^{\times}$
there is a unique absolutely irreducible representation $\rho'$
over $K$ {\em (}up to isomorphism{\em )}
such that $\rho'|_{I_E}$ contains $\tau_1 \otimes_{K_0} K$
and $\Psi$ acts on $\Hom_{K[I_E]}(\tau_1 \otimes_{K_0} K, \rho')$
via $\lambda$.
\end{lemma}
\begin{proof}
If $r = 1$, then the restriction of $\rho'$ to $I_E$ is
given by $\tau_1 \otimes_{K_0} K$, and so to determine $\rho'$ it suffices to
give an action of $\Phi$ on the representation space of $\tau_1$, compatible with
the action of $I_E$.  This amounts to giving an isomorphism
$\tau_1 \otimes_{K_0} K \iso (\tau_1 \otimes_{K_0} K)^{\Phi}$.  As we have already fixed 
an isomorphism $\tau_1 \iso \tau_1^{\Phi}$, such an isomorphism is determined by
$\lambda$.

If $r > 1$, let $E'$ be the unramified
extension of $E$ of degree $r$.  The restriction of $\rho'$ to $W_{E'}$ breaks
up as a sum of irreducible representations $\rho'_0,\dots,\rho'_{r-1}$ such
that the restriction of $\rho'_i$ to $I_E$ is isomorphic to $\tau_1^{\Phi^i}$.
Thus $\rho'_i$ is determined by $\lambda$ and $\tau_1$, and $\rho'$ is isomorphic
to $\Ind_{W_{E'}}^{W_E} \rho'_0$ by Frobenius reciprocity.
\end{proof}

\begin{lemma}
Let $K$ be a field containing $\Q_p$, and let $\rho'$ be an absolutely
irreducible representation of $W_E$ over $K$.  Then there exists an unramified
character $\chi: W_E \rightarrow \overline{K}^{\times}$ such that
the twist $\rho' \otimes \chi$ is defined over ${\Qbar_p}$.
\end{lemma}
\begin{proof}
The representation $\tau_1$ is defined over a finite extension of $\Q_p$,
so it suffices to show that after a twist we can take the scalar
$\lambda$ to be in ${\Qbar_p}$.  Twisting by an unramified $\chi$ changes
$\lambda$ to $\chi(\Phi)^r\lambda$, so this is clear.
\end{proof}

\begin{df} \label{df:special}
Let $\rho^{\prime}$ be an absolutely irreducible smooth representation
$W_E \rightarrow \GL_n(K)$, and let $d$ be a positive integer.  The
special representation $\Sp_{\rho^{\prime},d}$ is the pair
$$\Sp_{\rho^{\prime},d} = 
(V_0 \oplus \dots \oplus V_{d-1}, N),$$
where $W_E$ acts on $V_i$ by $\abs^i\rho^{\prime}$ and
$N$ maps $V_i$ isomorphically onto $V_{i+1}$ for $0 \leq i \leq d-2$.
\end{df}

The representation $\Sp_{\rho^{\prime},d}$ is well-defined up to isomorphism,
and is an absolutely indecomposable Weil--Deligne representation.  If $K$ is
algebraically closed, then every indecomposable Frobenius-semisimple
Weil--Deligne representation
has the form $\Sp_{\rho^{\prime},d}$ for a unique absolutely irreducible
representation $\rho^{\prime}$ of $W_E$ over $K$.  Combining this with
the previous lemma, we find:

\begin{lemma} \label{lem:twist2}
Let $K$ be a field containing $\Q_p$, and let
$(\rho^{\prime},N)$ be an indecomposable Frobenius-semisimple Weil--Deligne
representation over $K$.  Then there exists a character $\chi: W_E \rightarrow \overline{K}^{\times}$
such that the twist $(\rho^{\prime} \otimes \chi,N)$ is defined over ${\Qbar_p}$.
\end{lemma}

In those situations in which we will apply the local Langlands
correspondence, we will be beginning not with Weil--Deligne representations,
but with Galois representations.  Thus we recall the recipe of Deligne
for associating a Weil--Deligne representation to a continuous
Galois representation,
in a slightly broader context than that in which it is usually considered.

Let $A$ be a complete Noetherian local domain of residue characteristic
$p$ different from $\ell$, maximal ideal $\mathfrak m$, and field of fractions
$\CK$ of characteristic zero.  Let $R$ be any subring of $\CK$ containing $A$ and
$\frac{1}{p}$.  (In most applications, $R$ will equal either~$\CK$,
or else a complete discrete valuation ring $\O$
containing $A$ and contained in $\CK$.)

For any $n \geq 0,$ we say that
a representation $\rho: G_E \rightarrow \GL_n(R)$
is continuous if we can find a finitely generated $A$-submodule $M$
of $R^n$ that is invariant under $\rho(G_E)$,
spans $R^n$ over $R$, 
and such that the induced $G_E$-action on $M$
is $\mathfrak m$-adically continuous.  (Note that if $R=\CK$, and $\CK$
is a finite extension of $\Q_p$, then this coincides with the usual
notion of continuity of a $G_E$-representation.)

As in \cite[(4.2)]{Ta}, we fix a non-zero homomorphism
$t_p:I_E \rightarrow \Q_p$.  (When comparing the present
discussion with that of \cite{Ta}, note that the roles of $\ell$
and $p$ are reversed.)
This homomorphism is uniquely determined up to scaling by an
element of $\Q_p^{\times}$.  
The following result then extends a theorem of Deligne
\cite[\S 8]{De}, \cite[Thm.~(4.2.1)]{Ta} (which treats the case when
the coefficient field is a finite extension of $\Q_p$).

\begin{prop}
\label{prop:WD}
A continuous representation
$\rho: G_E \rightarrow \GL_n(R)$ 
uniquely determines the following data:
\begin{enumerate}
\item
a representation
$\rho': W_E \rightarrow \GL_n(R)$ 
that is continuous
when the target is equipped with its discrete topology;
\item
a nilpotent matrix $N \in \mathrm{M}_n(R)$;
\smallskip
\newline
\smallskip
subject to the following condition:
\item
$\rho(\Phi^r \sigma) = \rho'(\Phi^r \sigma)\exp(t_p(\sigma) N)$
for all $\sigma \in I_E$ and $r \in \Z$.
\end{enumerate}

Furthermore, as a Weil--Deligne representation,
the pair $(\rho',N)$ is independent, up to isomorphism,
of the choice of $t_{p}$ and $\Phi$.
\end{prop}
\begin{proof}
Let $(\rho'_1,N_1)$ and $(\rho'_2,N_2)$ be two Weil--Deligne representations
satisfying the condition of the proposition.  Then there is an open subgroup
of $I_E$ on which both $\rho'_1$ and $\rho'_2$ are trivial; we can thus find an
element $\sigma$ of $I_E$ for which $\rho'_1(\sigma)$ and $\rho'_2(\sigma)$ are
the identity but $t_{p}(\sigma)$ is nonzero.  Then 
$N_1 =  \frac{1}{t_p(\sigma)}\log \rho(\sigma) = N_2$.  The identity
$$\rho(\Phi^r \sigma) = \rho'_i(\Phi^r \sigma)\exp(t_p(\sigma) N_i)$$
then forces $\rho'_1 = \rho'_2$.

It thus suffices to construct a $(\rho',N)$ as above.  Choose a finitely-generated
$A$-submodule $M$ of $R^n$
that is preserved by $\rho$ and spans $R^n$ over $R$.  Then $G_E$ acts via $\rho$
on $M/{\mathfrak m}^{i+1} M$ for all $i$, and these $A$-modules
are discrete with respect to the ${\mathfrak m}$-adic topology.  In particular for each $i$
the subgroup $H_i$ of $I_E$ that acts trivially on $M/{\mathfrak m}^i M$ is a compact open
subgroup of $I_E$.

The group of automorphisms of $M/{\mathfrak m}^{i+1} M$ that reduce to the identity in
$M/{\mathfrak m}^i M$ is an abelian $p$-group for all $i \geq 2$.  Thus the action
of $H_2$ on $M$ factors through the map $t_p: I_E \rightarrow \Q_p$.  Let $\sigma$
be an element of $H_2$; the action of $\sigma$ on $M$ yields an element of $\alpha$ 
of $\End(M)$
that is congruent to the identity modulo ${\mathfrak m}^2$.  The power series $\log(\alpha)$
thus converges in the ${\mathfrak m}$-adic topology on $\End(M)$; set $N = \frac{1}{t_p(\sigma)} \log(\alpha)$.
Then any $\tau \in H_2$ acts on $M$ via $\exp(t_p(\tau) N)$.  It follows that for all
$\tau \in G_E$, $\rho(\tau) N \rho(\tau)^{-1} = |\tau| N$.  In particular $N$ must be
nilpotent.
 
We can then set $\rho'(\Phi^r\sigma) = \rho(\Phi^r\sigma)\exp(t_p(\sigma) N)^{-1}$ for all
$\sigma \in I_E$ and $r \in \Z$; this gives a well-defined $\rho'$ that is
trivial on the compact open subgroup $H_2$ of $W_E$.
\end{proof}

Let $\O$ be a discrete valuation ring containing $A$ and contained in $\CK$, with
residue field $K$
of characteristic zero
and uniformizer $\unif$.
We will be interested in the reduction mod $\unif$ of both Galois representations
and Weil--Deligne representations over~$\O$.  One has:

\begin{lemma} \label{lem:irreducible reduction}
Let $\rho^{\prime}: W_E \rightarrow \GL_n(\O)$ be a representation of $W_E$
over $\O$ such that $\rho^{\prime} \otimes_{\O} \CK$ is absolutely irreducible.
Then $\rhobar^{\prime} := \rho^{\prime} \otimes_{\O} K$ is also absolutely irreducible.
\end{lemma}
\begin{proof}
By Lemma~\ref{lem:clifford}, over a finite extension of $\CK$, the
restriction of $\rho^{\prime}$ to $I_E$ splits as a direct sum
of absolutely irreducible representations $\rho^{\prime}_i$
of $I_E$, each of which factors through a finite quotient if $I_E$
and is defined over ${\Qbar_p}$.
The representations $\rho^{\prime}_i$ are distinct and cyclically
permuted by conjugation by $\Phi$.
As $K$ has characteristic zero, the $\rho^{\prime}_i$ remain irreducible
and distinct after ``reduction mod $\unif$''.

Thus, over a finite extension of $K$, $\rhobar^{\prime}$ splits
as a direct sum of absolutely irreducible representations
$\rhobar^{\prime}_i$ which are distinct and cyclically permuted under
conjugation by $\Phi$.  It is thus clear that $\rhobar^{\prime}$ is
absolutely irreducible.
\end{proof}

It follows that if $\rho^{\prime}: W_E \rightarrow \GL_n(\CK)$ is absolutely
irreducible and contains an $\O$-lattice $L$, then the mod $\unif$ reduction
of $L$ is independent, up to isomorphism, of the lattice $L$.  We denote this
reduction by $\rhobar^{\prime}$.  

The passage from Galois representations to Weil--Deligne representations
commutes with reduction modulo $\unif$:

\begin{lemma} \label{lem:WD reduction}
Let $\rho: G_E \rightarrow \GL_n(\O)$ be a continuous Galois representation,
with mod $\unif$ reduction $\rhobar$, and let $(\rho^{\prime},N)$ and
$(\rhobar^{\prime},\Nbar)$ be the Weil--Deligne representations attached
to $\rho$ and $\rhobar$, respectively.  Then $(\rhobar^{\prime},\Nbar)$ is isomorphic
to $(\rho^{\prime} \otimes_{\O} K, N \otimes_{\O} K)$. 
\end{lemma}
\begin{proof}
This follows immediately from the identities
$$\rho(\Phi^r \sigma) = \rho^{\prime}(\Phi^r \sigma)\exp(t_p(\sigma) N)$$
$$\rhobar(\Phi^r \sigma) = \rhobar^{\prime}(\Phi^r \sigma)\exp(t_p(\sigma) \Nbar)$$
and the fact that the latter identity characterizes $(\rhobar^{\prime},N)$ up to isomorphism.
\end{proof}

Given a Weil--Deligne representation $(\rho',N)$ over $\CK$,
one can associate a natural Frobenius-semisimple representation
$(\rho',N)^{\Fss}$, (the Frobenius-semisimplification of
$(\rho',N)$.)  We recall the definition; see~\cite[8.5]{De}
for details.

The matrix $\rho'(\Phi)$ factors uniquely
as a product $su$, with $s$ and $u$ elements of $\GL_n(\CK)$ that 
are semisimple and unipotent, respectively, and commute with each other.
Moreover, if $\rho'(\Phi)$ lies in $\GL_n(\O)$, then so do $s$ and $u$.
The element $u$ then commutes with $N$, and one defines $(\rho')^{\Fss}$
to be the representation of $W_F$ that satisfies $(\rho')^{\Fss}(\Phi^r \sigma)
= u^{-r}\rho'(\Phi^r\sigma)$.  Then $(\rho')^{\Fss}$ is a semisimple
representation of $W_F$ over $\CK$, and the pair $((\rho')^{\Fss},N)$
is a Frobenius-semisimple Weil--Deligne representation which we write
$(\rho',N)^{\Fss}$.

It will be necessary for us to understand how Frobenius-semisimplification
commutes with reduction modulo $\unif$.  Note that even if
$(\rho',N)$ is a Frobenius-semisimple Weil--Deligne representation over
$\O$, its reduction modulo $\unif$ need not be, as the mod $\unif$
reduction of a semisimple element of $\GL_n(\O)$ need not be semisimple.

\begin{lemma} \label{lem:semisimple reduction}
Let $(\rho',N)$ be a Weil--Deligne representation over $\O$,
and let $(\rhobar',\Nbar)$ be its reduction mod $\unif$.  Then
$(\rho',N)^{\Fss}$ is defined over $\O$.  Moreover, the reduction mod $\unif$
of $(\rho',N)^{\Fss}$ has Frobenius-semisimplification
$(\rhobar,\Nbar)^{\Fss}$.
\end{lemma}
\begin{proof}
If $\rho'(\Phi)$ decomposes as $su$, with $s$ and $u$ as above, then
$s$ and $u$ lie in $\GL_n(\O)$, so $(\rho',N)^{\Fss}$ is defined over $\O$.
Thus $\rhobar'(\Phi) = \overline{s}\overline{u}$, where $\overline{s}$
and $\overline{u}$ are the mod $\unif$ reductions of $s$ and $u$.

The element $\overline{s}$ decomposes uniquely as $\overline{s}'\overline{u}'$,
where $\overline{s}'$ is semisimple and $\overline{u}'$ is unipotent and
commutes with $\overline{s}'$.  As $\overline{u}$ commutes with $\overline{s}$,
$\overline{s}$ also decomposes as a product of $\overline{u}\overline{s}'\overline{u}^{-1}$
with $\overline{u}\overline{u}'\overline{u}^{-1}$; the uniqueness of this
decomposition shows that these two decompositions coincide.  That is,
$\overline{u}$ commutes with $\overline{s}'$ and $\overline{u}'$.

We have $\rhobar'(\Phi) = \overline{s}'\overline{u}'\overline{u}$,
and the unipotent element $\overline{u}'\overline{u}$ commutes with
$\overline{s}'$.  Thus the Frobenius-semisimplification of $(\rhobar',\Nbar)$
sends $\Phi$ to $\overline{s}'$.  On the other hand, the reduction
of $(\rho',N)^{\Fss}$ takes $\Phi$ to $\overline{s}$, which equals
$\overline{s}'\overline{u}'$.  Hence the Frobenius-semisimplification
of the reduction of $(\rho',N)^{\Fss}$ takes $\Phi$ to $\overline{s}$,
and therefore coincides with $(\rhobar',\Nbar')^{\Fss}$.
\end{proof}

\subsection{The generic local Langlands correspondence of Breuil and Schneider}
\label{subsec:BS}

We are now in a position to describe the ``generic local Langlands correspondence''
of Breuil and Schneider \cite[pp. 162--164]{BS}.  This is a map $(\rho',N) \mapsto \LL(\rho',N)$
from Frobenius-semisimple Weil--Deligne representations over a finite extension $K$ of
$\Q_p$ to indecomposable admissible representations of $\GL_n(E)$ over $K$.
Fix a choice of $\ell^{\frac{1}{2}}$ in ${\Qbar_p}$ (and thus a choice of square root of the
character $\abs \circ \det$ of $\GL_n(E)$, as well as a unitary local Langlands correspondence
for representations over ${\Qbar_p}$).  With this choice,
the properties of this correspondence can be summarized as follows (c.f.\ \cite[4.2]{BS}):

\begin{enumerate}
\item For any character $\chi: W_F \rightarrow \Q_p^{\times}$, one has
$\LL(\rho' \otimes \chi, N) = \LL(\rho', N) \otimes \chi$.
\item If $K'$ is a finite extension of $K$, then $\LL(\rho' \otimes_K K', N) = \LL(\rho', N) \otimes_K K'$.
\item If $(\rho',N)$ is a direct sum of representations of the form $\Sp_{\rho_i', n_i}$ over
${\Qbar_p}$, then $\LL(\rho', N)$ is defined by the parabolic induction:
$$\LL(\rho',N) = (\abs \circ \det)^{-\frac{n-1}{2}} \Ind_Q^{\GL_n(E)} \St_{\pi_1, n_1} \otimes \dots \otimes \St_{\pi_r, n_r},$$
where $\pi_i$ corresponds to $\rho_i$ under the unitary local Langlands
correspondence, $\St_{\pi_i, n_i}$ is the generalized Steinberg representation
(and thus corresponds to $\Sp_{\rho_i, n_i}$ under unitary local Langlands,)
and $Q$ is the upper triangular parabolic subgroup of $\GL_n(E)$ whose Levi subgroup
is block diagonal with block sizes $(n_1 \dim \rho_1', \dots, n_r \dim \rho_r')$.  The
symbol $\Ind_Q^{\GL_n(E)}$ denotes {\em normalized} parabolic induction.  The representations
$\St_{\rho_i,n_i}$ are ordered so that the condition of~\cite[Def.~1.2.4]{Ku} holds. 
(As long
as this condition is satisfied, the resulting parabolic induction is independent, up to isomorphism, of the
precise choice of ordering, as well as of the choice the square root of $\ell$ needed to
define $(\abs \circ \det)^{\frac{1}{2}}$.)
\end{enumerate}

These properties uniquely characterize the generic local Langlands correspondence.  We
will need a slight extension of this correspondence to the case of coefficients in
an arbitrary field extension $K$ of $\Q_p$.  Let $(\rho',N)$ be a Frobenius-semisimple
Weil--Deligne representation over $K$, and suppose that it decomposes over $\overline{K}$
as a direct sum of representations of the form $\Sp_{\rho_i', n_i}$.  Then by 
Lemma~\ref{lem:twist2}, there exist characters $\chi_i: W_E \rightarrow \overline{K}^{\times}$
such that $\rho_i' \otimes \chi_i$ is defined over ${\Qbar_p}$.  For such representations
the unitary local Langlands correspondence is defined, and we can take
$\pi_i$ to be the representation over $K$ such that $\pi_i \otimes \chi_i$ corresponds to
$\rho'_i \otimes \chi_i$ via the unitary local Langlands correspondence over ${\Qbar_p}$.
$$\St_{\pi_i, n_i} = \St_{\pi'_i \otimes_{\chi_i}, n_i} \otimes (\chi_i^{-1} \circ \det).$$
$$\LL(\rho',N) = (\abs \circ \det)^{-\frac{n-1}{2}} \Ind_Q^{\GL_n(E)} 
\St_{\pi_1, n_1} \otimes \dots \otimes \St_{\pi_r, n_r},$$
where the $\St_{\pi_i,n_i}$ are ordered as before.  {\em A priori}, this is a representation of $\GL_n(E)$
over $\overline{K}$, but the argument of~\cite[Lem.~4.2]{BS} shows that $\LL(\rho',N)$ is defined
over $K$ itself.  Moreover, $\LL(\rho',N)$ is independent of the choices of $\chi_i$.

As was the case over finite extension of $\Q_p$, the map $(\rho', N) \mapsto \LL(\rho', N)$
is compatible with twists, and also with arbitrary field extensions.

We extend this definition to a map from representations of $G_E$ to admissible representations of
$\GL_n(E)$ as follows:

\begin{df} \label{def:LL galois}
Let $\rho$ be a continuous $n$-dimensional representation of $G_E$ over $K$, and
let $(\rho',N)$ be the corresponding Weil--Deligne representation.  We define
$\LL(\rho)$ to be $\LL((\rho',N)^{\Fss})$.
\end{df}

\subsection{Segments and the Zelevinski classification}
Our next goal is to establish key properties of the generic local Langlands 
correspondence (in particular, we will show that $\LL(\rho',N)$ is essentially AIG).

Following~\cite{Ze1}, we define a segment to be a set of supercuspidal representations of the
form:
$[\pi, (\abs \circ \det)\pi, \dots, (\abs \circ \det)^{r-1}\pi]$,
where $\pi$ is an irreducible supercuspidal representation of $\GL_n(E)$ over $\overline{K}$.
We think of the segment $\Delta$ given by $[\pi, (\abs \circ \det)\pi, \dots, (\abs \circ \det)^{r-1}\pi]$ as corresponding
to the generalized Steinberg representation $\St_{\pi,r}$; this gives a bijection between
segments and generalized Steinberg representations.  If $\St_{\pi,r}$ corresponds to a segment $\Delta$,
we will often write $\St_{\Delta}$ for $\St_{\pi,r}$.  Similarly, we will write
$\Sp_{\Delta}$ for the indecomposable Weil--Deligne representation $\Sp_{\rho,r}$,
where $\rho$ is the irreducible Weil--Deligne representation corresponding to $\pi$
under the unitary local Langlands correspondence.

Two segments $\Delta,\Delta^{\prime}$ are said to be {\em linked} if neither contains the
other, and if $\Delta \cup \Delta^{\prime}$ is a segment.  The segment $\Delta$
{\em precedes} $\Delta^{\prime}$ if $\Delta$ and $\Delta^{\prime}$ are linked 
and $\Delta^{\prime}$ has the form $[(\abs \circ \det)\pi, \dots, (\abs \circ \det)^{r-1}\pi]$ for some
$\pi$ in $\Delta$.

We consider the following condition on a sequence $\CS$ of segments $\Delta_i$:

\begin{condition}
\label{cond:Zel}
For all $i < j,$ the segment $\Delta_i$ does not precede the segment $\Delta_j$.
\end{condition}

It is clear that any unordered collection of segments can be given an ordering
that satisfies Condition~\ref{cond:Zel}.
If $\CS$ is an unordered collection of segments,
then we let $\pi(\CS)$ denote the normalized
parabolic induction
$$\Ind_Q^{\GL_n(E)} \St_{\Delta_1} \otimes \dots \otimes \St_{\Delta_n},$$
where $\Delta_1, \dots, \Delta_n$ are the segments in $\CS$, taken with multiplicities
and ordered so that Condition~\ref{cond:Zel} holds.
By~\cite[Prop.~6.4]{Ze1}, the representation $\pi(\CS)$ does not depend, up to isomorphism, 
on the order of the collection of
segments in $\CS$ (as long as Condition~\ref{cond:Zel} holds).
Note that if $(\rho',N)$ is an $n$-dimensional Frobenius-semisimple 
Weil--Deligne representation that decomposes as the direct sum of $\Sp_{\Delta_i}$ for $\Delta_i \in \CS$,
then $\LL(\rho',N)$ is isomorphic to $(\abs \circ \det)^{-\frac{n-1}{2}}\pi(\CS)$. 
By~\cite[1.2.5]{Ku}, 
$\pi(\CS)$ admits a unique irreducible quotient $Q(\CS)$, and $Q(\CS)$
is the irreducible representation corresponding to $(\rho',N)$ under the unitary
local Langlands correspondence.

\begin{prop} 
\label{prop:AIG}
If $\CS$ is an unordered collection of segments,
then every irreducible submodule of $\pi(\CS)$ is generic.
\end{prop}
\begin{proof}
We will prove a stronger statement --- namely, that every irreducible $P_n$-submodule
of the restriction $\pi(\CS)|_{P_n}$ is generic.  (In other words, $\pi(\CS)$
embeds in its Kirillov model.)  Over the complex numbers this is a result of
Jacquet-Shalika~\cite{J-S}.  Their argument does not seem to adapt easily to
other fields of characteristic zero.  One could reduce this proposition to
their result by choosing an isomorphism of $\overline{K}$ with $\C$; we instead
give an algebraic argument over $\overline{K}$ that is an adaptation of the argument
of~\cite[4.15]{BZ}.  Their argument necessarily uses $\Psi^{\pm}$ and $\Phi^{\pm}$ functors
that are normalized differently from ours, to avoid unpleasant combinatorial issues.
Therefore, {\em for the purposes of this proof only} we take the functors $\Psi^{\pm}$
and $\Phi^{\pm}$ to be normalized as in~\cite{BZ}, rather than as in section~\ref{subsec:kirillov}.

Let $\CS$ be the collection $(\Delta_1, \dots, \Delta_n)$, where $\Delta_i$ does not
precede $\Delta_j$ for any $j > i$.  We can assume without loss of generality
that the $\Delta_i$ are ordered so that if 
$\Delta_i = [\pi_i, (\abs \circ \det)\pi_i, \dots (\abs \circ \det)^{r_i-1}\pi_i]$,
where $\pi_i$ is a supercuspidal representation of $\GL_{n_i}(E)$,
then $(\abs \circ \det)^{r_i}\pi_i$ is not contained in any segment $\Delta_j$ with $j > i$;
clearly for such an ordering $\Delta_i$ never precedes a $\Delta_j$ with $j > i$.
We proceed by induction on the sum of the lengths of the segments $\Delta_i$.  Note
that the result is clear for a single segment, as $\St_{\Delta_i}$ is absolutely
irreducible and generic.  Let $\CS'$ be the collection 
$(\Delta_2, \dots, \Delta_n)$; by the induction hypothesis every irreducible
submodule of $\pi(\CS')$ is generic.

Suppose we have an irreducible, non-generic submodule $\omega$
of $\pi(\CS)|_{P_n}$.  We have $\pi(\CS) = \St_{\Delta_1} \times \pi(\CS')$,
where ``$\times$'' is the product defined in~\cite[4.12]{BZ}.  By~\cite[4.13a]{BZ}, 
we have an exact sequence:
$$0 \rightarrow 
(\St_{\Delta_1})|_{P_{r_1n_1}} \times \pi(\CS') \rightarrow
\pi(\CS)|_{P_n} \rightarrow 
\St_{\Delta_1} \times \pi(\CS')|_{P_{n - r_1n_1}}
\rightarrow 0.$$
In particular, $\omega$ is a submodule of one of
$(\St_{\Delta_1})|_{P_{r_1n_1}} \times \pi(\CS')$ 
or 
$\St_{\Delta_1} \times \pi(\CS')|_{P_{n - r_1n_1}}.$

Suppose first that $\omega$ is contained in
$(\St_{\Delta_1})|_{P_{r_1n_1}} \times \pi(\CS')$. 
By~\cite[9.6]{Ze1}, $\St_{\Delta_i}^{(k)}$ is zero if
$k$ is not divisible by $n_i$, whereas $\St_{\Delta_i}^{(kn_i)}$
is $\St_{\Delta_i^{(k)}}$, where $\Delta_i^{(k)}$
is the segment $[(\abs \circ \det)^k\pi_i, \dots, (\abs \circ \det)^{r-1}\pi_i]$.
It follows by~\cite[4.13c]{BZ}, that, for $i < n_1$,
$$(\Phi^-)^i((\Phi^-)^{kn_1}(\St_{\Delta_1}|_{P_{r_1n_1}}) \times \pi(\CS'))
= (\Phi^-)^{kn_1 + i}(\St_{\Delta_1}|_{P_{r_1n_1}}) \times \pi(\CS'),$$
so that for such $i$, 
$((\Phi^-)^{kn_1}(\St_{\Delta_1}|_{P_{r_1n_1}}) \times \pi(\CS'))^{(i)} = 0$.
For $i = n_1$, \cite[4.13c]{BZ} shows that the representation
$(\Phi^-)^{(k+1)n_1}(\St_{\Delta_1}|_{P_{r_1n_1}}) \times \pi(\CS')$ is instead
a proper submodule of
$(\Phi^-)^{n_1}((\Phi^-)^{kn_1}(\St_{\Delta_1}|_{P_{r_1n_1}}) \times \pi(\CS'))$;
the quotient of the latter by the former is isomorphic to
$\St_{\Delta_1^{(k+1)}} \times \pi(\CS')|_{P_{n-r_1n_1}}$.

Since $\omega$ is contained in $(\St_{\Delta_1})|_{P_{r_1n_1}} \times \pi(\CS')$,
we have $\omega^{(i)} = 0$ for $i < n_p$.  As $\omega$ has at least one nonzero
derivative it follows that $(\Phi^-)^{n_1-1}\omega$ is nonzero.  On the other hand,
we have $(\Phi^-)^{n_1-1}\omega \subset (\Phi^-)^{n_1-1}((\St_{\Delta_1})|_{P_{r_1n_1}} \times \pi(\CS'))$;
by~\cite[4.13d]{BZ} it follows that $(\Phi^-)^{n_1}\omega$ is nonzero.

Then $(\Phi^-)^{n_1}\omega$ is a non-generic submodule of
$(\Phi^-)^{n_1}((\St_{\Delta_1})|_{P_{r_1n_1}} \times \pi(\CS'))$, and
is therefore a non-generic submodule of either
$(\Phi^-)^{n_1}((\St_{\Delta_1})|_{P_{r_1n_1}}) \times \pi(\CS')$,
or $\St_{\Delta_1^{(1)}} \times \pi(\CS')|_{P_{n-r_1n_1}}$.
It is easy to rule out the latter case: by the inductive hypothesis
$\pi(\CS')|_{P_{n-r_1n_1}}$ has no non-generic submodules; by
\cite[5.3]{Ze1} neither does $\St_{\Delta_1^{(1)}} \times \pi(\CS')|_{P_{n-r_1n_1}}$.

Thus $(\Phi^-)^{n_1}\omega$ is a non-generic submodule of 
$(\Phi^-)^{n_1}((\St_{\Delta_1})|_{P_{r_1n_1}}) \times \pi(\CS')$.
In particular $((\Phi^-)^{n_1}\omega)^{(i)} = 0$ for $i < n_1$;
it follows as above that $(\Phi^-)^{2n_1 - 1}\omega$ is nonzero,
and by~\cite[4.13d]{BZ}, that $(\Phi^-)^{2n_1}\omega$ is nonzero.
Then $(\Phi^-)^{2n_1}\omega$ is a nonzero non-generic submodule of
$(\Phi^-)^{n_1}((\Phi^-)^{n_1}((\St_{\Delta_1})|_{P_{r_1n_1}}) \times \pi(\CS'))$,
and hence (with another use of the inductive hypothesis and~\cite[5.3]{Ze1}),
is a nonzero non-generic submodule of $(\Phi^-)^{2n_1}((\St_{\Delta_1})|_{P_{r_1n_1}}) \times \pi(\CS').$

Proceeding in this fashion we find that $(\Phi^-)^{kn_1}\omega$ is a nonzero non-generic 
submodule of $(\Phi^-)^{kn_1}((\St_{\Delta_1})|_{P_{r_1n_1}}) \times \pi(\CS')$ for all $k$,
which is impossible since the latter vanishes for large $k$.

We have thus ruled out the possibility that $\omega$ is contained in
$(\St_{\Delta_1})|_{P_{r_1n_1}} \times \pi(\CS')$.  The other alternative is
that $\omega$ is contained in
$\St_{\Delta_1} \times \pi(\CS')|_{P_{n-r_1n_1}}$.  Suppose this were the case,
and let $k$ be the largest integer such that $\omega^{(k)}$ is nonzero.
Then $\omega^{(k)}$ is nonzero and embeds in the $k$-the derivative of
$\St_{\Delta_1} \times \pi(\CS')|_{P_{n-r_1n_1}}$, which is
$\St_{\Delta_1} \times \pi(\CS')^{(k)}$.  It follows that
the supercuspidal support of $\omega^{(k)}$ contains that of $\St_{\Delta_1}$;
in particular it contains $(\abs \circ \det)^{r-1}\pi_1$.  By~\cite[4.7b]{BZ}, it follows
that $(\abs \circ \det)^r\pi_1$ is contained in the supercuspidal support
of $\pi(\CS)$; this is impossible by our choice of ordering on the $\Delta_i$.
\end{proof}

\begin{cor} \label{cor:essential}
If $\pi$ is an admissible representation of $\GL_n$ over a field $K$ of
characteristic zero, such that
$\pi \otimes_K \overline{K}$ is isomorphic to $\pi(\CS)$ for some $\CS$ satisfying
Condition~{\em \ref{cond:Zel}},
then $\pi$ is essentially AIG.
In particular, every representation $\LL(\rho',N)$ over a field $K$ of characteristic zero
is essentially AIG.
\end{cor}
\begin{proof}
It suffices to show that $\pi \otimes_K \overline{K}$ is essentially AIG,
as then the socle of $\pi$ must be absolutely irreducible and generic,
and $\pi$ must contain no other irreducible generic subquotients.
But $\pi \otimes_K \overline{K}$ has the form $\pi(\CS)$ for some $\CS$, 
so the previous proposition shows that the socle of $\pi(\CS)$ is a direct sum
of irreducible generic representations.  It thus suffices to show that $\pi(\CS)^{(n)}$
is one-dimensional; this follows from the fact that $\St_{\Delta_i}$ is irreducible
and generic, together with Theorem \ref{thm:kirillov} and Proposition \ref{prop:leibnitz}.
\end{proof}

If $\CS$ and $\CS'$ are two unordered collections of segments, we say $\CS'$ arises from $\CS$
by an elementary operation if $\CS'$ is obtained from $\CS$
by replacing a pair of linked segments $\Delta, \Delta^{\prime}$ in $\CS$
with the pair $\Delta \cup \Delta^{\prime},\Delta \cap \Delta^{\prime}$.  
We say that $\CS' \preceq \CS$ if $\CS'$ can be obtained from $\CS$ by a
sequence of elementary operations.  This partial order contains information
about the Jordan--H\"older constituents of a given $\pi(\CS)$.  More precisely:

\begin{theorem}  If $\CS$ satisfies Condition~{\em \ref{cond:Zel}},
then every Jordan--H\"older constituent of $\pi(\CS)$ is isomorphic to
$Q(\CS')$ for some $\CS' \preceq \CS$.
\end{theorem}
\begin{proof} This follows by applying the Zelevinski involution to~\cite[7.2]{Ze1}.
\end{proof}

In fact, the relationship between $\pi(\CS)$ and $\pi(\CS')$ is 
considerably stronger than the theorem above suggests.  We will construct maps
of $\pi(\CS')$ into $\pi(\CS)$ for all $\CS' \preceq \CS$, and show that
any nonzero such map is an embedding, and unique up to scaling.  Before we do so,
however, we need a preliminary result about the partial order $\preceq$.
Let $\CS$ be an unordered collection of segments, and suppose that
$\CS'$ is obtained from $\CS$ by a single elementary operation.  We say this
elementary operation is {\em primitive} if there is no collection of segments
$\CS''$ with $\CS' \preceq \CS'' \preceq \CS$ other than $\CS'' = \CS$ and
$\CS'' = \CS'$.

\begin{lemma} Let $\CS$ be an unordered collection of segments, let $\Delta$
and $\Delta;$ be two linked segments in $\CS$, such that $\Delta$ precedes $\Delta'$.
Suppose that the elementary operation that replaces $\Delta$ and $\Delta'$
with $\Delta \cap \Delta'$, $\Delta \cup \Delta'$ is primitive.  Then there
exists an ordering on $\CS$ that satisfies Condition~{\em \ref{cond:Zel}},
and in which $\Delta'$ and $\Delta$ appear consecutively.
\end{lemma}
\begin{proof}
Choose an ordering on $\CS$ that satisfies Condition~\ref{cond:Zel},
and that minimizes the number of segments that appear between $\Delta'$
and $\Delta$.  Suppose there is a segment $\Delta''$ between $\Delta$
and $\Delta'$.

By our assumption on the chosen ordering, the ordering on $\CS$ obtained by
moving $\Delta''$ after $\Delta$ fails to satisfy Condition~\ref{cond:Zel}.
There must thus be a segment $\Delta'''$ that appears between $\Delta''$
and $\Delta$ in the chosen ordering, for which $\Delta'''$ precedes
$\Delta''$.  Similarly, $\Delta''$ must precede a segment that appears
between $\Delta'$ and $\Delta''$ in the chosen ordering.

Applying these considerations repeatedly we obtain a chain:
$$\Delta' = \Delta_0, \Delta_1, \dots, \Delta_r = \Delta$$
such that each $\Delta_i$ precedes $\Delta_{i-1}$, and appears
after $\Delta_{i-1}$ in the chosen order on $\CS$.
Moreover, since $\Delta$ precedes $\Delta'$, it follows that $\Delta$
precedes $\Delta_1$.  The elementary operation on $\CS$
that replaces $\Delta$ and $\Delta'$ with $\Delta \cap \Delta'$ and
$\Delta \cup \Delta'$ then factors as:
\begin{enumerate}
\item Replace $\Delta$ and $\Delta_1$ with $\Delta \cup \Delta_1$ and $\Delta \cap \Delta_1$.
\item Replace $\Delta'$ and $\Delta \cup \Delta_1$ with $\Delta \cup \Delta_1 \cup \Delta'$
and $\Delta' \cap (\Delta \cup \Delta_1)$.  (Note that $\Delta \cup \Delta_1 \cup \Delta'$ is
equal to $\Delta \cup \Delta'$.)
\item Replace $\Delta' \cap (\Delta \cup \Delta_1)$ and $\Delta \cap \Delta_1$ with
$[\Delta' \cap (\Delta \cup \Delta_1)] \cup [\Delta \cap \Delta_1]$ (which is equal to $\Delta_1$),
and $[\Delta' \cap (\Delta \cup \Delta_1)] \cap [\Delta \cap \Delta_1]$ (which is equal
to $\Delta \cap \Delta'.$)
\end{enumerate}
In particular the elementary operation that replaces $\Delta$ and $\Delta'$
with $\Delta \cap \Delta'$ and $\Delta \cup \Delta'$ is not primitive, as required.
\end{proof}
 
\begin{prop} Suppose that $\CS$ satisfies Condition~{\em \ref{cond:Zel}},
and that $\CS' \preceq \CS$.  Then $\Hom_{\overline{K}[\GL_n(E)]}(\pi(\CS'),\pi(\CS))$
is one-dimensional over $\overline{K}$, and every nonzero map $\pi(\CS') \rightarrow \pi(\CS)$ is
an embedding.
\end{prop}
\begin{proof}
As $\pi(\CS')$ and $\pi(\CS)$ are essentially AIG, it suffices to show that
there exists a nonzero map $\pi(\CS') \rightarrow \pi(\CS)$.  Moreover,
we may reduce to the case where $\CS'$ and $\CS$ differ by a single, primitive, elementary
operation.  Let $\CS'$ differ from $\CS$ by replacing $\Delta,\Delta'$
with $\Delta \cup \Delta', \Delta \cap \Delta'$, where $\Delta'$ precedes $\Delta$.
By the above lemma we may choose an ordering on $\CS$ that satisfies Condition~\ref{cond:Zel}
in which $\Delta$ and $\Delta'$ are adjacent.  We obtain from this ordering on $\CS$ an
ordering on $\CS'$ in which $\Delta \cap \Delta'$ replaces $\Delta$ and
$\Delta \cup \Delta'$ replaces $\Delta'$; this ordering also satisfies Condition~\ref{cond:Zel}.
Let $\CS_0$ be the collection of segments that appear in $\CS$ before $\Delta$ and $\Delta'$ in
this chosen ordering, and let $\CS_1$ be the collection of segments that appear in $\CS$
after $\Delta$ and $\Delta'$.

By applying the Zelevinski involution to \cite[Prop.~4.6]{Ze1}, we find that
$\pi(\Delta \cup \Delta',\Delta \cap \Delta')$ embeds in $\pi(\Delta,\Delta')$.
But $\pi(\CS')$ is isomorphic to 
$$\Ind_P^{\GL_n} \pi(\CS_0) \otimes \pi(\Delta \cup \Delta',\Delta \cap \Delta') \otimes \pi(\CS_1),$$
and $\pi(\CS)$ is isomorphic to
$$\Ind_P^{\GL_n} \pi(\CS_0) \otimes \pi(\Delta,\Delta') \otimes \pi(\CS_1),$$ 
for a suitably chosen parabolic subgroup $P$ of $\GL_n$.
The embedding of $\pi(\Delta \cup \Delta',\Delta \cap \Delta')$
in $\pi(\Delta,\Delta')$ thus gives rise to a nonzero map of $\pi(\CS')$ into $\pi(\CS)$, as required.
\end{proof}

Moreover, the embeddings of $\pi(\CS')$ into $\pi(\CS)$ constructed above descend
to fields of definition:

\begin{prop} \label{prop:embedding descent}
Let $\pi$ and $\pi^{\prime}$ be admissible representations over $K$,
and suppose there are unordered collections of segments $\CS$ and $\CS'$,
with $\CS' \preceq \CS$,
such that $\pi \otimes_K \overline{K}$ is isomorphic to $\pi(\CS)$
and $\pi' \otimes_K \overline{K}$ is isomorphic to $\pi(\CS')$.  Then
$\Hom_{K[\GL_n(E)]}(\pi',\pi)$ is one-dimensional over $K$, and every nonzero 
map $\pi' \rightarrow \pi$ is an embedding.
\end{prop}
\begin{proof}
As $\pi$ and $\pi'$ are essentially AIG, $\Hom_{K[GL_n(E)]}(\pi',\pi)$ is either
zero or one-dimensional over $K$, and every nonzero map $\pi' \rightarrow \pi$
is an embedding.  It thus suffices to construct a nonzero map from $\pi'$ to $\pi$.
Let $\phi: \pi' \otimes_K \overline{K} \rightarrow \pi \otimes_K \overline{K}$
be an embedding.  By Lemma~\ref{lem:AIG descent}, a scalar multiple of
$\phi$ descends to the desired embedding of $\pi'$ in $\pi$.
\end{proof}

We immediately deduce:

\begin{cor} \label{cor:LL descent}
Let $\rho$ be a continuous $n$-dimensional representation of $G_E$ over~$K$,
and let $\pi$ be an admissible representation of $GL_n(E)$ over $K$, such that
$\pi \otimes_K \overline{K}$ is isomorphic to $\LL(\rho \otimes_K \overline{K})$.  Then
$\pi$ is isomorphic to $\LL(\rho)$.
\end{cor}

The above results allow us to establish some useful facts about essentially AIG
envelopes in characteristic zero, that will be useful in the proof of 
Proposition~\ref{prop:small n interpolation}.

\begin{lemma}
\label{lem:char zero AIG}
Let $\pi$ be an irreducible generic representation of $\GL_n(E)$ over an
algebraically closed field $K$ of characteristic zero, and let $\pi_1, \dots, \pi_r$ be
the supercuspidal support of $\pi$, ordered so that Condition~{\em \ref{cond:Zel}}
holds {\em (}when the $\pi_i$ are treated as one-element segments.{\em )}  Then the
parabolic induction
$$\Ind_P^{\GL_n(E)} \pi_1 \otimes \dots \otimes \pi_r$$
is an essentially AIG envelope of $\pi$.  {\em (}Here $P=MU$ is a suitably
chosen parabolic subgroup of $\GL_n(E)$, with Levi subgroup $M$ and
unipotent radical $U$.{\em )}
\end{lemma}
\begin{proof}
By Corollary~\ref{cor:essential}, the representation 
$$\Ind_P^{\GL_n(E)} \pi_1 \otimes \dots \otimes \pi_r$$
is essentially AIG.  Its socle is thus an irreducible generic representation
with the same supercuspidal support as $\pi$, and is therefore isomorphic to
$\pi$, by Proposition~\ref{prop:generic supercuspidal support}. 
It thus remains to show that any essentially AIG representation
$W$ whose socle is isomorphic to $\pi$ embeds in
$$\Ind_P^{\GL_n(E)} \pi_1 \otimes \dots \otimes \pi_r.$$
Note that as any map of essentially AIG representations is injective,
it suffices to construct a map:
$$W \rightarrow \Ind_P^{\GL_n(E)} \pi_1 \otimes \dots \otimes \pi_r.$$
By Frobenius Reciprocity, this is equivalent to constructing a map:
$$\Res_{\GL_n(E)}^P W \rightarrow \pi_1 \otimes \dots \otimes \pi_r.$$

As every Jordan--H\"older constituent of $W$ has supercuspidal support
$\{\pi_1, \dots, \pi_r\}$ by Corollary~\ref{cor:AIG supercuspidal support},
it follows that every Jordan--H\"older constituent of
$\Res_{\GL_n(E)}^P W$ is a supercuspidal representation of $M$, and
at least one of these Jordan--H\"older constituents is isomorphic
to $\pi_1 \otimes \dots \otimes \pi_r$.  By Theorem~\ref{thm:ext},
$\pi_1 \otimes \dots \otimes \pi_r$ only admits nontrivial extensions 
(as an $M$-representation) with irreducible representations isomorphic to
$\pi_1 \otimes \dots \otimes \pi_r$.  Thus $\Res_{\GL_n(E)}^P W$ admits a quotient
isomorphic to $\pi_1 \otimes \dots \otimes \pi_r$ and the result follows.
\end{proof}

\begin{cor}
Let $W$ be an essentially AIG representation of $\GL_n(E)$ over
a field $K$ of characteristic zero.  Then $W$ has finite length.
\end{cor}
\begin{proof}
Let $W'$ be the essentially AIG envelope of $\soc(W)$.  By the preceding lemma,
$W' \otimes_K \overline{K}$ has finite length, so $W'$, and hence $W$, has finite length.
\end{proof}

\begin{cor}
\label{cor:mult}
Let $W$ be an essentially AIG representation of $\GL_2(E)$ or $\GL_3(E)$
over a field $K$ of characteristic zero.  Then no Jordan--H\"older constituent
of $W$ appears with multiplicity greater than one.
\end{cor}
\begin{proof}
Lemma~\ref{lem:char zero AIG} above
shows that $W$ embeds in some parabolic induction
$$\Ind_P^{\GL_n(E)} \pi_1 \otimes \dots \otimes \pi_r$$
with each $\pi_i$ cuspidal.  Zelevinski's computations of
multiplicities of the Jordan--H\"older constituents of such inductions
(\cite[\S 11]{Ze1}) shows that when $r \leq 3$, each Jordan--H\"older
constituent of such an induction occurs with multiplicity one.  The
result follow immediately. 
\end{proof}

\begin{remark}
{\em
In contrast to the preceding proposition,
if $n = 4$, and if we choose a Levi subgroup of the form
$(E^{\times})^4$ of $\GL_4(E)$,
$$\Ind_P^{\GL_n(E)} \abs^2 \otimes \abs
\otimes \abs \otimes 1$$
has a Jordan--H\"older constituent that appears with multiplicity two.
}
\end{remark}

\subsection{Reduction of $\pi(\CS)$}

We now turn to integrality considerations.  We continue to suppose
that $\O$ is a discrete valuation
ring, with residue field $K$ of characteristic zero, uniformizer $\unif$,
and field of fractions $\CK$.  We say an admissible
representation $\pi$ over $\CK$ is $\O$-integral if it contains
a $\unif$-adically separated $\O$-lattice.

\begin{lemma} \label{lem:supercuspidal reduction}
Let $\pi$ be an absolutely irreducible supercuspidal representation of $\GL_n(E)$ over $\CK$.
Then $\pi$ is $\O$-integral if and only if its central character takes values in
$\O^{\times}$.  In this case there is a $\unif$-adically separated $\GL_n(E)$-stable
$\O$-lattice
$\pi^{\circ}$ in $\pi$, unique up to homothety, such that the reduction $\pi^{\circ}/\unif \pi^{\circ}$ 
is absolutely irreducible and supercuspidal.
\end{lemma}
\begin{proof}
Clearly if $\pi$ is $\O$-integral, then its central character takes values in $\O^{\times}$.
Let $\CK'$ be a finite Galois extension of $\CK$, such that there exists a character 
$\chi: E^{\times} \rightarrow \CK^{\times}$ whose $n$th power is the central character of $\pi$.
If the central character of $\pi$ takes values in $\O^{\times}$, then $\chi$ takes values
in $(\O')^{\times}$, where $\O'$ is the integral closure of $\O$ in $\CK'$.

The central character of $\pi \otimes \chi^{-1}\circ \det$ is trivial.
By~\cite[II.4.9]{Vig2}, $\pi \otimes \chi^{-1}\circ \det$ is defined over
a finite extension $F$ of $\Q_p$, contained in $\CK'$.  That is, there 
exists an admissible representation
$\pi_0$ over $F$ such that $\pi_0 \otimes_F \CK'$ is isomorphic to 
$\pi \otimes \chi^{-1}\circ \det$.  As $\O'$ has residue characteristic zero,
$F$ is contained in $\O'$.  Thus $\pi^{\circ} := (\pi_0 \otimes_F \O') \otimes (\chi \circ \det)$
is a $\unif'$-adically separated $\O'$-lattice in $\pi \otimes_{\CK} \CK'$, where
$\unif'$ is a uniformizer of $\O'$.

The reduction modulo $\unif'$ of $\pi^{\circ}$
is $(\pi_0 \otimes_F K') \otimes (\overline{\chi} \circ \det)$, where $K'$ is
the residue field of $\O'$ and $\overline{\chi}$ is the reduction of $\chi$
modulo $\unif'$.  In particular $\pi^{\circ}/\unif \pi^{\circ}$ is absolutely
irreducible and supercuspidal (and therefore $\pi^{\circ}$ is unique up to
homothety.)  It follows that $\pi^{\circ}$ is stable under the action of
$\Gal(\CK'/\CK)$, and hence descends to a $\GL_n(E)$-stable lattice in $\pi$ (which
must also be unique up to homothety.)
\end{proof}

Given an $\O$-integral absolutely irreducible supercuspidal representation $\pi$ of $\GL_n(E)$
over $\CK$, we can thus define $\pibar$ to be the reduction mod $\unif$ of any $\unif$-adically
separated $\GL_n(E)$-stable $\O$-lattice in $\pi$.
For a segment $\Delta = [\pi,(\abs \circ \det)\pi, \dots, (\abs \circ \det)^{r-1}\pi]$, let
$\overline{\Delta}$ be the segment $[\pibar,(\abs \circ \det)\pibar, \dots, (\abs \circ \det)^{r-1}\pibar]$.
If $\CS$ is a collection of integral segments, define $\overline{\CS}$ to be the
collection containing the segments $\overline{\Delta_i}$ for $\Delta_i \in \CS$.

\begin{lemma} Let $\pi$ be an $\O$-integral, absolutely irreducible supercuspidal
representation of $\GL_n(E)$ over $\CK$, and let $\Delta$ be the segment
$[\pi,(\abs \circ \det)\pi, \dots, (\abs \circ \det)^{r-1}\pi]$.  There is a $\unif$-adically separated,
$\GL_n(E)$-stable $\O$-lattice $\St^{\circ}_{\Delta}$ in $\St_{\Delta}$,
unique up to homothety, and $\St^{\circ}_{\Delta}/\unif \St^{\circ}_{\Delta}$
is isomorphic to $\St_{\overline{\Delta}}$.
\end{lemma}
\begin{proof}
This follows by precisely the same argument as in 
Lemma~\ref{lem:supercuspidal reduction}.
\end{proof}

If we want to consider the reduction mod $\unif$ of representations of the form
$\pi(\CS)$,
then the situation is more complicated, as $\pi(\CS)$ typically contains more than one
homothety class of lattices.   However, Proposition~\ref{prop:lattice} allows
us to single out a preferred such homothety class.


\begin{prop} If $\CS$ is an unordered collection of segments over $\CK$ that are $\O$-integral,
then
there is an $\O$-lattice $\pi(\CS)^{\circ}$ in $\pi(\CS)$, unique up to homothety,
such that $\pi(\CS)^{\circ}/\unif \pi(\CS)^{\circ}$ is essentially AIG.  Moreover,
$\pi(\CS)^{\circ}/\unif \pi(\CS)^{\circ}$ is isomorphic
to~$\pi(\overline{\CS})$.
\end{prop}
\begin{proof}
Proposition~\ref{prop:lattice} shows that $\pi(\CS)^{\circ}$ exists and is unique
up to homothety.
Let $\Delta_1, \dots, \Delta_r$ be the segments in $\CS$,
and fix for each $i$ a $\unif$-adically separated $\O$-lattice $L_i$ in $\St_{\Delta_i}$.
Then $L_i/\unif L_i$ is isomorphic to $\St_{\overline{\Delta}_i}$.  Recall that
$$\pi(\CS) = \Ind_P^{\GL_n(E)} \St_{\Delta_1} \otimes \dots \otimes \St_{\Delta_r},$$
and hence contains the integral induction
$\Ind_P^{\GL_n(E)} L_1 \otimes \dots \otimes L_r$
as a lattice.
The mod $\unif$ reduction of this lattice is clearly isomorphic to
$\pi(\overline{\CS})$, which is essentially~AIG.  Thus
$\Ind_P^{\GL_n(E)} L_1 \otimes \dots \otimes L_r$ is homothetic to $\pi(\CS)^{\circ}$,
and hence 
$\pi(\CS)^{\circ}/\unif \pi(\CS)^{\circ}$ is indeed isomorphic to $\pi(\overline{\CS})$,
as claimed.
\end{proof}

\begin{cor} \label{cor:admissible sublattice}
Let $\CK'$ be a finite Galois extension of $\CK$, and let $\O'$ be the integral closure of
$\O$ in $\CK'$.  Let $\pi$ be an admissible representation of $\GL_n(E)$ over
$\CK$, and let $\CS$ be a collection of segments over $\CK'$ that are $\O'$-integral.
Suppose that $\pi \otimes_{\CK} \CK'$ is isomorphic to $\pi(\CS)$.  Then
$\pi$ is $\O$-integral, and there is a $\unif$-adically separated $\O$-lattice $\pi^{\circ}$
in $\pi$, unique up to homothety, such that $\pi^{\circ}/\unif \pi^{\circ}$ is essentially AIG.  
Moreover, $\pi^{\circ}/\unif \pi^{\circ} \otimes_K K'$ is isomorphic to $\pi(\overline{\CS})$.
\end{cor}
\begin{proof}
It suffices to show that the lattice $\pi(\CS)^{\circ}$ constructed in the previous
proposition is stable under the action of $\Gal(\CK'/\CK)$.  This is clear since
$\pi(\CS)^{\circ}$ is unique up to homothety.
\end{proof}

\subsection{Compatibility with specialization}

We now use the results of the previous sections to understand the relationship
between $\LL(\rho \otimes_{\O} \CK)$ and $\LL(\rho \otimes_{\O} K)$, where
$\rho: G_E \rightarrow \GL_n(\O)$ is a continuous Galois representation.
The key idea is a geometric interpretation of the partial order $\preceq$ on
collections of segments, due to Zelevinski~\cite{Ze2}.

Let $V = \oplus_{\pi} V_{\pi}$ be a finite-dimensional vector space over a field $F$, ``graded''
by the set of isomorphism classes of irreducible supercuspidal representations
of $\GL_m(E)$ over $\overline{\CK}$, for all $m$. 
We denote the automorphisms of $V$ as a graded $F$-vector space by $\Aut(V)$, and let 
$\End^+(V)$ denote the space of $F$-linear endomorphisms of $V$ that take
$V_{\pi}$ to $V_{(\abs \circ \det)\pi}$ for all $\pi$.  Let $N_V$ be an element of $\End^+(V)$;
it is a nilpotent endomorphism of $V$.

We construct a bijection between the set of isomorphism classes of pairs $(V,N_V)$
and the set of collections $\CS$ of segments over $\overline{\CK}$, as follows:
For any segment $\Delta = [\pi,(\abs \circ \det)\pi,\dots,(\abs \circ \det)^{r-1}\pi]$,
let $V_{\Delta,F}$ be the vector space defined by
$(V_{\Delta,F})_{\pi^{\prime}} = F$ if $\pi^{\prime}$ is in $\Delta$, and
zero otherwise.  We define an endomorphism $N_{\Delta,F}$ of $V_{\Delta,F}$
that is an isomorphism $(V_{\Delta,F})_{(\abs \circ \det)^i\pi} \rightarrow (V_{\Delta,F})_{(\abs \circ \det)^{i+1}\pi}$
for $0 \leq i < r-1$, and zero otherwise.

For a collection $\CS$ of segments, we define: 
$$(V_{\CS,F},N_{\CS,F}) = \bigoplus_{\Delta \in \CS} (V_{\Delta,F},N_{\Delta,F}).$$
It is easy to see (for instance, by the structure theory of graded $F[N]/n^r$-modules) that
the association $\CS \mapsto (V_{\CS,F},N_{\CS,F})$ yields a bijection between collections of segments and
isomorphism classes of pairs $(V,N_V)$.  

\begin{theorem}[{\cite[\S 2]{Ze2}}]  
\label{thm:orbits}
Let $\CS'$ and $\CS$ be collections of segments over $\overline{K}$.
Then $\CS' \preceq \CS$ if and only if $V_{\CS,F}$ is isomorphic to $V_{\CS',F}$ as a graded $F$-vector space, and
$N_{\CS,F}$ is in the closure of the orbit of $N_{\CS',F}$ under the action of $\Aut(V_{\CS',F})$ on
$\End^+(V_{\CS',F})$.
\end{theorem}

As a result, if $\CS' \preceq \CS$, then, for all $i$,
the rank of $N_{\CS,F}^i$ is less than or equal to that of $N_{\CS',F}^i$.
These ranks are equal for all $i$ if, and only if, $\CS' = \CS$.

If $(\rho',N)$ is a Frobenius-semisimple Weil--Deligne representation over~$\overline{\CK}$,
and if $\CS$ is the collection of segments such that $\pi(\CS) = (\abs \circ \det)^{\frac{n-1}{2}}\LL(\rho',N)$, then the pair
$(V_{\CS,\overline{\CK}},N_{\CS,\overline{\CK}})$ can be described easily in terms of $(\rho',N)$.  Indeed, one has:

\begin{lemma}
For any supercuspidal representation $\pi$ of $\GL_m(E)$ over $\overline{K}$,
there is a natural isomorphism:
$$(V_{\CS,\overline{\CK}})_{\pi} \iso \Hom_{\overline{\CK}[W_E]}(\rho,\rho'),$$
where $\rho$ is the absolutely irreducible representation of $W_E$ that corresponds to $\pi$ under
the unitary local Langlands correspondence.  Moreover, under these isomorphisms, the map
$$N_{\CS,\overline{\CK}}: (V_{\CS,\overline{\CK}})_{\pi} \rightarrow (V_{\CS,\overline{\CK}})_{(\abs \circ \det)\pi}$$
is identified with the map
$$\Hom_{\overline{\CK}[W_E]}(\rho,\rho') \rightarrow \Hom_{\overline{\CK}[W_E]}(\abs \otimes \rho, \rho')$$
induced by $N$.
\end{lemma}
\begin{proof}
This is true by construction if $(\rho',N)$ is indecomposable, and extends to the general case by taking
direct sums.
\end{proof}

Zelevinski's result strongly suggests a connection between the Zelevinski partial order
and reduction of Weil--Deligne representations.  In order to make this connection
precise we need a compatibility between the reduction mod $\unif$ and local Langlands:

\begin{lemma} If $\rho^{\prime}$ is absolutely irreducible, and
$\pi \otimes_{\CK} \overline{\CK}$ corresponds to $\rho^{\prime} \otimes_{\CK} \overline{\CK}$ 
under the unitary local Langlands
correspondence, then $\pibar \otimes_K \overline{K}$ and 
$\rhobar^{\prime} \otimes_K \overline{K}$ correspond 
under the unitary local Langlands correspondence.
\end{lemma}
\begin{proof}
We first translate this into a statement in terms of the generic local
Langlands correspondence.  From this point of view the representation 
$\LL(\rho^{\prime})$ is isomorphic to
$(\abs \circ \det)^{-\frac{n-1}{2}}\pi$, and we must show that
$\LL(\rhobar^{\prime})$ is isomorphic to
$(\abs \circ \det)^{-{\frac{n-1}{2}}}\pibar$.

There is a finite extension $\CK'$ of $\CK$, and a character $\chi: W_E \rightarrow (\CK')^{\times}$,
such that $\rho' \otimes \chi$ is defined over $\Q_p$; as $\rho'$ is integral
$\chi$ takes values in $(\O')^{\times}$, where $\O'$ is the integral closure of $\O$ in $K'$.
In particular, there is a finite extension $K_0$ of $\Q_p$, contained in $\CK'$, and a representation
$\rho_0: W_E \rightarrow \GL_n(K_0)$, such that $\rho_0 \otimes_{K_0} \CK'$ is isomorphic
to $\rho' \otimes \chi$.  If we let $K'$ be the residue field of $\O'$, and let
$\overline{\chi}$ be the reduction mod $\unif'$ of the character $\chi$, then
$K_0$ is contained in $K'$ and $\rho_0 \otimes_{K_0} K'$ is isomorphic to $\rhobar' \otimes \chi$.
It follows that $\LL(\rhobar') \otimes \chi$ is isomorphic to
$(\abs \circ \det)^{-\frac{n-1}{2}}\LL(\rho_0) \otimes \chi$.

Let $\pi_0 = (\abs \circ \det)^{\frac{n-1}{2}}\LL(\rho_0)$.  As the generic local Langlands
correspondence is compatible with twists and base change, the representation
$\pi \otimes \chi$ is isomorphic to $\pi_0 \otimes_{K_0} \CK'$.  Thus $\pibar \otimes \chi$
is isomorphic to $\pi_0 \otimes_{K_0} K'$, and hence to 
$(\abs \circ \det)^{\frac{n-1}{2}}\LL(\rhobar') \otimes \chi$.  The result follows.
\end{proof}

If $\overline{\CS}$ is obtained from a collection of segments $\CS$ by reduction mod $\unif$,
the pairs $(V_{\CS,F},N_{\CS,F})$ and $(V_{\overline{\CS},F},N_{\overline{\CS},F})$
are related by $(V_{\overline{\CS},F})_{\pibar} = \oplus_{\pi^{\prime}} (V_{\CS,F})_{\pi^{\prime}}$,
where the sum is over $\pi^{\prime}$ with $\pibar^{\prime} = \pibar$.

Now let $(\rho',N)$ be a Weil--Deligne representation over $\O$ such that the restriction of
$\rho' \otimes_{\O} \CK$ to $I_E$ is a direct sum of absolutely irreducible representations of $I_E$ over $\CK$,
and such that $\rho' \otimes_{\O} \CK$ is a direct sum of absolutely irreducible representations
$\rho'_i$ of $W_E$.  (We can always arrange this by replacing $\CK$ with a finite extension.)  Let $\CS$
be the segment associated to $(\rho',N) \otimes_{\O} \CK$; we have
$$(V_{\CS,\CK})_{\pi_i} = \Hom_{\CK[W_E]}(\rho_i,\rho' \otimes_{\O} \CK),$$ 
where $\pi_i$ corresponds to $\rho_i$
under unitary local Langlands.  We also consider the $K$-vector-space
$V_{\overline{\CS},K}$.

Let $(\rhobar',\Nbar)$ be the Weil--Deligne representation $(\rho',N) \otimes_{\O} K$,
and let $\CS'$ be the collection of segments associated to $(\rhobar',\Nbar)^{\Fss}$.  
Our goal is to compare $\CS'$ to $\overline{\CS}$; we will do this by comparing
$V_{\overline{\CS},\CK}$ to $V_{\CS',K}$.  The key difficulty is to construct an
$\O$-lattice in $V_{\overline{\CS},\CK}$, stable under $N_{\overline{\CS},\CK}$, 
whose reduction mod $\unif$ is isomorphic to $V_{\CS',K}$.

By Lemma~\ref{lem:clifford} there is a finite extension $K_0$ of $\Q_p$ and representations
$\tau_1, \dots, \tau_n$ of $I_E$ over $K_0$, each in its own orbit under conjugation by
$\Phi$, such that the restriction $\rho' \otimes_{\O} \CK$ is a direct sum of $\Phi$-conjugates
of the $\tau_i$, each with multiplicity one.  For each $i$, let $L_i$ be the $\O$-module
$\Hom_{\O[I_E]}(\tau_i \otimes_{K_0} \O,\rho')$.  

For each $j$ such that the restriction of $\rho'_j$ to $I_E$ contains a copy of
$\tau_i \otimes_{K_0} \CK,$ the restriction map
$$\Hom_{\CK[W_E]}(\rho'_j,\rho') \rightarrow \Hom_{\CK[I_E]}(\tau_i \otimes_{K_0} \CK,\rho')$$
is an injection.  
We thus obtain an isomorphism:
$$L_i \otimes_{\O} \CK \iso \Hom_{\CK[I_E]}(\tau_i \otimes_{K_0} \CK,\rho') \iso
\bigoplus_j \Hom_{\CK[W_E]}(\rho'_j,\rho'),$$
where the latter sum is over those $j$ such that
the restriction of $\rho'_j$ to $I_E$ contains a copy of $\tau_i \otimes_{K_0} \CK$.
Thus condition is satisfied if, and only if, the restriction of $\rhobar'_j$ to $I_E$ contains a copy
of $\tau_i \otimes_{K_0} K$.  We can thus view $L_i$ as a sublattice
of $\oplus_j (V_{\overline{\CS},\CK})_{\pibar_j}$, where the sum is over those $j$
such that the restriction of $\rhobar'_j$ to $I_E$ contains a copy of
$\tau_i \otimes_{K_0} K$.  For any such $j$, let
$L_{\pibar_j}$ be the intersection $L_i \cap (V_{\overline{\CS},\CK})_{\pibar_j}$.

In fact, this gives a direct sum decomposition of $L_i$.  To see this, first observe:
\begin{lemma} 
Let $M$ be a free $\O$-module of finite rank, and $\Psi$ is an
$\O$-linear endomorphism of $M$ that acts semisimply on $M \otimes_{\O} \CK$.
Suppose that all of the eigenvalues of $\Psi \otimes_{\O} \CK$ lie in $\CK$, and
for each eigenvalue $\tilde \lambda$ of $\Psi \otimes_{\O} \CK$, let
$M_{\tilde \lambda}$ be the intersection of $M$ with the $\tilde \lambda$-eigenspace
of $\Psi \otimes_{\O} \CK$.
Then $M$ decomposes as $$M = \oplus_{\lambda} M_{\lambda},$$
where $\lambda$ runs over the eigenvalues of $\overline{\Psi}: M/\unif M \rightarrow M/\unif M$
and $M_{\lambda}$ is the sum of $M_{\tilde \lambda}$ for those ${\tilde \lambda}$
congruent to $\lambda$ modulo $\unif$.
The endomorphism $\Phi$ acts on $M_{\lambda}/\unif M_{\lambda}$ as the product of
$\lambda$ with a unipotent endomorphism of $M_{\lambda}/\unif M_{\lambda}$.
\end{lemma}
\begin{proof}
Let $P(t)$ be the minimal polynomial of $\Psi$, and consider $M$ as an
$\O[t]/P(t)$-module on which $t$ acts by $\Psi$.  The connected components of $\Spec \O[t]/P(t)$
are in bijection with the roots $\lambda$ of the mod $\unif$ reduction
$\overline{\Psi}$ of $\Psi$; these are the eigenvalues of $\overline{\Psi}$.
Thus, considered as a sheaf on $\Spec \O[t]/P(t)$, $M$ decomposes
as a direct sum of sheaves $M_{\lambda}$ supported on each connected component.
On each $M_{\lambda}$, the minimal polynomial of $\overline{\Psi}$ is a power
of $t - \lambda$, so $\lambda^{-1}\overline{\Psi}$ is unipotent on $M_{\lambda}$.
\end{proof}

\begin{lemma}
We have a direct sum decomposition:
$L_i = \oplus_j L_{\pibar_j}$, where the sum is over those $j$
such that the restriction of $\rhobar'_j$ to $I_E$ contains a copy of
$\tau_i \otimes_{K_0} K$.
\end{lemma}
\begin{proof}
Let $r$ be the size of the orbit of $\tau_i$ under the conjugation action of $\Phi$,
and fix an isomorphism $\tau_i^{\Phi^r} \iso \tau_i$.  This isomorphism
induces an endomorphism $\Psi$ of 
$\Hom_{\O[I_E]}(\tau_i \otimes_{K_0} \O, \rho')$ via
$$
\Hom_{\O[I_E]}(\tau_i \otimes_{K_0} \O, \rho') \stackrel{\Phi^r}{\rightarrow}
\Hom_{\O[I_E]}(\tau_i^{\Phi^r} \otimes_{K_0} \O, \rho') \iso
\Hom_{\O[I_E]}(\tau_i \otimes_{K_0} \O, \rho').
$$

Let $\rho'_j$ be an absolutely irreducible summand of $\rho' \otimes_{\O} \CK$ whose
restriction to $I_E$ contains a copy of $\tau_i \otimes_{K_0} \CK$.  This copy is unique,
and yields a restriction map:
$$\Hom_{\CK[W_E]}(\rho'_j, \rho' \otimes_{\O} \CK) \rightarrow \Hom_{\CK[I_E]}(\tau_i, \rho' \otimes_{\O} \CK).$$
This restriction map is injective, and its image can be characterized in terms of $\Psi$.
In particular, the endomorphism:
$$
\Hom_{\CK[I_E]}(\tau_i \otimes_{K_0} \CK, \rho'_j) \stackrel{\Phi^r}{\rightarrow}
\Hom_{\CK[I_E]}(\tau_i^{\Phi^r} \otimes_{K_0} \CK, \rho'_j) \iso{\rightarrow}
\Hom_{\CK[I_E]}(\tau_i \otimes_{K_0} \CK, \rho'_j).
$$
is an endomorphism of one-dimensional $\CK$ vector spaces and is thus given by
a scalar ${\tilde \lambda}$; it follows by Lemma~\ref{lem:lambda} that $\rho'_j$
is determined by ${\tilde \lambda}$ and $\tau_i$, and that the image of the map:
$$\Hom_{\CK[W_E]}(\rho'_j, \rho' \otimes_{\O} \CK) \rightarrow \Hom_{\CK[I_E]}(\tau_i, \rho' \otimes_{\O} \CK)$$
is the ${\tilde \lambda}$-eigenspace of $\Psi$.

Now let $\rhobar_j$ be an absolutely irreducible summand of $\rho' \otimes_{\O} K$
whose restriction to $I_E$ contains a copy of $\tau_i \otimes_{K_0} K$, and let
$\pibar_j$ be the corresponding admissible representation.  Then the endomorphism
$\Psi$ of $\Hom_{K[I_E]}(\tau_i \otimes_{K_0} K, \rhobar_i)$ is a scalar $\lambda$,
and, by the same reasoning as above, $(V_{\overline{\CS},\CK})_{\pibar_j}$ is
the sum of the ${\tilde \lambda}$-eigenspaces of $\Psi$ for those $\tilde \lambda$
congruent to $\lambda$ modulo $\unif$.  Thus, by the preceding lemma,
$L_{\pibar_j}$ is a direct summand of $L_i$.
\end{proof}

Let $L$ be the lattice in $V_{\overline{\CS},K}$ defined by:
$$L = \oplus_{\pibar} L_{\pibar}.$$  Note that as $N_{\overline{\CS},K}$
preserves each $L_i$, it also preserves $L$.

\begin{lemma} 
\label{lem:specialization}
There is a natural isomorphism
$L/\unif L \iso V_{\CS',K}$.  Moreover, the endomorphism $N_{\overline{\CS},\CK}$ of $L$ reduces to $N_{\CS',K}$
under this isomorphism.
\end{lemma}
\begin{proof}
Recall that $\CS'$ is the collection of segments associated to
$(\rhobar',\Nbar)^{\Fss}$.  Let $\rhobar_i$ be any absolutely irreducible Jordan--H\"older constituent
$\rhobar_i$ of $(\rhobar')^{\ss}$, corresponding to an admissible representation
$\pibar_i$ under unitary local Langlands.  Then $(V_{\CS',K})_{\pibar_i}$ is equal to
$\Hom_{K[W_E]}(\rhobar_i,(\rhobar')^{\ss})$. 
It thus suffices to construct, for each $i$, a natural isomorphism of
$(L/\unif L)_{\pibar_i}$ with $(V_{\CS',K})_{\pibar_i}$.

Let $\tau$ be an absolutely irreducible representation of $I_E$ over $K_0$
such that $\rhobar_i$ contains $\tau \otimes_{K_0} K$; let $r$ be the order
of the orbit of $\tau$ under conjugation by $\Phi$, and fix an isomorphism
of $\tau$ with $\tau^{\Phi^r}$.  Let $\lambda \in K^{\times}$ be the scalar
giving the action of $\Phi^r$ on $\Hom_{k[I_E]}(\tau \otimes_{K_0} K, \rhobar_i|_{I_E})$
under this identification.

We also have an action of $\Phi^r$ on 
$\Hom_{K[I_E]}(\tau \otimes_{K_0} K, (\rhobar')^{\ss}|_{I_E})$;
this yields a linear endomorphism $\overline{\Psi}^{\ss}$ of 
$\Hom_{K[I_E]}(\tau \otimes_{K_0} K, (\rhobar')^{\ss}|_{I_E})$.
The natural map
$$(V_{\CS',K})_{\pibar_i} \rightarrow
\Hom_{K[I_E]}(\tau \otimes_{K_0} K, (\rhobar')^{\ss}|_{I_E})$$
identifies $(V_{\CS',K})_{\pibar_i}$ with the $\lambda$-eigenspace
of $\overline{\Psi}^{\ss}$.

On the other hand, the previous lemma shows that
$L_{\pibar_i}$ is the sum of the $\tilde \lambda$-eigenspaces
of $\Psi$ on $\Hom{\O[I_E]}(\tau \otimes_{K_0} \O, \rho'|_{I_E})$;
it follows that $L_{\pibar_i}/\unif L_{\pibar_i}$ is
the $\lambda$-generalized eigenspace of $\Psi/\unif \Psi$ on
$\Hom{K[I_E]}(\tau \otimes_{K_0} K, \rhobar'|_{I_E})$.

Finally, observe that $\overline{\Psi}^{\ss}$ is the semisimplification of
$\Psi/\unif \Psi$, so that the $\lambda$-generalized eigenspace of
$\Psi/\unif \Psi$ is equal to the $\lambda$-eigenspace of
$\overline{\Psi}^{\ss}$, and hence to $(V_{\CS',K})$.  One verifies
easily that these identifications are all compatible with the monodromy
operators.
\end{proof}

By Theorem~\ref{thm:orbits} it follows that for $\CS$ and $\CS'$ as in
Lemma~\ref{lem:specialization}, we must have $\overline{\CS} \preceq \CS'$.
Moreover, we have equality if, and only if, the ranks of the operators
$N_{\CS,\CK}^i$ and $N_{\CS',K}^i$ agree for all $i$.
We are thus finally in a position to prove:

\begin{theorem} \label{thm:specialization}
Let $\rho: G_E \rightarrow \GL_n(\O)$ be a continuous Galois representation, and
$(\rho',N)$ the Frobenius-semisimplification of the
corresponding Weil--Deligne representation.  Then
there is a $\unif$-adically separated $\O$-lattice $\LL(\rho)^{\circ}$ in $\LL(\rho \otimes_{\O} \CK)$, 
unique up to homothety, such that $\LL(\rho)^{\circ}/\unif \LL(\rho)^{\circ}$
is essentially AIG, and an embedding 
$$\LL(\rho)^{\circ}/\unif \LL(\rho)^{\circ} \rightarrow \LL(\rhobar),$$
where $\rhobar = \rho \otimes_{\O} K$.
This embedding is an isomorphism if, and only if, the $K$-rank of $\Nbar^i$
equals the $\CK$-rank of $(N \otimes_{\O} \CK)^i$ for all $i$.
\end{theorem}
\begin{proof}
Let $(\rhobar',\Nbar)$ be the reduction mod $\unif$ of $(\rho',N)$.  Then,
by Lemmas~\ref{lem:WD reduction} and~\ref{lem:semisimple reduction},
$(\rhobar',\Nbar)^{\Fss}$ is the Frobenius-semisimplification
of the Weil--Deligne representation attached to $\rhobar$.

Over a finite extension $\CK'$ of $\CK$, we may assume that
$\rho'$ splits as a direct sum of absolutely irreducible representations
of $W_E$, and similarly for its restriction to $I_E$.  The corresponding
statements then hold for the semisimplification of $\rhobar'$.

Let $\O'$ be the integral closure of $\O$ in $\CK'$, and let $K'$ be
its residue field.
Let $\CS$ and $\CS'$ be the segments associated to $(\rho',N) \otimes_{\O} \O'$ and
$(\rhobar',\Nbar)^{\Fss} \otimes_K K'$.  We have shown that 
$\overline{\CS} \preceq \CS'$.

On the other hand, we have $\LL(\rho \otimes_{\O} \CK')$
is isomorphic to $(\abs \circ \det)^{-\frac{n}{2}}\pi(\CS)$; by 
Corollary~\ref{cor:admissible sublattice}
there is, up to homothety, a unique lattice $\LL(\rho)^{\circ}$ in 
$\LL(\rho \otimes_{\O} \CK)$ such that $\LL(\rho)^{\circ}/\unif \LL(\rho)^{\circ}$
is essentially AIG; moreover one has an isomorphism
$$[\LL(\rho)^{\circ}/\unif \LL(\rho)^{\circ}] \otimes_K K' \iso \pi(\overline{\CS}).$$

As $\LL(\rhobar \otimes_K K')$ is isomorphic to $(\abs \circ \det)^{-\frac{n}{2}}\pi(\CS')$,
and $\overline{\CS} \preceq \CS'$, we have an embedding of
$[\LL(\rho)^{\circ}/\unif \LL(\rho)^{\circ}] \otimes_K K'$ 
in $\LL(\rhobar \otimes_K K')$.  This embedding descends to $K$ by
Proposition~\ref{prop:embedding descent}.

Finally, this embedding is an isomorphism if, and only if,
$\overline{\CS}$ is equal to $\CS'$.  This is true if, and only if,
the ranks of $N_{\overline{\CS},K}^i$ and $N_{\CS',K}^i$ agree for all $i$;
it is easy to see this is equivalent to requiring that the ranks
of $(N \otimes_{\O} K)^i$ and $\Nbar^i$ agree for all~$i$.
\end{proof}

\begin{remark} \rm
An alternative approach to some of the above questions is
given in~\cite{Ch}, particularly Proposition 3.11.  Chenevier
constructs, for each Bernstein component ${\mathcal B}$ of the category 
of smooth representations of $\GL_n(E)$, a pseudocharacter of $W_E$ valued
in the algebra of functions on ${\mathcal B}$ that ``is compatible
with the local Langlands correspondence'', in the sense that if
one specializes this pseudocharacter at any irreducible representation
of $\GL_n(E)$ that lies in ${\mathcal B}$, one obtains the pseudocharacter
of the semisimplification of the corresponding representation of $W_E$.
From our perspective, this result allows us to deduce that the
supercuspidal support of $\LL(\rhobar)$ is the reduction modulo $\unif$
of the supercuspidal support of $\LL(\rho)$, but it does not contain
any information about the monodromy operator.
\end{remark}

In cases where the embedding arising in the previous proposition is
an isomorphism, we say that $\rho$ is a {\em minimal lift} of
$\rhobar$.  (Such lifts are lifts of $\rhobar$ in which the ramification arising
from the monodromy operator is as small as possible.)
We will need this language in a broader context than that
of representations over discrete valuation rings:

\begin{df}
Let $A$ be a reduced complete Noetherian local ring with finite residue field
$k$ of characteristic $p$, that is flat over $W(k)$,
and let $\rho$ be a continuous representation
of $G_E$ into $\GL_n(A)$.  Let $(\rho',N)$ be the associated
Weil--Deligne representation over $\GL_n(A[\dfrac{1}{p}])$.  If ${\mathfrak p}$
is a characteristic zero prime of $A$, and ${\mathfrak a}$ is a prime of
$A$ whose closure contains ${\mathfrak p}$, we say
$\rho_{\mathfrak a}$ is a {\em minimal lift} of $\rho_{\mathfrak p}$
if, for all $i$, the rank of $(N \otimes_A \kappa({\mathfrak a}))^i$ is equal
to the rank of $(N \otimes_A \kappa({\mathfrak p}))^i$.
\end{df}

Note that, for any given ${\mathfrak a}$, the locus of ${\mathfrak p}$ such that
$\rho_{\mathfrak a}$ is a minimal lift of $\rho_{\mathfrak p}$ is Zariski open
in the closure of ${\mathfrak a}$ in $\Spec A[\dfrac{1}{p}]$.

\section{The local Langlands correspondence in characteristic $p$}
\label{sec:LL p}
\subsection{Definition and basic properties}
We now construct an analogue of the Breuil-Schneider local Langlands
correspondence for representations of $G_E$ over finite fields
of characteristic $p$.  Such a correspondence should satisfy an analog of
Theorem~\ref{thm:specialization} for representations over discrete valuation rings
of characteristic zero and residue characteristic $p$.  Throughout this section
we fix a finite field $k$ of characteristic $p$, and let
$\O$ denote a complete discrete valuation ring of characteristic zero
with field of fractions $\CK$ and finite residue field $k'$ containing $k$.

Our starting point is the semisimple mod $p$ local Langlands correspondence
of Vigneras~\cite{Vig4}.  This is a map $\rhobar \mapsto \LLbar_{\ss}(\rhobar)$
that associates to each $n$-dimensional irreducible representation
$\rhobar: W_E \rightarrow \GL_n(\overline{\F}_p)$ an irreducible supercuspidal
representation $\LLbar_{\ss}(\rhobar)$ over $\overline{\F}_p$.  
If $q$ denotes the order of the residue field of $E$,
and if $k'$ is a finite field of characteristic $p$ containing a square root of $q$,
then this correspondence is defined over $k'$; that is, if $\rhobar$ is defined
over $k'$, then $\LLbar(\rhobar)$ descends uniquely to a representation over $k'$.
Moreover, the correspondence is compatible with 
``reduction mod $p$'' in the following sense:

\begin{theorem}[{\cite[Thm.~1.6]{Vig4}}]
\label{thm:LL semisimple}
Suppose that $k'$ contains a square root of~$q$.
Let $(\rho,N)$ be an $n$-dimensional Frobenius-semisimple
Weil--Deligne representation of $W_E$ over $\O$, 
and let $\pi$ be the irreducible representation
of $\GL_n(E)$ over $\CK$ attached to $(\rho,N) \otimes_{\O} \CK$ by the unitary local
Langlands correspondence.  Let $\rhobar = \rho \otimes_{\O} \overline{\F}_p$,
and let
$$\rhobar^{\ss} = \rhobar_1 \oplus \dots \oplus \rhobar_r$$
be a decomposition of $\rhobar^{\ss}$ into irreducible representations
of $W_E$ over $\overline{\F}_p$. Then $\pi$ is $\O$-integral, and for any
$\GL_n(E)$-stable $\O$-lattice $L$ in $\pi$, and any Jordan--H\"older
constituent $\pibar$ of $L \otimes_{\O} \overline{\F}_p$, one has:
$$\scs(\pibar) = \{\LLbar_{\ss}(\rhobar_1) \dots \LLbar_{\ss}(\rhobar_r)\}.$$
\end{theorem}

\begin{cor} \label{cor:independent generic constituent}
Suppose that $k'$ contains a square root of $q$.
Let $\rhobar: G_E \rightarrow \GL_n(k)$ be a Galois representation,
and let $\rho: G_E \rightarrow \GL_n(\O)$ be a lift of $\rhobar \otimes_k k'$.  Then
$\LL(\rho \otimes_{\O} K)$ is
$\O$-integral, and for any $\O$-lattice $L$ in $\LL(\rho \otimes_{\O} \CK)$,
the supercuspidal support of any Jordan--H\"older constituent
$\pibar$ of $L \otimes_{\O} \overline{\F}_p$ depends only on $\rhobar$.
\end{cor}
\begin{proof}
Let $(\rho',N)$ be the Frobenius-semisimple Weil--Deligne representation
over $\CK$ attached to $\rho$.  Then $\rho'$ is $\O$-integral and the semisimplification
of its reduction mod $p$ depends only on $\rhobar$.  By the definition of
the Breuil-Schneider local Langlands correspondence, $\LL(\rho \otimes_{\O} \overline{\CK})$
is (up to a twist by an integral character) a parabolic induction of representations
that correspond (via unitary local Langlands) to irreducible summands of
$\rho' \otimes_{\O} \overline{\CK}$.  These summands are integral, so 
$\LL(\rho \otimes_{\O} \overline{\CK})$ is as well, and so is $\LL(\rho \otimes_{\O} \CK)$.
Moreover, (up to a twist by $(\abs \circ \det)^{\frac{n-1}{2}}$),
every Jordan--H\"older constituent of $\LL(\rho \otimes_{\O} \overline{\CK})$ corresponds via
unitary local Langlands to a Weil--Deligne representation of the form
$(\rho' \otimes_{\O} \overline{\CK} ,N')$ for some choice of monodromy operator $N'$.  

Now if $L$ is a lattice in $\LL(\rho \otimes_{\O} \CK)$, and $\pibar$ is a
Jordan--H\"older constituent of $L \otimes_{\O} \overline{\F}_p$, then there exists
a Jordan--H\"older constituent of $\LL(\rho \otimes_{\O} \overline{\CK})$,
and a lattice $L'$ in this constituent, such that $\pibar$ is a Jordan--H\"older
constituent of the mod $p$ reduction of $L'$.  The result thus follows
from Theorem~\ref{thm:LL semisimple}.
\end{proof}

Let $L$ be a lattice in $\LL(\rho \otimes_{\O} \CK)$, where $\rho:G_E
\rightarrow \GL_n(\O)$ is a lift of $\rhobar\otimes_k k'$
for some $\rhobar:G_E \rightarrow \GL_n(k)$.
As $L \otimes_{\O} \Fbar_p$ has a unique generic Jordan--H\"older constituent,
and up to isomorphism there is only one irreducible generic representation
of $G$ with given supercuspidal support, the generic Jordan--H\"older constituent
of $L \otimes_{\O} \Fbar_p$ likewise depends only on $\rhobar$.

We will also need to control the length of $L/\unif L$, for lattices $L$
of the sort appearing in the Corollary above.  We first show:

\begin{prop} \label{prop:induction length bound}
Let $P = MU$ be a parabolic subgroup of $\GL_n(E)$, and let $\pi = \pi_1 \otimes \dots \otimes \pi_r$
be an irreducible representation of $M$.  There exists an integer~$c$, depending only
on~$n$, such that the length of
$\Ind_P^{\GL_n(E)} \pi_1 \otimes \dots \otimes \pi_r$ is bounded above by $c$.
\end{prop}
\begin{proof}
In fact, we bound the length of this induction as a representation of $P_n(E)$, by induction
on $n$.  Suppose we have a bound $c'$ on this length for representations of $P_m(E)$, $m < n$.
The length of a representation of $P_n(E)$ is equal to the sum of the lengths
of all of its derivatives, as every irreducible representation of $P_n(E)$ has a unique nonzero
derivative.  By the Leibniz rule for derivatives, each derivative of $P_n(E)$ is a sum
of parabolic inductions of proper subsets of $\{\pi_1, \dots \pi_r\}$; there are at most
$2^n - 1$ such subsets, and each such parabolic induction has length at most $c'$.  In
particular, we can take $c = (2^n - 1)c'$ (although this is most likely far from sharp).
\end{proof}

\begin{prop} \label{prop:length bound}
Let $\rhobar: G_E \rightarrow \GL_n(k)$ be a Galois representation, let
$\rho: G_E \rightarrow \GL_n(\O)$ be a lift of $\rhobar \otimes_k k'$, and let
$L$ be a $\GL_n(E)$-stable lattice in $\LL(\rho \otimes_{\O} \CK)$.  There exists
an integer $c$, depending only on $n$, such that the length of
$L/\unif L$ is bounded above by $c$.
\end{prop}
\begin{proof}
The length of $L/\unif L$ is independent of $L$.  As $\LL(\rho \otimes_{\O} \CK)$
is a parabolic induction of a tensor product of
integral Steinberg representations, we can write
$$\LL(\rho \otimes_{\O} \CK) = \Ind_P^{\GL_n(E)} \St_{\pi_1,n_1} \otimes \dots \otimes \St_{\pi_i, n_i},$$
where the $\pi_i$ are integral cuspidal representations of $\GL_n$.  For each $i$,
$\St_{\pi_i,n_i}$ arises as the normalized parabolic induction of a tensor product of
the form:
$$(\abs \circ \det)^{-\frac{n_i - 1}{2}}\pi_i \otimes \dots \otimes (\abs \circ \det)^{\frac{n_i - 1}{2}}\pi_i.$$
Thus there is a parabolic induction of a tensor product of irreducible, integral, cuspidal representations $\pi'_j$
(all of which are twists of the $\pi_i$) that maps surjectively onto $\LL(\rho \otimes_{\O} \CK)$; if we choose a 
lattice $L_j$ inside each of the $\pi'_j$, the parabolic induction of the tensor product of the $L_j$
maps into a lattice $L$ in $\LL(\rho \otimes_{\O} \CK)$.  We then have a surjection of the parabolic induction
of the tensor product of $L_j/\unif L_j$ onto $L/\unif L$.  As each $\pi'_j$ is cuspidal, so is $L_j/\unif L_j$;
as $(L_j/\unif L_j)^{(n)}$ is one-dimensional we must have $L_j/\unif L_j$ irreducible for all $j$.
Thus the length of $L/\unif L$ is bounded above by the maximum length of a parabolic induction of
an irreducible representation of a Levi subgroup of $\GL_n(E)$, and the desired result
follows by Proposition~\ref{prop:induction length bound}.
\end{proof}

We can now prove the main result of this subsection.

\begin{theorem}
\label{thm:mod p LL}
There is a map
$\rhobar \mapsto \LLbar(\rhobar)$
from the set of isomorphism classes
of continuous representations $G_E \rightarrow \GL_n(k)$
to the set of isomorphism classes of finite length admissible smooth
$\GL_n(E)$-representations on $k$-vector spaces, uniquely
determined by the following three conditions:

\begin{enumerate}
\item For any $\rhobar$, the associated $\GL_n(E)$-representation
$\LLbar(\rhobar)$ is essentially~AIG.
\item If $\CK$ is a finite extension of $\Q_p$, with ring of integers $\O$,
uniformizer $\unif$, and residue field $k'$ containing $k$, $\rho: G_E \rightarrow \GL_n(\O)$
is a continuous representation lifting $\rhobar \otimes_k k'$, and $L$ is a $\GL_n(E)$-invariant
$\O$-lattice in $\LL(\rho)$ such that $L/\unif L$ is essentially AIG,
then there is a $\GL_n(E)$-equivariant embedding
$L/\unif L \hookrightarrow \LLbar(\rhobar) \otimes_k k'$.  {\em (}Note that 
Proposition~{\em \ref{prop:lattice}} shows
that such an $L$ always exists, and that it is unique up to homethety.{\em )}
\item The representation $\LLbar(\rhobar)$ is minimal with respect
to satisfying conditions~{\em (1)} and {\em (2)}, i.e.\ given any continuous representation
$\rhobar: G_E \rightarrow \GL_n(k)$, and any representation $\pibar$ of $\GL_n(E)$
satisfying these two conditions with respect to $\rhobar$, there is a
$\GL_n(E)$-equivariant embedding $\LL(\rhobar) \hookrightarrow \pibar$.

\smallskip
\noindent
Furthermore:

\smallskip

\item The formation of $\LLbar(\rhobar)$ is compatible with
extending scalars, i.e.\ given $\rhobar: G_E \rightarrow \GL_n(k)$,
and a finite extension $k'$ of $k$, one has
$$\LLbar(\rhobar \otimes_k k') \cong \LLbar(\rhobar) \otimes_k k'.$$

\item The formation of $\LLbar(\rhobar)$ is compatible with
twists, i.e.\ given $\rhobar: G_E \rightarrow \GL_n(k)$, and
a continuous character $\chibar: G_E \rightarrow k^{\times}$, one has
$$\LLbar(\rhobar \otimes \chibar) = \LLbar(\rhobar) \otimes (\chibar \circ \det).$$

\item $\End_{\GL_n(E)}(\LLbar(\rhobar)) = k$.

\item The representation $\LLbar(\rhobar)$ has central character equal to
$\absbar^{\frac{n(n-1)}{2}} (\det \rhobar)$.

\item Suppose $(\rhobar \otimes_k \overline{k})^{\ss}$ is the direct sum
of irreducible representations $\rhobar_1, \dots, \rhobar_r$.
Then every Jordan--H\"older constituent of $\LL(\rhobar)$ has 
supercuspidal support equal to $\abs^{\frac{n-1}{2}}\{\LLbar_{\ss}(\rhobar_1), \dots, \LLbar_{\ss}(\rhobar_r) \} $

\end{enumerate}
\end{theorem}
\begin{proof}
We first establish uniqueness: If $\pibar$ and $\pibar'$ are two finite length
admissible smooth representations of $\GL_n(E)$ that satisfy properties (1),
(2), and (3) with respect to $\rhobar$, then by property (3)
we have embeddings of $\pibar$ in $\pibar'$ and vice versa.  As both $\pibar$
and $\pibar'$ have finite length these embeddings are isomorphisms.

We now turn to the construction of $\LLbar(\rhobar)$.  Let $\rho: G_E \rightarrow
\GL_n(\O)$ be a lift of $\rhobar \otimes_k k'$, for some $\O,k'$ as in property {\em 2},
and suppose $L$ is an $\O$-lattice in
$\LL(\rho)$ such that $L/\unif L$ is essentially AIG.  The socle $V$ of
$L/\unif L$ is absolutely irreducible and generic, and its supercuspidal
support depends only on $\rhobar$ and not the specific lift $\rho$
chosen.  As there is a unique generic representation with given supercuspidal
support, $V$ depends only on $\rhobar$ up to isomorphism.  In particular
$V$ is defined over~$k$, as we can take $\O$ to have residue field $k$.

Let $\env(V_{\overline{k}})$ be the essentially AIG envelope of $V \otimes_k \overline{k}$.  
For each lift $\rho$
of $\rhobar$, and each lattice $L$ in $\pi(\rhobar)$ such that $L/\unif L$ is
essentially AIG, the socle of $(L/\unif L) \otimes_k \overline{k}$ is isomorphic to 
$V \otimes_k \overline{k}$.  Hence
$(L/\unif L) \otimes_{k'} \overline{k}$ embeds uniquely (up to the action of 
$\overline{k}^{\times}$) in $\env(V_{\overline{k}})$. 
Let $\LLbar(\rhobar)_{\overline{k}}$ be the sum, in $\env(V_{\overline{k}})$, of the images of
$(L/\unif L) \otimes_k \overline{k}$ in $\env(V_{\overline{k}})$ as $\rho$ ranges over all lifts of $\rhobar$.  

By construction, $\Gal(\overline{k}/k)$ acts on $\env(V_{\overline{k}})$.  This action
preserves $\LLbar(\rhobar)_{\overline{k}}$, as it permutes the images of $L/\unif L$
for various $\O$ and $\rho$.  Thus $\LLbar(\rhobar)_{\overline{k}}$ descends uniquely
to a submodule $\LLbar(\rhobar)$ of $\env(V)$.  Clearly,
$\LLbar(\rhobar)$ satisfies properties (1) and (2).  On the other hand,
if $\pibar$ is any other representation satisfying properties (1) and (2),
then the socle of $\pibar$ is isomorphic to $V$ and hence $\env(V)$ contains
a unique submodule isomorphic to $\pibar$.  As $\pibar$ satisfies property (2),
$\pibar \otimes_k \overline{k}$
contains the images of $L/\unif L$ in $\env(V_{\overline{k}})$ for all lifts $\rho$ of $\rhobar$,
and thus contains $\LLbar(\rhobar) \otimes_k \overline{k}$.  It follows
that $\pibar$ contains $\LLbar(\rhobar)$, so $\LLbar(\rhobar)$ satisfies property (3).
Finally, $\LLbar(\rhobar)$ is finite length by Proposition~\ref{prop:length bound}
and Corollary~\ref{cor:AIG bounded}.

Now let $k'$ be a finite extension of $k$.  Then $\LLbar(\rhobar) \otimes_k k'$ clearly
satisfies properties (1) and (2) with respect to $\rhobar \otimes_k k'$, and
thus admits an embedding of $\LLbar(\rhobar \otimes_k k')$ that is unique up to rescaling.
The above construction shows that $\LLbar(\rhobar) \otimes_k \overline{k}$ and 
$\LLbar(\rhobar \otimes_k k') \otimes_k \overline{k}$ coincide as submodules of $\env(V_{\overline{k}})$,
so this embedding is an isomorphism.

Similarly, if $\chibar$ is a character of $E^{\times}$ with values in $k^{\times}$,
we can choose a lift $\chi$ of $\chibar$ to a character with values in $W(k)^{\times}$.
Then if $\rho \otimes \chi$ is a lift of $\rhobar \otimes (\chibar \circ \det)$ to a representation
over $\O$, and $L \otimes \chi$
is a lattice in $\pi(\rho \otimes \chi)$ with $(L \otimes (\chi \circ \det))/\unif(L \otimes (\chi \circ \det))$ 
essentially AIG, 
then $L$ is a lattice in $\pi(\rho)$ with $L/\unif L$ essentially AIG.  Thus $L/\unif L$ embeds
in $\LLbar(\rhobar) \otimes_k k'$, so $(L/\unif L) \otimes (\chibar \circ \det)$ embeds in
$\LLbar(\rhobar) \otimes (\chibar \circ \det)$.  Thus $\LLbar(\rhobar) \otimes (\chibar \circ \det)$
has property (2) and hence contains $\LLbar(\rhobar \otimes \chibar)$.  Conversely,
replacing $\rhobar$ with $\rhobar \otimes \chibar$, we find that 
$\LLbar(\rhobar \otimes \chibar) \otimes(\chibar^{-1} \circ \det)$ contains $\LLbar(\rhobar)$.
Thus $\LLbar(\rhobar)$ and $\LLbar(\rhobar \otimes \chibar)$ have the same length,
and so $\LLbar(\rhobar \otimes \chibar)$ and $\LLbar(\rhobar) \otimes (\chibar \circ \det)$
are isomorphic.

The endomorphisms of $\LLbar(\rhobar)$ are all scalar because $\LLbar(\rhobar)$
is essentially AIG.  In particular the center of $\GL_n(E)$ acts on $\LLbar(\rhobar)$
(and hence on all of its submodules) via a character.  To compute this character,
let $\rho$ be any lift of $\rhobar$, and let $L$ be a lattice in $\pi(\rho)$ such
that $L/\unif L$ is essentially AIG.  The center of $\GL_n(E)$ acts on $\pi(\rho)$
via the character $\abs^{\frac{n(n-1)}{2}}\det \rho$, and hence on
$L/\unif L$ via the character $\absbar^{\frac{n(n-1)}{2}}\det \rhobar$.
As $L/\unif L$ embeds in $\LLbar(\rhobar)$, this character is also the central
character of $\LLbar(\rhobar)$.

As $\LLbar(\rhobar)$ is essentially AIG, every Jordan--H\"older constituent of $\LLbar(\rhobar)$
has the same supercuspidal support.  To determine this supercuspidal support, let $\rho$ be
any lift of $\rhobar$, and let $L$ be a lattice in $\pi(\rho)$ such that $L/\unif L$ is
essentially AIG.  The representations $\abs^{-\frac{n-1}{2}}\pi(\rho)$ and $\rho$
then correspond under unitary local Langlands, and so, by Theorem~\ref{thm:LL semisimple},
the supercuspidal support of any Jordan--H\"older constituent of
$\abs^{-\frac{n-1}{2}} L/\unif L$ is equal to $\{\LLbar_{\ss}(\rhobar_1), \dots \LLbar_{\ss}(\rhobar_r)\}$.
\end{proof}

\subsection{The local Langlands correspondence for $\GL_2$ in characteristic $p$}

For $\GL_2(E)$, at least in odd characteristic, the correspondence
$\rhobar \mapsto \LL(\rhobar)$ can be made fairly concrete.  The first thing to observe is:

\begin{prop}
Let $\rhobar: G_E \rightarrow \GL_2(\Fbar_p)$ be a representation, and suppose that
$\rhobar^{\ss}$ is not a twist of $1 \oplus \abs$.  Then $\LLbar(\rhobar)$ is the unique
representation of $\GL_2(E)$ whose supercuspidal support is given by part~{\em (8)}
of Theorem~{\em \ref{thm:mod p LL}}.
\end{prop}
\begin{proof}
Part (8) of Theorem~\ref{thm:mod p LL}, together with our hypothesis 
on $\rhobar$ implies that the supercuspidal
support of any Jordan--H\"older constituent of $\LLbar(\rhobar)$ is either a 
single supercuspidal representation of $\GL_2(E)$, or a pair
of characters of $\GL_1(E)$ that do not differ by a factor of $\abs$.  In either case,
there is, up to isomorphism, a {\em unique} irreducible representation $\pibar$ of $\GL_2$ that has that 
particular supercuspidal
support; in particular $\pibar$ is generic.  Thus every Jordan--H\"older constituent of $\env(\pibar)$
is isomorphic to $\pibar$; as $\pibar$ is generic and $\env(\pibar)$ is essentially AIG there is only
one such Jordan--H\"older constituent.  In particular $\pibar = \env(\pibar)$. and so
As $\LL(\rhobar)$ is contained in $\env(\pibar)$, we must have $\LL(\rhobar) = \pibar$.
\end{proof}

When $\rhobar^{\ss}$ is a twist of $1 \oplus \abs$, the situation is more complicated, as
$\LL(\rhobar)$ will typically not be irreducible.  As the correspondence $\rhobar \mapsto \LL(\rhobar)$
is compatible with twists, it suffices to describe $\LL(\rhobar)$ when $\rhobar^{\ss} = 1 \oplus \abs$.
In this case $\LL(\rhobar)$ has supercuspidal support $\{1,\absbar\}$.
The details of this will be carried out in~\cite{He}; here we content ourselves with summarizing the results.

First, assume that the order $q$ of the residue field of $E$ is not congruent to $\pm 1$
modulo $p$.  (This is the so-called {\em banal} situation.)  Here there are two irreducible
representations of $G$ with supercuspidal support $\{1,\absbar\}$: the character $\absbar \circ \det$
and the twisted Steinberg representation $\St \otimes (\absbar \circ \det)$.  The latter
representation is generic, and its envelope is the unique nonsplit extension
of $\absbar \circ \det$ by $\St \otimes (\absbar \circ \det)$. 

On the Galois side there is, up to isomorphism, a unique nonsplit $\rhobar$ whose
semisimplification is $1 \oplus \abs$.  Then $\LL(\rhobar)$ is equal
to $\St \otimes (\absbar \circ \det)$ if $\rhobar$ is nonsplit, and to the
unique nonsplit extension of $\absbar \circ \det$ by $\St \otimes (\absbar \circ \det)$
if $\rhobar$ is split.

Next, assume that $p$ is odd and $q$ is congruent to $-1$ modulo $p$.  In this case there
are three irreducible representations of $G$ with supercuspidal support $\{1,\absbar\}$:
the trivial character, the character $\absbar \circ \det$, and a cuspidal generic representation
that Vigneras denotes by $\pi(1)$ (see \cite[II.2.5]{Vig2} for
a discussion of this).  Up to isomorphism, there is a unique nonsplit extension
of the trivial character by $\pi(1)$ and similarly a unique nonsplit extension of
$\absbar \circ \det$ by $\pi(1)$.  The envelope $\env(\pi(1))$ is the unique extension of
$1 \oplus (\absbar \circ \det)$ by $\pi(1)$ that contains both of these nonsplit extensions
as submodules.

In this case $\LL(\rhobar) = \env(\pi(1))$ if $\rhobar$ is split.  If $\rhobar$ is not
split, it is either an extension of $\absbar$ by $1$ or an extension of $1$ by $\absbar$.  In the
first case, $\LL(\rhobar)$ is the nonsplit extension
of $(\absbar \circ \det)$ by $\pi(1)$; in the second case 
$\LL(\rhobar)$ is the nonsplit extension
of the trivial character by $\pi(1)$.

Finally, assume that $p$ is odd and $q$ is congruent to $1$ modulo $p$.  In this case $\absbar$
is the trivial character.  The only irreducible representations of $G$ with supercuspidal
support $\{1,1\}$ in this case are the Steinberg representation $\St$ and the trivial representation.
Moreover, $\Ext^1(1,\St)$ is two-dimensional, and naturally isomorphic to $H^1(G_E,1)$.  The
envelope $\env(\St)$ is isomorphic to the universal extension of $1$ by $\St$, and thus has length
three.

In this case $\LL(\rhobar)$ is equal to $\env(\St)$ if $\rhobar$ is split.  On the other hand,
the nonsplit $\rhobar$ with trivial semisimplification are in bijection with the one-dimensional 
subspaces of $H^1(G_E,1)$, and hence
in natural bijection with the one-dimensional subspaces of $\Ext^1(1,\St)$.  These in turn 
correspond to the nonsplit extensions of $1$ by $\St$.  For any nonsplit $\rhobar$, $\LL(\rhobar)$ 
is the corresponding extension of $1$ by $\St$.

\section{The local Langlands correspondence in families}
\label{sec:LL families}
\subsection{The set-up}
Throughout this section 
we will be considering representations over rings $A$ satisfying
the following condition:

\begin{condition}
\label{cond:A}
{\em
$A$ is a complete reduced Noetherian local ring,
with finite residue field $k$ of characteristic $p$,
which is flat over the ring of Witt vectors $W(k)$.
}
\end{condition}

We will typically write $\mathfrak m$ for the maximal ideal of $A$.
Note that the condition of being flat over $W(k)$ is equivalent
to $A$ being $p$-torsion free, or again (since $A$ is reduced),
to each minimal prime of $A$ being of residue characteristic $0$.
We will write $\kappa(\mathfrak p)$ to denote the residue
field of a prime ideal $\mathfrak p$ of $A$; thus
$\kappa(\mathfrak p)$ is the fraction field of the
complete local domain $A/\mathfrak p$.
We write
$\displaystyle \CK(A) := 
\prod_{\mathfrak a \text{ minimal}} \kappa(\mathfrak a)$
(where, as indicated, the product is taken over the finitely
many minimal primes of $A$) for the total quotient ring of $A$.
Since $A$ is reduced, the natural map
$ A\rightarrow \CK(A)$
is an embedding.

\subsection{Statement of the correspondence and related results}
\label{subsec:results}
Now let $E$ be a number field.  Let $v$ be a non-archimedean place of $E$,
and let
$\rho: G_{E_v} \rightarrow \GL_n(A)$ be a continuous representation
(when the target is equipped with its $\mathfrak m$-adic topology).
For each prime ideal $\mathfrak p$ of $A$,
let
$\rho_{\mathfrak p}: G_{E_v} \rightarrow
\GL_n\bigl(\kappa(\mathfrak p)\bigr)$
denote the representation
obtained from $\rho$ by extending scalars from $A$ to $\kappa(\mathfrak p)$.
In the particular case of the maximal ideal, we also write
$\rhobar := \rho_{\mathfrak m}$.  If $\mathfrak p$ is a prime of $A$
with residue characteristic zero, we write $\LLcheck(\rho_{\mathfrak p})$
for the smooth $\kappa(\mathfrak p)$-dual of the representation
$\LL(\rho_{\mathfrak p})$ defined in Definition~\ref{def:LL galois}.

In the situations that we will consider below,
we will have a finite $S$ of non-archimedean places of $E$, all prime to $p$,
and for each $v \in S$ we will have a continuous representation
$\rho_v: G_{E_v} \rightarrow \GL_n(A)$.

\medskip

We are now ready to describe the local Langlands correspondence
for local Galois representations over $A$.

\begin{theorem}
\label{thm:family}
Let $S$ denote a finite set of non-archimedean places of $E$, none of which lie over $p$,
and suppose for each $v \in S$ that we are given a representation
$\rho_v: G_{E_v} \rightarrow \GL_n(A)$.
If we write $G := \prod_{v \in S}\GL_n(E_v),$
then there is {\em (}up to isomorphism{\em )} at most one
admissible smooth representation
$V$ of $G$ over $A$ satisfying the following conditions:
\begin{enumerate}
\item $V$ is $A$-torsion free {\em (}i.e.\ all associated
primes of $V$ are minimal primes of~$A$, or equivalently,
the natural map $V \rightarrow \CK(A) \otimes_A V$ is an
embedding{\em )}.
\medskip
\item For each minimal prime $\mathfrak a$ of $A$,
there is a $G$-equivariant isomorphism
$$\bigotimes_{v \in S}\LLcheck(\rho_{v, \mathfrak a}) \iso
\kappa(\mathfrak a)\otimes_A V.$$
\item The $G$-cosocle $\cosoc(V/\mathfrak m V)$
of $V/\mathfrak m V$ is absolutely irreducible and generic,
while the kernel of the natural surjection
$V/\mathfrak m V \rightarrow  \cosoc(V/\mathfrak m V)$ 
contains no generic subrepresentations.  {\em (}In other words,
the smooth dual of $V/\mathfrak m V$ is essentially AIG.{\em )}
\smallskip
\newline
Any such $V$ satisfies the following additional conditions:
\newline
\item $V$ is cyclic as an $A[G]$-module.
\medskip
\item
$\End_{A[G]}(V) = A$.
\end{enumerate}
\end{theorem}

We postpone the proof of the theorem to the following subsection.

\begin{df}
\label{def:family}
{\em
If in the context of the preceding theorem an $A[G]$-module $V$
satisfying conditions~(1), (2), and~(3) exists, then we write
$\LLcheck(\{\rho_v\}_{v \in S}) := V$.
(This is justified by the uniqueness statement of the theorem.)
If $S$ consists of a single place $v$ then we write
$\LLcheck(\rho_v)$ rather than $\LLcheck(\{\rho_v\}_{v \in S}).$
}
\end{df}

\begin{remark}
{\em
We don't consider here the problem of proving in general that a 
representation $V$ satisfying
conditions~(1), (2) and~(3) of Theorem~\ref{thm:family} exists, 
although we conjecture that it does.  (This is Conjecture~\ref{conj:LL}
of the introduction.)   When $n=2$ and $p$ is odd, this conjecture is
a result of the second author~\cite{He}.

In the global applications considered in the work of the first
author~\cite{Em8, Em9}, and in subsequent
applications,
the problem that we will confront will rather be that 
of having a smooth $G$-representation
at hand (for a certain ring~$A$), which we wish to show satisfies 
the conditions to be
$\LLcheck(\{\rho_v\}_{v \in S})$
for an appropriate $\rho$.  Thus one of our goals in the
following subsection is to establish a workable criterion
for recognizing 
$\LLcheck(\{\rho_v\}_{v \in S})$ (namely Theorem~\ref{thm:verify} below).
}
\end{remark}

The following result shows that the existence 
of $\LLcheck(\{\rho_v\}_{v \in S})$
is equivalent to the existence of the 
collection of representations $\LLcheck(\rho_v)$,
and explains the relation between them.
We postpone its proof to the following subsection.

\begin{prop}
\label{prop:individual}
In the context of Theorem~{\em \ref{thm:family}},
the $A[G]$-module $\LLcheck(\{\rho_v\}_{v \in S})$ exists
if and only if each of the individual $A[\GL_n(E_v)]$-modules
$\LLcheck(\rho_v)$ exist.
Furthermore, $\LLcheck(\{\rho_v\}_{v \in S})$ is isomorphic
to the maximal torsion free quotient of the tensor product
{\em (}taken over $A${\em )}
$\bigotimes_{v \in S} \LLcheck(\rho_v).$
\end{prop}

The following two theorems, whose proofs we again postpone, describe the sense
in which the representation $\LLcheck(\{\rho_v\})$ interpolates the Breuil-Schneider 
modified local Langlands correspondence over $\Spec A[\dfrac{1}{p}]$.

\begin{theorem}
\label{thm:interpolation1}
Let $\mathfrak p$ be a prime of $A[\dfrac{1}{p}]$, and suppose that
$\mathfrak p$ lies on exactly one irreducible component of
$\Spec A[\dfrac{1}{p}].$  
Then, assuming that $\LLcheck(\{\rho_v\}_{v \in S})$ exists, there is a
a $\kappa(\mathfrak p)$-linear $G$-equivariant surjection
$$\bigotimes_{v \in S}
\LLcheck(\rho_{v, {\mathfrak p}}) \rightarrow
\kappa(\mathfrak p) \otimes_A \LLcheck(\{\rho_v\}_{v \in S}).$$
Moreover, if there exists a minimal prime ${\mathfrak a}$ of $A$
such that $\rho_{\mathfrak a}$ is a minimal lift of~$\rho_{\mathfrak p}$,
then this surjection is an isomorphism.
\end{theorem}

It seems likely that the above result holds even when
$\mathfrak p$ is contained in multiple irreducible components
of $\Spec A[\dfrac{1}{p}]$.  Nonetheless we are at present only
able to prove a somewhat weaker statement:

\begin{theorem}
\label{thm:interpolation2}
Assume that $\LLcheck(\{\rho_v\}_{v \in S})$ exists, 
let ${\mathfrak p}$ be a prime of $\Spec A[\dfrac{1}{p}]$,
let ${\mathfrak a}_1, \dots, {\mathfrak a}_r$ be the minimal primes
of $A$ containing ${\mathfrak p}$, for each $i = 1,\ldots, r$ let $V_i$
be the maximal $A$-torsion free quotient of $\LLcheck(\{\rho_v\}_{v \in S}) \otimes_A A/{\mathfrak a}_i$,
and denote by $W$ the image of the diagonal map
$$\kappa(\mathfrak p) \otimes_A \LLcheck(\{\rho_v\}_{v \in S}) \rightarrow
\prod_i \kappa(\mathfrak p) \otimes_{A/{\mathfrak a}_i} V_i.$$
Then there is a
a $\kappa(\mathfrak p)$-linear $G$-equivariant surjection
$$\bigotimes_{v \in S}
\LLcheck(\rho_{v,{\mathfrak p}}) \rightarrow
W.$$
Moreover, if there exists a minimal prime ${\mathfrak a}$ of $A$
such that $\rho_{\mathfrak a}$ is a minimal lift of~$\rho_{\mathfrak p}$,
then this surjection is an isomorphism.
\end{theorem}

\begin{conj}
\label{conj:interpolation}
Under the hypotheses of Theorem~{\em \ref{thm:interpolation2}}, the map
$$\kappa({\mathfrak p}) \otimes_A \LLcheck(\{\rho_v\}_{v \in S}) \rightarrow W$$
is an isomorphism.  In particular the conclusion of Theorem~{\em \ref{thm:interpolation1}}
holds for all ${\mathfrak p}$.
\end{conj}

Although this conjecture seems difficult to establish in general, we
have the following result for small $n$, which we prove in the next section.

\begin{prop}
\label{prop:small n interpolation}
Conjecture~{\em \ref{conj:interpolation}} holds when $n=2$ or $n=3$.
\end{prop}

In a similar vein, we conjecture:

\begin{conj}
\label{conj:closed point}
Assuming that $\LLcheck(\{\rho_v\}_{v \in S})$ exists, 
there is a $G$-equivariant $k$-linear surjection:
$$\bigotimes_{v \in S} \LLbarcheck(\rhobar) \rightarrow k \otimes_A \LLcheck(\{\rho_v\}_{v \in S}).$$
\end{conj}
This conjecture is verified in~\cite{He} for $n=2$ and $p$ odd.

One can also describe the behavior of $\LLcheck(\{\rho_v\}_{v \in S})$ under
base change.
Suppose that $B$ is another ring satisfying Condition~\ref{cond:A},
and that $f: A \rightarrow B$ is a local homomorphism.
If we are given a Galois
representation $\rho_v: G_{E_v} \rightarrow \GL_n(A)$ 
where $v$ does not lie over $p$,
then we may then apply the preceding considerations
to the Galois representations $B\otimes_A \rho_v$.
The following proposition relates
$\LLcheck(\rho_v)$
and
$\LLcheck(B\otimes_A \rho_v)$.
(We again postpone the proof to the following subsection.)

\begin{prop}
\label{prop:tensor}
Suppose that for every minimal prime of $\Spec B$, 
its image ${\mathfrak p}$ in $\Spec A$ is contained in
a minimal prime ${\mathfrak a}$ of $A$ such that
$\rho_{v,\mathfrak a}$ is a minimal lift of $\rho_{v,\mathfrak p}$.
{\em (}For example, suppose that each component of $\Spec B$
dominates a component of $\Spec A$.{\em )}
Then if $\LLcheck(\{\rho_v\}_{v \in S})$ exists,
so does
$\LLcheck(\{B\otimes_A \rho_v\}_{v \in S})$,
and there is a natural surjection:
$B\otimes_A \LLcheck(\{\rho_v\}_{v \in S}) \rightarrow \LLcheck(\{B \otimes_A \rho_v\}_{v \in S})$.
\end{prop}

We now give some examples illustrating
Definition~\ref{def:family}.

\begin{example}
{\em
Suppose that $A = \mathcal O$ is the ring of integers in a finite extension
$\CK$ of $\Q_p$.  If $\rho:G_{E_v} \rightarrow \GL_n(\mathcal O)$
is continuous
(for some place $v$ of $E$ that does not lie over $p$),
write $\rho_{\CK} := \CK \otimes_{\mathcal O} \rho.$
Suppose that $\LL(\rho_{\CK})$ is absolutely irreducible.
Then $\LLcheck(\rho)$ exists, and is the smooth contragredient
to the lattice $\LL(\rho_{\CK})^{\circ}$ of Proposition~\ref{prop:lattice}.
}
\end{example}

\begin{remark}
{\em
Suppose given $A$ as in Theorem~\ref{thm:family},
and a continuous representation
$\rho:G_{\Q_{\ell}} \rightarrow \GL_2(A)$ for some $\ell \neq p$.
Consider a point $\mathfrak p \in \Spec A[\dfrac{1}{p}]$.
If $\rho_{\mathfrak p}$ is not of the form $\chi \oplus \abs\chi$
for some character $\chi$ of $G_{\Q_{\ell}}$, then for {\em any}
minimal prime ${\mathfrak a}$ of $A$, containing ${\mathfrak p}$,
$\rho_{\mathfrak a}$ is necessarily a minimal lift of
$\rho_{\mathfrak p}$, and so (assuming that $V := \LLcheck(\rho)$ exists),
the surjection of Theorem~\ref{thm:interpolation1} 
is an isomorphism; that is, $V_{\mathfrak p}$ is isomorphic
to $\LLcheck(\rho_{\mathfrak p})$.
On the other hand if $\rho_{\mathfrak p}$ does have the form
$\chi \oplus \abs\chi$, then there exist non-minimal lifts of
$\rho_{\mathfrak p}$, and so $V_{\mathfrak p}$
need not {\em a priori} be isomorphic to $\LLcheck(\rho_{\mathfrak p})$
We now give an example showing that it can indeed happen that 
$V_{\mathfrak p}$ is not isomorphic to $\LLcheck(\rho_{\mathfrak p})$.
}
\end{remark}

\begin{example}
{\em
Suppose that $\ell$ and $p$ are distinct, and that $\ell \not\equiv 1 \bmod p$.
If $A$ is as in Theorem~\ref{thm:family},
then $\Ext^1_{\Z_{p}[G_{\Q_{\ell}}]}(\abs, 1)$ is free of rank $1$ over $A$.
Let $c$ denote a generator of this $\Ext^1$-module, and for any $a \in A$,
let $\rho_a:G_{\Q_{\ell}} \rightarrow \GL_2(A)$ be the rank two 
representation underlying $a \cdot c$.  
One checks that if $a$ is a regular element, then $V := \LLcheck(\rho_a)$ exists,
and in fact
is isomorphic to $\St_A$ (the Steinberg representation of $\GL_2(\Q_{\ell})$
with coefficients in $A$); in particular, it is independent of the
regular element $a$.
Note that if $a \in \mathfrak p \in \Spec A[\dfrac{1}{p}]$ (i.e.\ the regular 
function associated to $a$ vanishes at $\mathfrak p$),
then $\rho_{a,\mathfrak p} := \kappa(\mathfrak p)\otimes \rho_a$
is split and hence unramified,
and thus $\LLcheck(\rho_{a,\mathfrak p})$ is a non-split extension of
Steinberg by trivial.  In particular, at such a point $\mathfrak p$,
$V_{\mathfrak p}$ fails to be isomorphic to $\LLcheck(\rho_{a,\mathfrak p})$.
}
\end{example}

%

We conclude with a ``recognition theorem'' that is useful for verifying that
a given $A[G]$-module is isomorphic to $\LLcheck(\{\rho_v\}_{v \in S})$.  As with the other
results of this section, we defer its proof to the next subsection.
\begin{theorem}
\label{thm:verify}
Let $V$ be an admissible smooth $A[G]$-module, such that the smooth dual
of $V/{\mathfrak m}V$ is essentially AIG, and suppose that
there exists a Zariski dense subset $\Sigma$ of $\Spec A[\dfrac{1}{p}]$
such that:
\begin{enumerate}
\item For all $v \in S$, and all ${\mathfrak p}$ in $\Sigma$, there exists a
minimal prime ${\mathfrak a}$ of $A$ such that $\rho_{v, {\mathfrak a}}$
is a minimal lift of $\rho_{v, {\mathfrak p}}$.
\item For each point $\mathfrak p$ of $\Sigma$ there exists an isomorphism:
$$\kappa(\mathfrak p) \otimes_A V \iso \bigotimes_{v \in S} \LLcheck(\rho_{v,{\mathfrak p}}).$$
\item The diagonal map:
$$V \rightarrow \prod_{\mathfrak p \in \Sigma} 
\bigotimes_{v \in S} \LLcheck(\rho_{v,\mathfrak p}).$$
is an injection.
\end{enumerate}
Then $V$ satisfies conditions {\em (1)}, {\em (2)}, and {\em (3)} of
Theorem~{\em \ref{thm:family}};
that is, $\LLcheck(\{\rho_v\}_{v \in S})$ exists and is isomorphic to $V$.
\end{theorem}

\subsection{The proofs of Theorem~\ref{thm:family} and some related results}
\label{subsec:proofs}
We develop a series of deductions involving the various conditions of
Theorem~\ref{thm:family}.  These will be used not only to prove
Theorem~\ref{thm:family}, and the other outstanding results
from the preceding subsection, but also to provide a criterion for
verifying the conditions of Theorem~\ref{thm:family}, which
will be useful in applications.

If $A$ is a local ring with residue field $K$, and $V$ is a
representation of $F$ over $A$, we let $\Vbar$ denote the 
representation $V \otimes_A K$.

\begin{lemma}
\label{lem:cyclic}
Let $A$ be a Noetherian $W(k)$-algebra that is a
local ring with residue field $K$.
If $V$ is an admissible smooth representation of $G$ over $A$,
then $\Vbar^{(n)}$ is $1$-dimensional over $K$
if and only if $V^{(n)}$ is a cyclic $A$-module.
\end{lemma}
\begin{proof}
Theorem~\ref{thm:fgtensor} shows that
if $\Vbar^{(n)}$ is one-dimensional then
$V^{(n)}$ is a finitely generated
$A$-module, and so the lemma 
follows from Nakayama's lemma together with the isomorphism 
$V^{(n)} \otimes_A K \iso \Vbar^{(n)}$.
\end{proof}

\begin{lemma}
\label{lem:Kirillov generator}
Let $A$ be a Noetherian $W(k)$-algebra that is a local ring
with residue field $K$
in which all $\ell_i$ are invertible.
If $V$ is an admissible smooth representation
of $G$ over $A$,
then the following are equivalent:
\begin{enumerate}
\item For any non-zero quotient $K[G]$-module $\Wbar$ of $\Vbar$,
one has $\Wbar^{(n)} \neq 0$.
\item For any non-zero quotient $A[G]$-module $W$ of $V$,
one has $W^{(n)} \neq 0$.
\item $\J(\Vbar)$ generates $\Vbar$ over $K[G]$.
\item $\J(V)$ generates $V$ over $A[G]$.
\end{enumerate}
\end{lemma}
\begin{proof}
It is clear that (2) implies (1), as any quotient of $\Vbar$ is
also a quotient of~$V$.  Suppose that (1) holds and that
$W$ is a quotient of $V$ with $W^{(n)} = 0$.  Then $\Wbar^{(n)} = 0$,
and so~(1) implies that $\Wbar = 0$.  Then $W = 0$ by Nakayama's Lemma,
so (1) implies~(2).

If $\Wbar$ is a quotient of $\Vbar$, then
we have that $\Wbar^{(n)}$ 
is a quotient of $\Vbar^{(n)}$ 
since the derivative functor is exact.   If we let $\overline{U}$
denote the $K[G]$-submodule of $\Vbar$
generated by $\J(\Vbar)$,  we see that 
$\Wbar^{(n)}$ vanishes if and only if $\Wbar$
is a quotient of $\Vbar/\overline{U}$. 
Thus~(1) and~(3) are equivalent.

Clearly~(4) implies~(3), since $\J(\Vbar)$ is the image of
$\J(V)$ in $\Vbar$.  Conversely, suppose $\J(\Vbar)$ generates
$\Vbar$ over $K[G]$.  Since $\J(V)$ maps surjectively onto $\J(\Vbar)$,
Lemma~\ref{lem:nakayama} implies that $\J(V)$ generates $V$ over $A[G]$.
\end{proof}

\begin{lemma}
\label{lem:free}
Let $A$ be a local ring satisfying Condition~{\em \ref{cond:A}}.
If $V$ is an admissible
smooth representation of $G$ over $A$,
and if $V_{\mathfrak a}^{(n)}$ is nonzero 
for each minimal prime $\mathfrak a$ of~$A$,
then if $V^{(n)}$ is a cyclic $A$-module,
it is in fact free of rank $1$ over $A$.
\end{lemma}
\begin{proof}
Since $(V^{(n)})_{\mathfrak a} = (V_{\mathfrak a})^{(n)}$, 
we have that $(V^{(n)})_{\mathfrak a}$ is nonzero
for all $\mathfrak a$.  Our hypotheses on $A$ imply
that $A$ injects into the
product of the fields $A_{\mathfrak a}$; it follows
that the annihilator of $V^{(n)}$ in $A$ is the zero ideal.
The lemma follows.
\end{proof}

\begin{prop}
\label{prop:second properties}
Let $A$ be a local ring satisfying Condition~{\em \ref{cond:A}},
let $V$ be an admissible smooth representation of $G$
over $A$, and
suppose that $(V_{\mathfrak a})^{(n)}$ is nonzero for each minimal
prime $\mathfrak a$ of $A$,
that $\Vbar^{(n)} $ is one-dimensional,
and that for any non-zero quotient $k[G]$-module $\Wbar$ of $\Vbar$,
one has $\Wbar^{(n)} \neq 0$.
Then:
\begin{enumerate}
\item
$V^{(n)}$ is free of rank $1$ over $A$.
\item
$\J(V)$ generates $V$ over $A[G]$.
\item
$\End_{A[G]}(V) = A$.
\end{enumerate}
\end{prop}
\begin{proof}
The first claim follows immediately from Lemmas~\ref{lem:cyclic}
and~\ref{lem:free}.  The second is a consequence of 
Lemma~\ref{lem:Kirillov generator}.

By Proposition~\ref{prop:schwartz endomorphisms},
the natural map $A \rightarrow \End_{A[P_n]}(\J(V))$
is an isomorphism.  This latter map factors as the composition
$$A \rightarrow \End_{A[G]}(V) \rightarrow \End_{A[P_n]}(\J(V)),$$
and restriction of endomorphisms from $V$ to $\J(V)$ is injective
because $\J(V)$ generates $V$.  Thus $\End_{A[G]}(V) = A$.
\end{proof}

The following result gives some equivalent formulations of the hypotheses on
$\Vbar$ appearing in the preceding proposition.

\begin{lemma}
\label{lem:Vbar Kirillov}
Let $K$ be a field in which all $\ell_i$ are invertible. 
If $\Vbar$ is an admissible smooth representation of 
$G$ over $K$, then the following are equivalent:
\begin{enumerate}
\item $\Vbar^{(n)}$ is one-dimensional,
and for any non-zero quotient $K[G]$-module $\Wbar$ of~$\Vbar$,
one has $\Wbar^{(n)} \neq 0$ {\em (}and hence 
$\Vbar^{(n)}$ is isomorphic to
$\Wbar^{(n)}$, so that $\Wbar^{(n)}$ 
is again one-dimensional{\em )}.
\item $\Vbar^{(n)}$ is one-dimensional, and $\J(\Vbar)$
generates $\Vbar$ over $K[G]$.
\item $\Vbar$ is of finite length {\em (}and hence has a cosocle{\em )},
$\cosoc(\Vbar)$ is absolutely irreducible, and 
$\Vbar^{(n)}$ is isomorphic to $\bigl(\cosoc(\Vbar)\bigr)^{(n)}$, 
with both being non-zero.
\item The smooth $K$-dual $\Vbar^{\vee}$ of $\Vbar$ is essentially AIG.
\end{enumerate}
\end{lemma}
\begin{proof}
The equivalence of conditions~(1) and~(2) follows from
Lemma~\ref{lem:Kirillov generator} (applied with $A = K$ and $V = \Vbar$).
If condition~(2) holds, then $\Vbar$ is finitely generated over $K[G]$,
and hence of finite length.  Write 
$\cosoc(\Vbar) = \bigoplus_j \Wbar_j$,
where each $\Wbar_j$ is irreducible.
Condition~(1) (which also holds, since it is equivalent
to condition~(2), as we have already observed) shows that $\Wbar_j^{(n)}$
is one-dimensional for each $j$.  Since the composition:
$$\Vbar^{(n)} \rightarrow \bigl(\cosoc(\Vbar)\bigr)^{(n)} 
\iso \bigoplus_j \Wbar_j^{(n)}$$
is surjective (the derivative functor is exact), we see that in fact
there is only one summand,
and hence that $\cosoc(\Vbar)$ is irreducible. 
Proposition~\ref{prop:second properties} (applied
with $A = K$ and $V = \Wbar$) then implies that
$\End_G\bigl(\cosoc(\Vbar)\bigr)) = K,$
and hence that $\cosoc(\Vbar)$ is in fact absolutely irreducible.
Thus~(2) implies~(3).

If condition~(3) holds, then by assumption $\bigl(\cosoc(\Vbar)\bigr)^{(n)}$
is nonzero.  It is therefore one-dimensional by Theorem~\ref{thm:kirillov tensor}.
Thus $\Vbar^{(n)}$ is one-dimensional,
giving the first half of condition~(1).    The second half of condition~(1)
follows from the fact that $\cosoc(\Vbar)$ is irreducible, and satisfies
$\bigl(\cosoc(\Vbar)\bigr)^{(n)} \neq 0$.  Thus~(3) implies~(1).

Now consider the smooth dual $\Vbar^{\vee}$.  If $\Vbar$ is finite length,
the socle of $\Vbar^{\vee}$ is the smooth dual of the cosocle of $\Vbar$.  In
particular $\soc(\Vbar^{\vee})$ is absolutely
irreducible and generic if and only if $\cosoc(\Vbar)$ is.  Moreover, the 
map 
$\bigl(\soc(\Vbar^{\vee})\bigr)^{(n)} \rightarrow (\Vbar^{\vee})^{(n)}$ 
is dual to the map 
$\Vbar^{(n)} \rightarrow \bigl(\cosoc(\Vbar)\bigr)^{(n)}$
so that one is an isomorphism if and only if the other is.
Thus~(3) is equivalent to~(4).  
\end{proof}

\begin{lemma} \label{lem:tf}
If $A$ is a reduced Noetherian $W(k)$-algebra, and if $V$ be an admissible $A[G]$-module,
then the following are equivalent:
\begin{enumerate}
\item $V$ is $A$-torsion free,
i.e.\ every associated prime of $V$ is a minimal prime of~$A$.
\item The natural map $V \rightarrow V \otimes_A \CK(A)$ is an injection, where $\CK(A)$
is the product over minimal primes ${\mathfrak a}$ of $A$ of the fields
$A_{\mathfrak a}$.
\item The natural map 
$$V \rightarrow \prod_{\mathfrak a} (V/{\mathfrak a}V)^{\tf}$$
is injective, where ${\mathfrak a}$ runs over the minimal primes of $A$
and $(V/{\mathfrak a}V)^{\tf}$ is the maximal $A/{\mathfrak a}$-torsion free
quotient of $V/{\mathfrak a}V$.
\end{enumerate}
If these equivalent conditions hold, then for any Zariski dense set of
primes $\Sigma$ of $\Spec A$, the map
$$V \rightarrow \prod_{\mathfrak p \in \Sigma} V \otimes_A \kappa({\mathfrak p})$$
is injective.
\end{lemma}
\begin{proof}
If $x$ lies in the kernel of the map $V \rightarrow V \otimes_A \CK(A)$, then the annihilator
of $x$ is not contained in any minimal prime of $A$.  There then exists a multiple of $x$
whose annihilator is prime; this prime cannot be a minimal prime of $A$ and is therefore 
a non-minimal associated
prime of $V$.  Conversely, if there exists a non-minimal associated prime ${\mathfrak p}$ of $V$,
there is an element $x$ of $V$ whose annihilator is ${\mathfrak p}$; then $x$ maps to zero
in $V \otimes_A A/{\mathfrak a}$ for every minimal prime ${\mathfrak a}$ of $A$.  Thus (1)
and (2) are equivalent.

Note that the inclusion $V \rightarrow V \otimes_A \CK(A)$ factors as the composite:
$$V \rightarrow \prod_{\mathfrak a} (V/{\mathfrak a}V)^{\tf} \rightarrow V \otimes_A \CK(A) = \prod_{\mathfrak a}
V \otimes_A \kappa({\mathfrak a}),$$
and the second map is always injective.  It thus follows that (2) and (3) are equivalent.

Now let $\Sigma$ be a Zariski dense set of primes of $A$, suppose that the equivalent
conditions (1), (2), and (3) hold, and suppose that $x$ is an element of $V$
that maps to zero in $V \otimes_A \kappa({\mathfrak p})$ for all ${\mathfrak p} \in \Sigma$.
Choose a compact open
subgroup $U$ of $G$ fixing $x$; then $V^U$ is finitely generated over $A$, and
$x$ maps to zero in $V^U \otimes_A \kappa({\mathfrak p})$ for all $\mathfrak p$
in $\Sigma$.  It follows that the support of $x$ (considered as an element of $V^U$)
is a closed subset of $\Spec A$ contained in the complement of $\Sigma$.  In particular
that the annihilator of $x$ is not contained in any minimal prime of $A$, contradicting
condition~(1).
\end{proof}

\begin{lemma} \label{lem:dense}
Let $A$ be a reduced Noetherian $W(k)$-algebra, and let $V_1$ and $V_2$
be two admissible smooth $A[G]$-modules such that:
\begin{enumerate}
\item For each $i$, $V_i^{(n)}$ is free of rank one over $A$.
\item For each $i$, $V_i$ is generated by $\J(V_i)$ as an $A[G]$-module.
\item There exists a Zariski dense set $\Sigma$ of primes of $A$ such that
for all ${\mathfrak p} \in \Sigma$, $(V_1)_{\mathfrak p}$ is isomorphic to
$(V_2)_{\mathfrak p}$, as $A_{\mathfrak p}[G]$-modules.
\item The natural map:
$$V_i \rightarrow \prod_{\mathfrak p \in \Sigma} V_i \otimes_A \kappa(\mathfrak p)$$
is injective for each $i$.  {\em (}This is automatic if $V_i$ is $A$-torsion free.{\em )}
\end{enumerate}
Then there is an $A$-linear $G$-equivariant isomorphism $V_1 \iso V_2$.
\end{lemma}
\begin{proof}
Let $\CK'$ be the product over ${\mathfrak p} \in \Sigma$ of the residue fields
$\kappa({\mathfrak p})$.  Condition 3) gives us an isomorphism
$V_1 \otimes_A \CK' \iso V_2 \otimes_A \CK'$.  Moreover, by Condition (4),
$V_i$ embeds in $V_i \otimes_A \CK'$ for each $i$.  We may thus regard
$V_1$ and $V_2$ as submodules of $V_1 \otimes_A \CK'$.

By (1), $(V_1 \otimes_A \CK')^{(n)}$ is free of rank one over $\CK'$,
and $V_i^{(n)}$ is a free $A$-submodule of $(V_1 \otimes_A \CK')^{(n)}$
for each $i$.  There thus exists an element $c$ of $(\CK')^{\times}$
such that $cV_2^{(n)}$ and $V_1^{(n)}$ coincide as submodules of
$(V_1 \otimes_A \CK')^{(n)}$.  It follows that
$c\J(V_2)$ and $\J(V_1)$ coincide as submodules
of $\J(V_1 \otimes_A \CK')$.  Since
$V_1$ and $V_2$ are generated by $\J(V_1)$ and $\J(V_2)$ over $A[G]$,
we must have $V_1 = cV_2$; in particular $V_1$ and $V_2$ are isomorphic.
\end{proof}

We can now prove the uniqueness claim of Theorem~\ref{thm:family}. 

\begin{prop}
\label{prop:unique}
Let $A$ be a local ring satisfying Condition~{\em \ref{cond:A}}, and let
$V_1$ and $V_2$ be two admissible smooth $A[G]$-modules.
Suppose that:
\begin{enumerate}
\item The $V_i$ are $A$-torsion free.
\item For each minimal prime ${\mathfrak a}$ of $A$, 
$(V_i)_{\mathfrak a}^{(n)}$
is nonzero.
\item For each $i$, $\Vbar_i$ satisfies the equivalent conditions of Lemma~{\em \ref{lem:Vbar Kirillov}}.
\item For each minimal prime $\mathfrak a$ of $A$, there is
a $G$-equivariant isomorphism $(V_1)_{\mathfrak a} \iso (V_2)_{\mathfrak a}$.
\end{enumerate}
Then there is an $A$-linear $G$-equivariant isomorphism
$V_1 \cong V_2$
{\em (}which, by part~{\em (3)} of Proposition~{\em \ref{prop:second properties}},
is uniquely determined up to multiplication by an element of~$A^{\times}${\em )}.
\end{prop}
\begin{proof}
By part (1) of Proposition~\ref{prop:second properties}, we have that
$V_i^{(n)}$ is free of rank $1$ over $A$ for each $i$.  As the minimal primes
of $A$ are dense in $\Spec A$, it thus follows
by Lemma~\ref{lem:dense} that $V_1$ is isomorphic over $A[G]$ to $V_2$.  
\end{proof}

The purpose of our next collection of results, which are rather 
technical, is to allow us to make a tensor factorization
in the context of Theorem~\ref{thm:family}, and hence work with
one $E_v$ at a time.

\begin{prop}
\label{prop:factor}
Let $K$ be a field in which all $\ell_i$ are invertible.
If $\Vbar$ is an admissible smooth representation of $G$ over $K$
satisfying the equivalent conditions of Lemma~{\em \ref{lem:Vbar Kirillov}},
there exist admissible smooth representations
$\Vbar_v$ of $G_v$ {\em (}$v \in S${\em )}, each individually satisfying
the equivalent conditions of Lemma~{\em \ref{lem:Vbar Kirillov}} 
{\em (}with $G$ replaced by $G_v${\em )},
together with a $G$-equivariant surjection
$\bigotimes_v \Vbar_v \rightarrow \Vbar$.
\end{prop}
\begin{proof}
We proceed by induction on the cardinality $s$ of $S$. In the case 
when $s = 1$ there
is nothing to prove, and so we assume that $s > 1,$ and write
$S = {v_1, \dots, v_s}$, 
$G' = G_{v_2} \times \cdots\times G_{v_s},$ so that $G = G_{v_1} \times G'$.
Since $\cosoc(\Vbar)$ is absolutely irreducible, there is an isomorphism
$\cosoc(\Vbar) \iso \pibar_{v_1}\otimes\pibar',$
where $\pibar_{v_1}$ (resp.\ $\pibar'$) is a generic absolutely irreducible
representation of $G_{v_1}$ (resp.\ $G'$).

Since $\Vbar$ is of finite length, we may and do choose a quotient
$\Wbar$ of $\Vbar$ which is maximal with respect to the following
property:
there is a surjective map
$$
\phi: \Vbar_{v_1}\otimes \Vbar' \rightarrow \Wbar,
$$
where $\Vbar_{v_1}$ and $\Vbar'$ each satisfy the equivalent conditions
of Lemma~\ref{lem:Vbar Kirillov} (with respect to $G_{v_1}$ and $G'$
respectively).
Since $\cosoc(\Vbar)$ satisfies these conditions, we see 
that $\Wbar \neq 0.$  Thus $\cosoc(\Vbar)$ is a quotient of $\Wbar$,
and hence $\Vbar^{(n)}$ is isomorphic to $\Wbar^{(n)}$, 
as both are one-dimensional.

Let $\Ubar$ be the kernel of the quotient map $\Vbar \rightarrow \Wbar$,
and suppose that $\Ubar$ is non-zero.  Extending scalars
if necessary, we then may find a non-zero absolutely irreducible quotient
$\thetabar_{v_1}\otimes\thetabar'$ of $\Ubar$, where $\thetabar_{v_1}$
(resp.\ $\thetabar'$) is an absolutely irreducible representation
of $G_{v_1}$ (resp.\ $G'$).   If we let $\Tbar$ denote the kernel of the
quotient map $\Ubar \rightarrow \thetabar_{v_1}\otimes \thetabar'$,
and if we write $\Xbar := \Vbar/\Tbar,$ 
then there is a short exact sequence
\begin{equation}
\label{eqn:1st ses}
 0 \rightarrow \thetabar_{v_1}\otimes\thetabar' \rightarrow \Xbar
\rightarrow \Wbar \rightarrow 0,
\end{equation}
which we may pullback via $\phi$ to obtain a short exact sequence
\begin{equation}
\label{eqn:2nd ses}
 0 \rightarrow \thetabar_{v_1}\otimes\thetabar' \rightarrow \Ybar
\rightarrow \Vbar_{v_1}\otimes\Vbar' \rightarrow 0.
\end{equation}
Applying the $n$-th derivative functor to (\ref{eqn:1st ses}),
we find
(recalling that the surjections $\Vbar^{(n)} \rightarrow 
\Xbar^{(n)} \rightarrow
\Wbar^{(n)})$ are in fact isomorphisms, we obtain an isomorphism:
$$\thetabar_{v_1}^{(n)} \otimes_{K} (\thetabar')^{(n)} \iso
(\thetabar_{v_1}\otimes_K \thetabar')^{(n)} = 0.$$
Hence either $\thetabar_{v_1}^{(n)} = 0$ or $(\thetabar')^{(n)} = 0.$
Also, we conclude that $\thetabar_{v_1}\otimes\thetabar$ cannot be a quotient
of $\Vbar$, and hence cannot be a quotient of $\Xbar$.  Thus~(\ref{eqn:1st ses})
is non-split, and hence~(\ref{eqn:2nd ses}) is also non-split (since
$\phi$ is surjective).

The non-split
short exact sequence~(\ref{eqn:2nd ses}) corresponds to a non-trivial element
of $\Ext^1_G(\Vbar_{v_1}\otimes\Vbar',\thetabar_{v_1}\otimes\thetabar'),$ which
by the K\"unneth formula admits the description
\begin{multline*}
\Ext^1_G(\Vbar_{v_1}\otimes \Vbar',\thetabar_{v_1}\otimes\thetabar') \\
\iso
\Hom_{G_{v_1}}(\Vbar_{v_1},\thetabar_{v_1})\otimes \Ext^1_{G'}(\Vbar',\thetabar')
\oplus
\Ext^1_{G_{v_1}}(\Vbar_{v_1},\thetabar_{v_1})\otimes \Hom_{G'}(\Vbar',\thetabar').
\end{multline*}
Now by assumption, the $n$th derivative (as a $\GL_n(E_{v_1})$-module)
of any non-zero quotient of $\Vbar_{v_1}$ is a non-zero space, while
the $n$th derivative (as a $G'$-module) of any non-zero quotient
of $\Vbar'$ is a non-zero space.
Thus if $(\thetabar_{v_1})^{(n)} = 0,$ then $\Hom_{G_1}(\Vbar_{v_1},\thetabar_{v_1}) = 0,$
and thus $\Ybar$ corresponds to a non-trivial element of the tensor product
$\Ext^1_{G_1}(\Vbar_{v_1},\thetabar_{v_1})\otimes \Hom_{G'}(\Vbar',\thetabar').$
Concretely, this means we may form a non-trivial extension $E_1$ of $\Vbar_{v_1}$ by
$\thetabar_{v_1}$, and find a non-zero map $\psi:\Vbar' \rightarrow \thetabar'$
(which is then surjective, since $\thetabar'$ is irreducible),
so that $\Ybar$ is obtained as the pushforward of $E_1\otimes \Vbar'$
via the map
$$
\id\otimes\psi:
\thetabar_{v_1}\otimes\Vbar' \rightarrow \thetabar_{v_1}\otimes\thetabar'.
$$
Thus $\Ybar,$ and hence $\Xbar$, is a quotient of $E_1\otimes \Vbar'$,
contradicting the maximality of $\Wbar$.
If instead we had $(\thetabar')^{(n)} = 0,$ then we would similarly conclude
that $\Xbar$ may be written as a quotient of $\Vbar_1\otimes E_2,$
for some non-trivial extension $E_2$ of $\Vbar'$ by $\thetabar'$,
again contradicting the maximality of $\Wbar$.

From these contradictions we conclude that in fact $\Ubar = 0,$
and thus that $\Vbar = \Wbar$.
Thus we may write $\Vbar$ as a quotient of $\Vbar_{v_1}\otimes \Vbar'$
as above.
Applying the inductive hypothesis to $\Vbar'$, the proposition follows.
\end{proof}

\begin{cor}
\label{cor:factor}
If $V$ is an admissible smooth representation of $G$ over a local ring $A$
satisfying Condition~{\em \ref{cond:A}},
such that $\Vbar:=V/\mathfrak m V$ satisfies the equivalent conditions of
Lemma~{\em \ref{lem:Vbar Kirillov}},
and if $S' \subset S$ is any subset,
then the $G_{S'}$-representation 
$V^{(n),S\setminus S'}/{\mathfrak m} V^{(n),S \setminus S'}$ satisfies 
the conditions of
Lemma~{\em \ref{lem:Vbar Kirillov} (}with respect to $A[G_{S'}]${\em )}.
\end{cor}
\begin{proof}
Choose a surjection
$\bigotimes_{v \in S}\Vbar_v
\rightarrow \Vbar$
satisfying the conditions of the preceding proposition.
Since $\Vbar_v^{(n)}$ is one-dimensional for each $v$ (and so in particular
for each $v \in S \setminus S'$),
applying the exact functor 
$\Vbar \mapsto \Vbar^{(n),S \setminus S'}$
yields a surjection
$$\bigotimes_{v \in S'} \Vbar_v
\iso
\bigl(\bigotimes_{v \in S}\Vbar_v\bigr)^{(n),S \setminus S'}
\rightarrow \Vbar^{(n),S \setminus S'} \iso 
V^{(n),S \setminus S'}/{\mathfrak m} V^{(n),S \setminus S'}.$$
The lemma follows.
\end{proof}

We now return to the setting of the previous subsection.  That is,
for each $v$ is $S$ we are given a representation $\rho_v: G_{E_v} \rightarrow \GL_n(A)$.
The above results allow us to establish Theorem~\ref{thm:family}
and Proposition~\ref{prop:individual} more or less immediately.

\medskip
\noindent
{\em Proof of Theorem~{\ref{thm:family}}}.
Suppose we have $V_1$, $V_2$ satisfying conditions (1), (2), and (3) of
Theorem~\ref{thm:family}.  Then for all minimal primes
${\mathfrak a}$ of $A$, we have a $\kappa({\mathfrak a})$-linear
$G$-equivariant isomorphism $(V_1)_{\mathfrak a} \iso (V_2)_{\mathfrak a}$.
Thus $V_1$ and $V_2$ satisfy all of the hypotheses of Proposition~\ref{prop:unique},
and are therefore isomorphic.  Moreover, $V_1$ is cyclic as a $A[G]$-module
by Lemma~\ref{lem:cyclic}.
Finally, $\End_{A[G]}(V_1)$ is isomorphic to $A$ by Proposition~\ref{prop:second properties}.
\qed

\medskip
\noindent
{\em Proof of Proposition~{\ref{prop:individual}}}.
Suppose that for each $v$, we have a representation $\LLcheck(\rho_v)$
satisfying conditions (1), (2), and (3) of Theorem~\ref{thm:family} for $\rho_v$.
Then it is clear that the maximal $A$-torsion free part of the
tensor product over all $v$ of $\LLcheck(\rho_v)$
satisfies the conditions of Theorem~\ref{thm:family} for the collection
$\{\rho_v\}$.

Conversely, suppose we have a representation $\LLcheck(\{\rho_v\}_{v \in S})$
satisfying the hypotheses of Theorem~\ref{thm:family} for
the collection $\{\rho_v\}$.  Then 
for any minimal prime ${\mathfrak a}$ of $A$,
we have an isomorphism:
$$\LLcheck(\{\rho_v\}_{v \in S}) \otimes_A \kappa({\mathfrak a}) \iso \bigotimes_{v \in S}
\LLcheck(\rho_{v, {\mathfrak a}}).$$
Fixing a place $v$, and taking derivatives at all $v' \neq v$, we obtain an isomorphism:
$$(\LLcheck(\{\rho_v\}_{v \in S}))^{(n), S \setminus \{v\}}_{\mathfrak a} 
\iso \LLcheck(\rho_{v, {\mathfrak a}}).$$
Moreover $(\LLcheck(\{\rho_v\}_{v \in S}))^{(n), S \setminus \{v\}}$ is $A$-torsion free,
and (by Corollary~\ref{cor:factor}) satisfies condition (3) of Theorem~\ref{thm:family}.
Thus $\LLcheck(\{\rho_v\}_{v \in S})^{(n), S \setminus \{v\}}$ is isomorphic to
$\LLcheck(\rho_v)$ (so in particular the latter exists).  
\qed

\medskip

We now turn to Theorems~\ref{thm:interpolation1} and~\ref{thm:interpolation2}.  Once we have established
these, Proposition~\ref{prop:tensor} will be an easy consequence.
We first need the following lemma:
\begin{lemma}
\label{lemma:decompose}
Let $A$ be a normal $\Q_p$-algebra that is an integral domain with field of fractions $\CK$,
and let $(\rho',N)$ be a Frobenius-semisimple
Weil--Deligne representation over $\CK$ that splits {\em (}over $\mathcal K${\em )} as a direct sum
of absolutely indecomposable Weil--Deligne representations $\Sp_{\rho_i,n_i}$.  There exist characters
$\chi_i: W_E \rightarrow A^{\times}$ such that $\rho_i \otimes_K \chi_i$ is defined over
a finite extension $\mathcal K_0$ of $\Q_p$ contained in $A$.
\end{lemma}
\begin{proof}
By Lemma~\ref{lem:twist2} we know that such characters $\chi_i$ exist with values in $\mathcal K^{\times}$;
it suffices to show that they take values in $A^{\times}$.  Let $\O$ be the localization of $A$
at a height one prime.  Then $\rho' \otimes_A \mathcal K$ is $\O$-integral, so each $\rho_i$ is $\O$-integral
as well.  Thus $\det \rho_i$ is a character with values in $\O^{\times}$.  Since this is true for
all $\O$, $\det \rho_i$ takes values in $A^{\times}$.  Moreover, $(\det \rho_i) \otimes_{\mathcal K} \chi_i$ 
takes values in ${\mathcal K}_0^{\times}$, and ${\mathcal K}_0$ is contained in~$A$, so some power of $\chi_i$ takes
values in $A^{\times}$.  But then $\chi_i$ must take values in $A^{\times}$ as well since $A$ is normal.
\end{proof}

\begin{lemma}
\label{lem:component}
Let $A$ be a local ring satisfying Condition~{\em \ref{cond:A}}, let $\{\rho_v\}$
a collection of representations $G_{E_v} \rightarrow \GL_n(A)$, and suppose
that $\LLcheck(\{\rho_v\}_{v \in S})$ exists.  Then, for each minimal prime
${\mathfrak a}$ of $A$,
the representation $\LLcheck(\{\rho_v \otimes_A A/{\mathfrak a}\}_{v \in S})$ exists,
and is isomorphic to the maximal $A/{\mathfrak a}$-torsion free part
of $\LLcheck(\{\rho_v\}_{v \in S}) \otimes_A A/{\mathfrak a}$.
\end{lemma}
\begin{proof}
It is straightforward to see that 
$\LLcheck(\{\rho_v\}_{v \in S}) \otimes_A A/{\mathfrak a}$ satisfies conditions 
(2) and (3) of Theorem~\ref{thm:family}, so its maximal $A/{\mathfrak a}$-torsion
free quotient does as well.  This quotient also satisfies condition (1)
of Theorem~\ref{thm:family} by construction.
\end{proof}

\medskip
\noindent
{\em Proof of Theorem~{\ref{thm:interpolation1}}}.
By Proposition~\ref{prop:individual}, it suffices to consider the case
when $S$ has only one element $v$.  By Lemma~\ref{lem:component} we may
assume $A$ is a domain with field of fractions $\mathcal K$.  
Fix an
algebraic extension $\mathcal K'$ of $\mathcal K$ such that
$\mathcal K'$ contains a square root of $\ell$, where $\ell$ 
is the residue characteristic of $v$,
and such that the Frobenius-semisimple Weil--Deligne representation associated
to $\rho_v \otimes_A {\mathcal K}'$ splits as a direct sum of absolutely indecomposable Weil--Deligne
representations $\Sp_{\rho_i,N_i}$ over $\mathcal K'$.  Let $\mathcal K_0$ be the maximal subfield of $\mathcal K'$
that is algebraic over $\Q_p$.

Let $A'$ be the integral closure of $A_{\mathfrak p}$ in $\mathcal K'$, and let
${\mathfrak p}'$ be a prime ideal of $A'$ over~$\mathfrak p$.  Then, by
Lemma~\ref{lemma:decompose}, there exist characters $\chi_i$, with values
in $(A'_{{\mathfrak p}'})^{\times}$, such that for each $i$,
$\rho_i \otimes \chi_i$ is defined over $\mathcal K_0$.

Let $\pi_i$ be the admissible representation of $\GL_n(E_v)$ over $\mathcal K_0$ that corresponds
to $\rho_i \otimes \chi_i$ under the unitary local Langlands correspondence.
Then for each $i$, $(\St_{\pi_i,n_i} \otimes_{\mathcal K_0} \mathcal K') \otimes (\chi_i \circ \det)$
corresponds to $\rho_i$ under unitary local Langlands.  Without loss of generality,
we assume that the representations $\pi_i$ are ordered so that for all $i < j$,
$(\St_{\pi_i,n_i} \otimes_{\mathcal K_0} \kappa({\mathfrak p}')) \otimes \chi_i$ does
not precede 
$(\St_{\pi_j,n_j} \otimes_{\mathcal K_0} \kappa({\mathfrak p}')) \otimes \chi_j$.  (It is then
also true that
$(\St_{\pi_i,n_i} \otimes_{\mathcal K_0} \mathcal K') \otimes \chi_i$ does
not precede 
$(\St_{\pi_j,n_j} \otimes_{\mathcal K_0} \mathcal K') \otimes \chi_j$.)
Let $M$ be the smooth $A'_{{\mathfrak p}'}$-linear dual of the module:
$$(\abs \circ \det)^{-\frac{n-1}{2}}\Ind_Q^{\GL_n(E_v)} \bigotimes_i 
\bigl[(\St_{\pi_i,n_i} \otimes_{\mathcal K_0} A'_{{\mathfrak p}'}) \otimes \chi_i\bigr],
$$
where $Q$ is a suitable parabolic subgroup of $\GL_n(E_v)$.  Then, by construction,
$M \otimes_{A'_{{\mathfrak p}'}} \mathcal K'$ is isomorphic to $\LLcheck(\rho \otimes_A \mathcal K').$
Moreover, because of our assumptions on the ordering of the $\pi_i$, the smooth
$\kappa({\mathfrak p}')$-dual of $M/{\mathfrak p}'M$ is essentially AIG by
Corollary~\ref{cor:essential}, and hence $\J(M)$ generates $M$ as
an $A'[G]$-module.  Moreover $M^{(n)}$ is free of rank one over $A$
by Corollary~\ref{cor:leibnitz}.
Finally, $M$ is $A'_{{\mathfrak p}'}$-torsion
free by construction (in fact, $M$ is free over $A'_{{\mathfrak p}'}$.)
Thus by Lemma~\ref{lem:dense}, $M$ is isomorphic to 
$\LLcheck(\rho_v) \otimes_A A'_{{\mathfrak p}'}$.

The injection of Theorem~\ref{thm:specialization} yields a surjection:
$$\LLcheck(\rho_v \otimes_A \kappa({\mathfrak p}')) \rightarrow
M \otimes_{A'_{{\mathfrak p}'}} \kappa({\mathfrak p}')$$
that is an isomorphism if, and only if, $\rho_v \otimes_A K$ is a minimal lift
of $\rho_v \otimes_A \kappa({\mathfrak p})$.  This descends to the desired surjection
$$\LLcheck(\rho_{v, {\mathfrak p}})) \rightarrow
\LLcheck(\rho_v) \otimes_A \kappa({\mathfrak p}).$$
\qed

\medskip
\noindent
{\em Proof of Theorem~{\ref{thm:interpolation2}}}.
As above, it suffices by Proposition~\ref{prop:individual} to
consider the case when $S$ has only one element $v$.  By Lemma~\ref{lem:component},
for each minimal prime ${\mathfrak a}$ of $A$ containing ${\mathfrak p}$ we have a surjection:
$$\LLcheck(\rho_v) \otimes_A \kappa(\mathfrak p) \rightarrow
\LLcheck(\rho_v \otimes_A A/{\mathfrak a}) \otimes_A \kappa(\mathfrak p).$$
Then $W$ is the image of $\LLcheck(\rho_v) \otimes_A \kappa(\mathfrak p)$
in the product
$$\prod_{\mathfrak a} \LLcheck(\rho_v \otimes_A A/{\mathfrak a}) \otimes_A \kappa(\mathfrak p).$$

By Theorem~\ref{thm:interpolation1} we also have surjections:
$$f_{\mathfrak a}: \LLcheck(\rho_{v,{\mathfrak p}}) \rightarrow
\LLcheck(\rho_v \otimes_A A/{\mathfrak a}) \otimes_A \kappa(\mathfrak p)$$
for all minimal primes ${\mathfrak a}$ of $A$ containing ${\mathfrak p}$.
This gives a diagonal map:
$$\LLcheck(\rho_{v,{\mathfrak p}}) \rightarrow
\prod_{\mathfrak a} \LLcheck(\rho_v \otimes_A A/{\mathfrak a}) \otimes_A \kappa(\mathfrak p).$$
Let $W'$ be the image of this map.  It suffices to show that $W'$ is isomorphic to $W$.

The spaces
$W^{(n)}$ and $(W')^{(n)}$
are one-dimensional $\kappa({\mathfrak p})$-subspaces of
$$\prod_{\mathfrak a} \bigl[\LLcheck(\rho_v \otimes_A A/{\mathfrak a}) \otimes_A \kappa({\mathfrak p}) \bigr]^{(n)}$$
that project isomorphically onto each factor.  There thus exists for each $\mathfrak a$ 
a scalar $c_{\mathfrak a}$ in $\kappa({\mathfrak p})^{\times}$ such that $c(W')^{(n)}$
coincides with $W^{(n)}$ as subspaces of
$$\prod_{\mathfrak a} \bigl[\LLcheck(\rho_v \otimes_A A/{\mathfrak a}) \otimes_A \kappa({\mathfrak p}\bigr]^{(n)},$$
where $c$ is the automorphism of this product given by multiplication by
$c_{\mathfrak a}$ on the factor corresponding to ${\mathfrak a}$.

As the $K$-duals of $W \otimes_{\kappa(\mathfrak p)} K$ and $W' \otimes_{\kappa(\mathfrak p)} K$
are essentially AIG, this implies that $W$ and
$cW'$ coincide, and thus $W$ and $W'$ are isomorphic.  
\qed
\medskip

\medskip
\noindent
{\em Proof of Proposition~{\ref{prop:small n interpolation}}}.
As usual, we invoke Proposition~\ref{prop:individual} to reduce to the case where $S$
has a single element $v$.  As in the proof of Theorem~\ref{thm:interpolation2},
let $W$
be the image of $\kappa({\mathfrak p}) \otimes_A \LLcheck(\rho_v)$
under the diagonal map
$$\kappa({\mathfrak p}) \otimes_A \LLcheck(\rho_v) \rightarrow \prod_i
\kappa(\mathfrak p) \otimes_{A/{\mathfrak a}_i} V_i,$$
where ${\mathfrak a}_1, \dots, {\mathfrak a}_i$ are the minimal primes
of $A$ containing ${\mathfrak p}$ and, for each $i$, $V_i$ is the maximal
$A$-torsion free quotient of $\LLcheck(\rho_v) \otimes_A A/{\mathfrak a}_i$.

As $\LLcheck(\rho_v)$ embeds in the product of the $V_i$, every Jordan--H\"older
constituent of $\kappa({\mathfrak p}) \otimes_A \LLcheck(\rho_v)$ is isomorphic
to a Jordan--H\"older constituent of $\kappa({\mathfrak p}) \otimes_{A/{\mathfrak a}_i} V_i$
for some~$i$, and hence to a Jordan--H\"older constituent of $W$.  In particular,
every Jordan--H\"older constituent of the kernel of the map
$$\kappa({\mathfrak p}) \otimes_A \LLcheck(\rho_v) \rightarrow W$$
is a Jordan--H\"older constituent of $\kappa({\mathfrak p}) \otimes_A \LLcheck(\rho_v)$
that appears with multiplicity at least two.  Since the smooth dual of
$\kappa({\mathfrak p}) \otimes_A \LLcheck(\rho_v)$ is essentially AIG,
Corollary~\ref{cor:mult} above shows that no such Jordan--H\"older constituent can exist
when $n=2$ or~$3$.
\qed

\medskip
\noindent
{\em Proof of Proposition~{\ref{prop:tensor}}}.
By Proposition~\ref{prop:individual} we may assume $S$ consists of a single element.
For any minimal prime ${\mathfrak b}$ of $B$, we have by Theorem~\ref{thm:interpolation2}
a surjection:
$$
\LLcheck(\rho_v) \otimes_A \kappa(f^{-1}({\mathfrak b})) \rightarrow
\LLcheck(\rho_{v,f^{-1}({\mathfrak b})})
$$
and hence (after a base change) a surjection:
$$
\LLcheck(\rho_v) \otimes_A \kappa({\mathfrak b}) \rightarrow
\LLcheck(\rho_v \otimes_A \kappa(\mathfrak b)).$$
Let $V$ be the image of the composed map:
$$
\LLcheck(\rho_v) \otimes_A B \rightarrow
\prod_{\mathfrak b} \LLcheck(\rho_v) \otimes_A \kappa({\mathfrak b}) \rightarrow
\prod_{\mathfrak b} \LLcheck(\rho_v \otimes_A \kappa({\mathfrak b})).
$$
One easily verifies that $V$ satisfies conditions (1), (2) and (3) of
Theorem~\ref{thm:family} for the representation $\rho_v \otimes_A B$ over $B$.
\qed

\medskip

We now turn to the proof of Theorem~\ref{thm:verify}.  This will require
several preliminary lemmas.

\begin{lemma}
\label{lem:reduction}
Suppose that Theorem~{\em \ref{thm:verify}} holds when $S$ has only one element.
Then Theorem~{\em \ref{thm:verify}} holds for an arbitrary finite set $S$.
\end{lemma}
\begin{proof}
Suppose we have established Theorem~\ref{thm:verify} in the case in
which $S$ has only one element.  We can then establish the general case
as follows: suppose $V$ satisfies the conditions of
Theorem~\ref{thm:verify} for the collection $\{\rho_v\}$.  If we fix
a place $v \in S$, then $V^{(n), S \setminus \{v\}}$ satisfies
the conditions of Theorem~\ref{thm:verify} for the representation
$\rho_v$.  Thus $V^{(n), S \setminus \{v\}}$
is isomorphic to $\LLcheck(\rho_v)$.  It follows
by Proposition~\ref{prop:individual} that $\LLcheck(\{\rho_v\}_{v \in S})$
exists and is isomorphic to the maximal torsion-free quotient of the 
tensor product of the representations $V^{(n), S \setminus \{v\}}$.

For any prime ${\mathfrak p}$ of $A$ lying over a prime of $\Sigma$,
we have an isomorphism: 
$$V \otimes_{A} \kappa({\mathfrak p}) \iso
\bigotimes_{v \in S} V^{(n), S \setminus \{v\}} \otimes_{A} \kappa({\mathfrak p}).$$
It thus follows by Lemma~\ref{lem:dense} that $V$ and $\LLcheck(\{\rho_v\}_{v \in S})$
are isomorphic, as required.
\end{proof}

\begin{prop}
\label{prop:tf verify}
Let $A$ be a local ring satisfying Condition~{\em \ref{cond:A}}, let $\rho_v$ be
an $n$-dimensional representation of $G_{E_v}$ over $A$, and let
$V$ be an admissible $A[G]$-module such that:
\begin{enumerate}
\item $V$ is torsion free over $A$.
\item The smooth dual of $V/{\mathfrak m}V$ is essentially AIG.
\item There exists a Zariski dense set of primes $\Sigma$ in $\Spec A[\dfrac{1}{p}]$
such that for each prime ${\mathfrak p} \in \Sigma$,
$V \otimes_A \kappa({\mathfrak p})$ is isomorphic to $\LLcheck(\rho_{v, {\mathfrak p}})$,
\end{enumerate}
Then $V$ satisfies conditions {\em (1)}, {\em (2)}, and {\em (3)} of Theorem~{\em \ref{thm:family}} with respect to
$\rho_v$; that is, $\LLcheck(\rho_v)$ exists and is isomorphic to $V$.
\end{prop}
\begin{proof}
Note that conditions (1) and (3) of Theorem~\ref{thm:family} are immediate from the hypotheses.
It thus suffices to construct, for each minimal prime ${\mathfrak a}$ of $A$,
an isomorphism $V_{\mathfrak a} \iso \LLcheck(\rho_{v,\mathfrak a})$.

Fix a minimal prime ${\mathfrak a}$ of $A$.  The map $V \rightarrow (V/{\mathfrak a}V)^{\tf}$
becomes an isomorphism after localizing at any prime ${\mathfrak p}$ of $A$
that contains ${\mathfrak a}$ but no other minimal prime of $A$.  Thus,
replacing $V$ with $(V/{\mathfrak a}V)^{\tf}$, $A$ with $A/\mathfrak a$,
and $\Sigma$ with the set of ${\mathfrak p}$ in $\Sigma$ that contain
${\mathfrak a}$ but no other minimal prime of $A$, we reduce to the case where
$A$ is a domain with field of fractions $\CK$.

Let $\CK'$ be an
algebraic extension of $\CK$ such that
$\CK'$ contains a square root of $\ell$, where $\ell$ is the residue characteristic of $v$,
and such that the Frobenius-semisimple Weil--Deligne representation associated
to $\rho_v \otimes_A \CK'$ splits as a direct sum of absolutely indecomposable Weil--Deligne
representations $\Sp_{\rho_i,N_i}$ over $\CK'$.  Let $\CK_0$ be the maximal subfield of $\CK'$
that is algebraic over $\Q_p$.

Let $A'$ be the integral closure of $A[\dfrac{1}{p}]$ in $K'$.  By
Lemma~\ref{lemma:decompose}, there exist characters $\chi_i$, with values
in $(A'_{{\mathfrak p}'})^{\times}$, such that for each $i$,
$\rho_i \otimes \chi_i$ is defined over $\CK_0$.

Let $\pi_i$ be the admissible representation of $\GL_n(E_v)$ over $K_0$ that corresponds
to $\rho_i \otimes \chi_i$ under the unitary local Langlands correspondence.
Then for each $i$, $(\St_{\pi_i,n_i} \otimes_{\CK_0} \CK') \otimes (\chi_i \circ \det)$
corresponds to $\rho_i$ under unitary local Langlands.  Without loss of generality,
we assume that the representations $\pi_i$ are ordered so that for all $i < j$,
$(\St_{\pi_i,n_i} \otimes_{\CK_0} \CK') \otimes (\chi_i \circ \det)$ does
not precede 
$(\St_{\pi_j,n_j} \otimes_{\CK_0} \CK') \otimes (\chi_j \circ \det)$.  Then, for all ${\mathfrak p}$ in
an open dense subset $U_1$ of $\Spec A'$, and all $i < j$,
$(\St_{\pi_i,n_i} \otimes_{\CK_0} \kappa({\mathfrak p})) \otimes (\chi_i \circ \det)$ does not precede
$(\St_{\pi_j,n_j} \otimes_{\CK_0} \kappa({\mathfrak p})) \otimes (\chi_j \circ \det).$ 

Let $M$ be the smooth $A'_{{\mathfrak p}'}$-linear dual of the module:
$$(\abs \circ \det)^{-\frac{n-1}{2}}\Ind_Q^{\GL_n(E_v)} \bigotimes_i 
\bigl[(\St_{\pi_i,n_i} \otimes_{\CK_0} A'_{{\mathfrak p}'}) \otimes \chi_i\bigr],
$$
where $Q$ is a suitable parabolic subgroup of $\GL_n(E_v)$.  
Let $U_2$
be the open subset of $\Spec A'$ consisting of those ${\mathfrak p'}$ such that
$\rho_v \otimes_A \CK'$ is a minimal lift
of $\rho_v \otimes_A \kappa({\mathfrak p'})$.  Then for all ${\mathfrak p'} \in U_2$,
$M \otimes_{A'} \kappa({\mathfrak p}')$
is isomorphic to $\LLcheck(\rho_v \otimes_A \kappa({\mathfrak p}')).$  

Now let $M'$ be
the $A'[G]$-submodule of $M$ generated by $\J(M)$.  Then $(M')^{(n)}$
is isomorphic to $M^{(n)}$, and the latter is locally free of rank one over $A'$
by Corollary~\ref{cor:leibnitz}.  The module $M'$ is $A'$-torsion free, as it is contained
in the free $A'$-module $M$.  Moreover, for all ${\mathfrak p'} \in U_1 \cap U_2$,
we have isomorphisms:
$$M' \otimes_{A'} \kappa({\mathfrak p'}) \iso M \otimes_{A'} \kappa({\mathfrak p'})
\iso \LLcheck(\rho_v \otimes_A \kappa({\mathfrak p}')).$$  Set
$$M'' = M' \otimes_{A'} \bigl((M')^{(n)}\bigr)^{-1}.$$  Then $(M'')^{(n)}$ is free of rank one over $A'$.

Let $\Sigma'$ be the set of primes of $A'$ lying over primes in $\Sigma$.  Then
$\Sigma' \cap U_1 \cap U_2$ is dense in $\Spec A'$, and we have isomorphisms:
$V \otimes_A \kappa({\mathfrak p'}) \iso M'' \otimes_{A'} \kappa({\mathfrak p'})$
for all ${\mathfrak p'} \in \Sigma' \cap U_1 \cap U_2$.  It follows by Lemma~\ref{lem:dense}
that $V \otimes_A A'$ is isomorphic to $M''$.  In particular $V \otimes_A \CK'$
is isomorphic to $\LLcheck(\rho_v \otimes_A \CK')$, and hence
$V \otimes_A \CK$ is isomorphic to $\LLcheck(\rho_v \otimes_A \CK)$, as required.
\end{proof}

\medskip
\noindent
{\em Proof of Theorem~{\ref{thm:verify}}}.
By Lemma~\ref{lem:reduction} it suffices to consider the case where $S$ has a single element.
Let $V^{\tf}$ be the maximal $A$-torsion free quotient of $V$.  As $V$ is a finitely
generated $A[G]$-module, the kernel
$V^{\tor}$ of the map $V \rightarrow V^{\tf}$ is also finitely generated over $A[G]$.
In particular its support is a closed subset $Z$ of $\Spec A$.  Let $U$ be the complement of
$Z$ in $\Spec A$.  Then for all ${\mathfrak p}$ in $\Sigma \cap U$, we have isomorphisms:
$$
V^{\tf} \otimes_A \kappa({\mathfrak p}) \iso V \otimes_A \kappa({\mathfrak p}) \iso
\LLcheck(\rho_{v,{\mathfrak p}}).$$
It follows by Proposition~\ref{prop:tf verify} that $V^{\tf}$ satisfies conditions
(1), (2), and (3) of Theorem~\ref{thm:family}; in particular $\LLcheck(\rho_v)$
exists and is isomorphic to $V^{\tf}$.  Thus, by Theorem~\ref{thm:interpolation1},
$V^{\tf} \otimes_A \kappa({\mathfrak p})$ is isomorphic to
$\LLcheck(\rho_v \otimes_A \kappa({\mathfrak p}))$ for all ${\mathfrak p}$
for which there exists a minimal prime ${\mathfrak a}$ of $A$ such that $\rho_{v,\mathfrak a}$
is a minimal lift of $\rho_{v,\mathfrak p}$.  In particular this holds for all
${\mathfrak p}$ in $\Sigma$.  Thus we have an isomorphism:
$$V^{\tf} \otimes_A \kappa({\mathfrak p}) \iso V \otimes_A \kappa({\mathfrak p})$$
for all ${\mathfrak p} \in \Sigma$; by Lemma~\ref{lem:dense} it follows that
$V$ is isomorphic to $V^{\tf}$, and hence that $V$ satisfies conditions (1), (2), and (3)
of Theorem~\ref{thm:family}.
\qed

\medskip

\end{document}